\documentclass[11pt]{article}

\usepackage{amsfonts, latexsym, amssymb,amsmath,amsthm}
\usepackage[dvips]{graphicx,color}
\usepackage{caption}
\usepackage{relsize}
\usepackage[mathscr]{eucal}
\usepackage{stmaryrd}
\usepackage{esint}

\newtheorem{theorem}{Theorem}[section]
\newtheorem{theorems}{Theorem}[section]

\newtheorem{corollary}{Corollary}[section]
\newtheorem{conjecture}{Conjecture}[section]

\newtheorem{proposition}{Proposition}[section]
\newtheorem{lemma}{Lemma}[section]
\theoremstyle{definition}
\newtheorem{remark}{Remark}[section]

\numberwithin{equation}{section}

\allowdisplaybreaks[4]
\newcommand{\me}{\mathrm{e}}
\newcommand{\mi}{\mathrm{i}}
\newcommand{\md}{\mathrm{d}}

\begin{document}

\title{The Degenerate Third Painlev\'{e} Equation:
Complete Asymptotic Classification of Solutions
in the Neighbourhood of the Regular Singular Point\footnote{In the event
of any discrepancies between the published and preprint versions of the
paper, the reader is advised to refer to the preprint version.}}
\author{A.~V.~Kitaev\thanks{\texttt{E-mail: kitaev@pdmi.ras.ru}} \, and \,
A.~Vartanian\\
Steklov Mathematical Institute, Fontanka 27, St. Petersburg 191023, Russia}
\date{August 15, 2025}
\maketitle
\begin{abstract}
\noindent
We give a classification for the small-$\tau$ asymptotic behaviours of solutions to the degenerate third Painlev\'e
equation,
\begin{equation*}
u^{\prime \prime}(\tau) \! = \! \frac{(u^{\prime}(\tau))^{2}}{u(\tau)} \! - \! \frac{u^{\prime}(\tau)}{\tau}
\! + \! \frac{1}{\tau} \! \left(-8 \varepsilon (u(\tau))^{2} \! + \! 2ab \right) \! + \! \frac{b^{2}}{u(\tau)},\quad
\varepsilon=\pm1,\quad\varepsilon b>0,\quad
a\in\mathbb{C}\setminus\mi\mathbb{Z},
\end{equation*}
in terms of the monodromy data of a $2\times2$ matrix linear ODE whose isomonodromy deformations they describe.
We also study the complete asymptotic expansions of the solutions.
\vspace{0.30cm}


34M55, 34M56

\vspace{0.23cm}

\textbf{Abbreviated Title.} The Degenerate Third Painlev\'e Equation

\vspace{0.23cm}


transformations, generating function

\end{abstract}
\pagebreak
%\clearpage
\tableofcontents
%\listoffigures
\clearpage
\section{Introduction} \label{sec:Introduction}
We consider the degenerate third Painlev\'e equation in the following form,
\begin{equation}\label{eq:dp3}
u^{\prime \prime}(\tau) \! = \! \frac{(u^{\prime}(\tau))^{2}}{u(\tau)} \! - \! \frac{u^{\prime}(\tau)}{\tau}
\! + \! \frac{1}{\tau} \! \left(-8 \varepsilon (u(\tau))^{2} \! + \! 2ab \right) \! + \! \frac{b^{2}}{u(\tau)},\qquad
\varepsilon=\pm1,\quad\varepsilon b>0,\quad a\in\mathbb{C}.
\end{equation}
The parameters $\varepsilon$ and $b\in\mathbb{R}$ can be fixed, as particular real numbers, by a scaling
transformation of the dependent and independent variables, while the parameter $a$ coincides with the formal monodromy
of the associated Fuchs-Garnier pair \cite{KitVar2004}, and, thus, plays a more substantial role; we call it the
parameter of formal monodromy. Due to the classification given in \cite{OKSO2006}, equation~\eqref{eq:dp3} is
referred to as the $D_7$ case of the third Painlev\'e equation. In this paper, however, we do not make reference to
the space of initial values, and consider equation~\eqref{eq:dp3} from the point of view of isomonodromy deformation
theory. The latter point of view suggests calling equation~\eqref{eq:dp3} the $A_3$ case of the third Painlev\'e
equation (see \cite{KitVar2023}, Section 5).\footnote{This is in accordance with the classification of singularities
of normal forms of singular cubic surfaces in terms of Dynkin diagrams \cite{Sakamaki2010}. Such cubic surfaces
appear as the result of the application of the standard projectivization procedure~\cite{KitVar2023} to the monodromy
manifolds of Painlev\'e equations.}

In all of our works, we use the name ``degenerate third Painlev\'e equation'', because, according to the canonical
classification of the Painlev\'e equations given by Ince \cite{Ince}, it is a special
case of the third Painlev\'e equation that can be obtained from the complete third Painlev\'e equation by a
double-scaling limit \cite{KitVar2004}; this fact does not depend on the
methodology used to study the equation, and, at the same time, makes reference to the theory of special
functions.\footnote{\label{foot:D8} There is one more case of the third Painlev\'e equation that can be obtained
by double-scaling limits of the complete and degenerate third Painlev\'e equations; in the classification of
\cite{OKSO2006}, it is called the $D_8$ case of the third Painlev\'e equation. From our point of view, it can be
referred to as a ``doubly-degenerate'' third Painlev\'e equation. The latter equation is related to a special case
of the complete third Painlev\'e equation via a simple quadratic transformation; therefore, the analytic and
asymptotic properties of its solutions can be obtained from the corresponding properties of the complete third
Painlev\'e equation.}

%The fact that the general solution of equation~\eqref{eq:dp3} is recognized as a special function that is used by
%a broad array of specialists in mathematical physics implies the existence of a detailed description of the
%properties of this function.

The fact that the general solution to equation~\eqref{eq:dp3} is recognized as a special function used by
a wide range of specialists in mathematical physics~\cite{buckmil2024,GIKMO} suggests the need for a detailed
description of the properties of this function.\footnote{\label{foot:refs}
Here, we cite only the two recent papers~\cite{buckmil2024,GIKMO}; many more references can be found in
\cite{KitVar2004,KitVarZapiski2024}.}

We decided to begin this description by considering one of the most rudimentary
analytical questions, namely, the asymptotic description of the degenerate third Painlev\'e function in the
neighbourhood of the regular singular point of equation~\eqref{eq:dp3}. Our attention to this question was drawn by
B. I. Suleimanov, who realized that, although the small-$\tau$ asymptotics for the general solution of
equation~\eqref{eq:dp3} is obtained in \cite{KitVar2004}, one cannot extract from it the asymptotics of the solution
that appeared in his work~\cite{Suleimanov2017}. Actually, from the formal point of view, it is not possible to obtain
the answer to his question by simply referring to Theorem 3.4 of \cite{KitVar2004}, because, for the Suleimanov
solution, the leading term of our asymptotic formula vanishes, and the estimate for the correction term does not
allow one to calculate the full set of the monodromy data corresponding to this solution from the remaining terms
of our asymptotic formula. To thoroughly study this case, we wrote the two papers~\cite{KitSIGMA2019} and
\cite{KitVar2023}; although the Suleimanov case is now resolved, there are other solutions for which the absence of
a proper estimate for the correction term creates a similar problem.

There is another problem with our small-$\tau$ asymptotic formula: the set of the general solutions considered in
Theorem 3.4 of \cite{KitVar2004} does not include all solutions of equation~\eqref{eq:dp3}. This fact is
straightforward to observe, because the real part of the parameter $\rho$ describing the branching of solutions
$u(\tau)$ at $\tau=0$ obeys the restriction $|\mathrm{Re}\,\rho|<1/2$; therefore, for those solutions with
$|\mathrm{Re}\,\rho|=1/2$, the asymptotic formulae are absent, even though the corresponding set of the monodromy
data depends on three real parameters.
Furthermore, for $\mathrm{Re}\,\rho\to\pm1/2$, in order to achieve a good approximation via the asymptotics obtained
in Theorem 3.4 of \cite{KitVar2004}, one has to consider this approximation in an ever-shrinking neighbourhood of
the origin, which, ultimately, is numerically unattainable. The standard paradigm for dealing with this problem
would be to invoke the correction terms, which were not considered in Theorem 3.4 of \cite{KitVar2004}; in fact,
an infinite number of such correction terms would be needed if $\rho$ is not bounded away from $\pm1/2$.

In Theorem 3.4 of \cite{KitVar2004}, there are additional restrictions on the monodromy data, namely,
$|\mathrm{Im}\,a|<1$ and $g_{11}g_{22}\neq0$ (see Section~\ref{sec:Baecklund}). The first of these restrictions
does not appear to be ``crucial'', because it is clear that the asymptotics can be extended via B\"acklund
transformations shifting the parameter $a\to a\pm\mi$. The following questions, however, remain to be answered:
(i) should the application of the B\"acklund transformations be left as an iterative
procedure;\footnote{\label{foot:itterativeBaecklund} If $a$ satisfies the condition $n<\mathrm{Im}\,a<n+1$
for $n\in\mathbb{Z}$, then $|n|$ B\"acklund transformations are required in order to find a desired parametrization.}
(ii) can the parametrization of the asymptotics via the monodromy data be presented in closed form; and
(iii) how can asymptotics be constructed for the cases $\mathrm{Im}\,a=n$, $n\in\mathbb{N}$, or $g_{11}g_{22}=0$?

As mentioned in \cite{KitVar2023}, there are solutions of equation~\eqref{eq:dp3} which depend on a parameter that
is ``concealed'' in the $n$th term of the asymptotic expansion, so that the leading term of asymptotics
does not allow for the unique specification of such a solution.

In this paper, we address all of the questions outlined above. Before we started working on this paper,
we amassed considerable experience by applying the asymptotics obtained in Theorem 3.4 of \cite{KitVar2004} to our
study \cite{KitVar2023} of algebroid solutions of equation~\eqref{eq:dp3} for $a=0$, and, as a result, presented
in Appendix B of \cite{KitVar2023} a more convenient version of the asymptotic formulae (equivalent to those in
\cite{KitVar2004}), together with the asymptotics of the auxiliary \emph{mole function} $\varphi(\tau)$
(see Section~\ref{sec:Baecklund} for its definition).
In preparation for this work, we wrote the paper \cite{KitVarZapiski2024}, where we: (i) removed the
restriction $g_{11}g_{22}\neq0$; (ii) simplified the notation (without corrections of the results) of Theorem 3.5
of \cite{KitVar2004} for solutions $u(\tau)$ possessing logarithmic behaviour as $\tau\to0$; (iii) included
asymptotics for the function $\varphi(\tau)$ in the logarithmic case; and (iv) presented a numerical visualization of
their asymptotics for $a=0$.

After the preliminaries delineated above, we are ready for the classification of the small-$\tau$ asymptotic
behaviour of the degenerate third Painlev\'e transcendent. This classification is based on three analytic
ingredients: our parametrization for the leading term of the small-$\tau$ asymptotics of the general solution,
$u(\tau)$, and the function $\varphi(\tau)$ in terms of the monodromy data \cite{KitVar2023,KitVarZapiski2024};
B\"acklund transformations; and complete asymptotic expansions for $u(\tau)$ at the origin.

Why do we claim that our classification of solutions of equation~\eqref{eq:dp3} via their small-$\tau$
asymptotics is complete? The answer to this question is based on the isomonodromy deformation method: for every
point of the monodromy manifold (see Section~\ref{sec:Baecklund}), we proved the existence of the solution, and
obtained the corresponding asymptotics as $\tau\to0$. The proof is based on the justification scheme for the
isomonodromy deformation method suggested in \cite{Kit1989} and our results for the small-$\tau$ asymptotics
obtained in \cite{KitVar2004,KitVarZapiski2024}.

The $\pmb{\tau}$-functions for the Painlev\'e equations, since their appearance in the Jimbo-Miwa paper
\cite{JM-2-1981}, have proved to be very important objects in applications related to integrable models in
quantum field theory and random matrix theory. The reader may, therefore, pose the following---natural---question:
why is the $\pmb{\tau}$-function not considered in this paper? Our answer to this question is simple: a paper based
on the isomonodromy approach to the $\pmb{\tau}$-function will appear in the not-too-distant future! It is in this
latter paper, and not the present one, that we study the connection problem for the $\pmb{\tau}$-function of the
degenerate third Painlev\'e equation; otherwise, since a simple classification of its small-$\tau$ asymptotic
behaviour can be gleaned straightforwardly from the results obtained in this paper and will not, therefore,
supplement additional knowledge about this function, its inclusion in the present work would only lead to inflate
the pagination count.

Another topic that isn't included in this work, although it is closely related to the study of small-$\tau$
asymptotics of solutions of equation~\eqref{eq:dp3}, is the description of the properties of algebroid solutions.
The construction of asymptotics for algebroid solutions does not present any difficulties, since such solutions
correspond to rational values of the branching parameter $\rho$, and asymptotic formulae for the general solution
of equation~\eqref{eq:dp3} are applicable to them without any additional restrictions. At the same time, our study
of algebroid solutions of equation~\eqref{eq:dp3} for $a=0$ in \cite{KitVar2023} shows that these solutions possess
interesting properties that are worthy of further investigation.

Here, in addition to the degenerate third Painlev\'e transcendent, we also give the corresponding results for its
associated mole function, $\varphi(\tau)$, which was introduced by us in \cite{KitVar2023}. Analogous functions are
not a novelty in the theory of Painlev\'e equations; in fact, without endowing them with any special name(s),
such functions, which, in our notation, are equivalent to $\me^{\mi \varphi (\tau)}$, were introduced in
\cite{JM-2-1981} for all the Painlev\'e equations, with the exceptions of the first and the degenerate third
Painlev\'e equations.
Such functions play an important role in the isomonodromy deformations of $2\times2$ matrix
linear ODEs, since it is these functions, together with the corresponding Painlev\'e functions, which define
the isomonodromy deformations of these linear ODEs. Unlike the function $u(\tau)$, the function
$\me^{\mi\varphi(\tau)}$ depends on an additional non-vanishing multiplicative parameter, which, in this work, we
express in terms of the monodromy data; this fact allows us to calculate connection formulae for asymptotics of some
interesting integrals related to $u(\tau)$.\footnote{\label{foot:integrals} See, for example, \cite{KitVar2019}.}
The mole function does not possess the Painlev\'e property, so that its analytic continuation depends on the path of
continuation; this ambiguity, however, is defined by a period of the exponential function, that is, $2\pi\mi k$,
$k\in\mathbb{Z}$. Thus, one can also write connection formulae for asymptotics of $\varphi(\tau)\!\!\!\mod\!(2\pi)$.
The function $\varphi(\tau)$ can be defined as an integral of the function $u(\tau)$ (see Section~\ref{sec:Baecklund},
equation~\eqref{eq:varphi}), and the calculation of the
parameter $k$ as a function of the monodromy data of the solution $u(\tau)$ for some special paths
of integration, e.g., along the real axis, may represent an interesting technical problem. The name ``mole function''
for $\varphi(\tau)$ appeared as our emotional reaction to observing how much the graph of the function
$\varphi(\tau)$, weaving up, and then down, the real axis, resembled the trajectory of a mole’s movements (see
Section 6 of \cite{KitVar2023} for the corresponding details).

To finalize the general part of the Introduction, it is worth mentioning that, to the best of our knowledge
at the present time, the classification of small-$\tau$ asymptotics in terms of the monodromy data of
associated $2\times2$ matrix linear ODEs has not yet been completed for all of the Painlev\'e equations which have
a regular singular point, that is, the sixth, all versions of the fifth, and the third Painlev\'e equations. At the
same time, though, all the ingredients that we use in this paper are well known for the aforementioned Painlev\'e
equations, so that, with their help, such a classification for these Painlev\'e equations can be completed without
having to create any additional technical tools.

In Section~\ref{sec:Baecklund}, we summarize all the facts that are necessary in order to understand the
results presented in this paper; in particular, the definitions of the function $\varphi(\tau)$, the monodromy
manifold, and the B\"acklund transformations. Section~\ref{sec:general} concerns the asymptotics of the general
solutions which are valid in the neighbourhood of $\rho=1/2$.
In Sections~\ref{sec:special cases}--\ref{sec:Meromorphic}, we present asymptotic descriptions of the solutions
whose asymptotics are not described by the formulae for the general solutions. In
Appendices~\ref{app:sec:full0expansion}, \ref{app:sec:full0expansionLog}, and \ref{app:sec:full0expansionLog2},
we study various features of the complete asymptotic expansions of the solutions; in particular, we develop the
technique of generating functions for these expansions.\footnote{\label{foot:gengenfunctions} Such generating
functions, with minor modifications, can be constructed for all the Painlev\'e equations.}
These generating functions not only allow one to explicitly calculate the coefficients of the expansions,
but are also very helpful for the study of the special solutions in
Sections~\ref{sec:special cases}--\ref{sec:Meromorphic}. Distinguished amongst the results obtained in these
appendices we mention the new type of small-$\tau$ asymptotic formula for $u(\tau)$ obtained in
Appendix~\ref{app:subsec:asympt-universal} that is uniform with respect to the branching parameter $\rho$;
we also explicitly obtain the first correction term of this asymptotic formula, and our calculations show
that one can develop this uniform asymptotics into a complete asymptotic expansion.

During the course of the implementation of our original plan to describe all solutions of equation~\eqref{eq:dp3}
by relating their small-$\tau$ asymptotics to the monodromy data, we exceeded both the time and scope of our
intended presentation. Therefore, some of the technical issues which we had planned on addressing in this paper have
been moved to a follow-up work in which: (i) a complete classification of the small-$\tau$ asymptotics
of solutions for $a\in\mi\mathbb{Z}$ is given; (ii) additional results regarding the description of the poles and
zeros considered in Section~\ref{sec:poles} and in Section 4 of \cite{KitVarZapiski2024},  including the asymptotics
of the corresponding expansion parameters, are obtained; and (iii) a numerical visualization of the asymptotics
derived in this work is presented.

As the paper is relatively long and contains a variety of results and ideas, we provide guidance on how interested
readers can use the results of this work. Our presumption is that there are two categories of readers of this paper:
(i) those who will use our results to solve specific problems of mathematical physics; and
(ii) specialists working in the field of Painlev\'{e} equations and asymptotics.

How can the results of this paper be used by those who have obtained, whilst studying a specific mathematical model,
a particular solution (or a family of solutions) of equation~\eqref{eq:dp3}? Such solutions are distinguished by
properties that are inherited from the mathematical model being considered. The following properties of the solutions
will be helpful in applying the results of this paper:
\begin{enumerate}
\item
for a solution holomorphic at $\tau=0$, see Section~\ref{sec:Meromorphic};
\item
for a solution having an infinite sequence of poles accumulating at $\tau=0$, see Section~\ref{sec:poles};
\item
for a solution having an infinite sequence of zeros accumulating at $\tau=0$, refer to Section 4 of
\cite{KitVarZapiski2024};
\item
the small-$\tau$ asymptotics of the solution of interest are obtained by the reader in terms of the parameters of
the mathematical model being considered; see Sections~\ref{sec:general}--\ref{sec:logarithm} or
Appendix~\ref{app:subsec:error-correction-terms}: in these sections find asymptotics for $u(\tau)$ with the same
$\tau$-dependence as obtained, and equate the parameters of the model under investigation with the monodromy data
provided in this paper.

The $\tau$-dependence of the asymptotics in Section~\ref{sec:general} and the asymptotics presented in
Section~\ref{sec:special cases} (item {\bf(1)} of Theorems~\ref{th:Asympt0-rho-eq-1pn+ia2}
and \ref{th:Asympt0-rho-eq-1pn-ia2}) coincide. In Remark~\ref{rem:uas-pq1q2} we explain how the reader can make the
distinction between these cases.
By following the above steps, one can obtain expressions for the original parameters via the quadratic products of
the monodromy data and use these formulae to find large-$\tau$ asymptotics of the solution by employing the results
presented in Appendix C of \cite{KitVar2023}.

The asymptotics of the mole function for the solution presented in Appendix~\ref{app:subsec:error-correction-terms}
is given in Appendix B of \cite{KitVar2023}.
If the mole function does not appear in the model under investigation, it does not create any difficulties in finding
the monodromy parametrization of the function $u(\tau)$, because it depends on an additional monodromy parameter
which does not affect the monodromy parametrization of $u(\tau)$.
\end{enumerate}

The starting point for reading this paper by experts in Painlev\'e equations and asymptotics will likely be the
monodromy data defined in Section 2.
In this paper we present the monodromy data as co-ordinates of an 8-component vector whose first co-ordinate is $a$.
The three successive co-ordinates, $s_{0}^{0}$, $s_{0}^{\infty}$, and $s_{1}^{\infty}$, called the Stokes multipliers,
play a crucial role in our classification of the small-$\tau$ asymptotic behavior of solutions to
equation~\eqref{eq:dp3}:
\begin{enumerate}
\item
if $s_0^0 = \pm 2 \mi$, then the asymptotics contain logarithmic terms (see Section 5);
\item
if $s_0^{\infty}s_1^{\infty}=0$,\footnote{In this case, note that, as follows from equation~\eqref{eq:monodromy:s}
(see Section~\ref{sec:Baecklund} below),
$s_0^0\neq\pm 2\mi$, since it is assumed in the paper that $\mi a\notin\mathbb{Z}$.} then there exist special
solutions, i.e., solutions depending on one complex parameter, with power-like asymptotic behaviour
(see Section~\ref{sec:special cases}).
The case $s_0^{\infty}=s_1^{\infty}=0$, which is a particular sub-case of the solutions studied in
Section~\ref{sec:special cases}, is considered separately in Section~\ref{sec:Meromorphic}; and
\item
the general case $s_0^{\infty}s_1^{\infty}(s_0^0\pm2\mi)\neq0$ is considered in Section~\ref{sec:general} and
Appendix~\ref{app:subsec:error-correction-terms}. In Appendix~\ref{app:subsec:error-correction-terms}, we present
a generic asymptotic expansion for general solutions of the Painlev\'e equations having a regular singular point
at $\tau=0$.
In Section 3, we propose another formula for the leading term of asymptotics for the general solutions, and in
Appendix~\ref{app:subsec:super-generating-function}, we demonstrate how to construct the corresponding complete
asymptotic expansion. Both asymptotic expansions have overlaping domains of applicability; however, the generic
expansion (Appendix~\ref{app:subsec:error-correction-terms}) ``works better''
{}\footnote{\label{foot:works} An asymptotic formula ``works better'' means that, for the same values of $\tau$,
it better approximates the corresponding solution $u(\tau)$.} when the real part of the branching parameter
$\sigma=4\rho$ is close to $0$, while the asymptotics of Section~\ref{sec:general} ``works better'' when
$\mathrm{Re}\,\sigma\approx2$. For $\mathrm{Re}\,\sigma=2$, the generic asymptotic formula is not applicable,
whilst the asymptotics of Section~\ref{sec:general} is still valid (the situation is reversed for
$\mathrm{Re}\,\sigma=0$).
In Appendix~\ref{app:subsec:asympt-universal}, we propose a unique formula for the
small-$\tau$ asymptotics of $u(\tau)$ that is valid for all admissible values of $\mathrm{Re}\,\sigma\in[-2,2]$.
\end{enumerate}
%%%%%%%%%%%%%%%%%%%%%%%%%%%%%%%%%%%%%%%%%%%%%%%%%%%%%%%%%%%%%%%%%%%%%%%%%%%%%%%%%%%%%%%%%%%%%%%%%%%%%%%%%%%%%%
%%%%%%%%%%%%%%%%%%%%%%%%%%%%%%%%%%%%%%%%%%%%%%%%%%%%%%%%%%%%%%%%%%%%%%%%%%%%%%%%%%%%%%%%%%%%%%%%%%%%%%%%%%%%%
\section{The Monodromy Manifold and B\"acklund Transformations}\label{sec:Baecklund}
In \cite{KitVar2004}, we introduced a $2\times2$ matrix linear ODE whose isomonodromy deformations are governed by
the pair of functions $(u(\tau),\me^{\mi\varphi(\tau)})$; furthermore, it was shown that $\varphi(\tau)$ solves the
ODE
\begin{equation}\label{eq:varphi}
\varphi{'}(\tau)=\frac{2a}{\tau}+\frac{b}{u(\tau)},
\end{equation}
where $u(\tau)$ is a solution of equation~\eqref{eq:dp3}.\footnote{\label{foot:comment-def-varphi} The function
$\varphi(\tau)$ is an important ingredient of the theory of the degenerate third Painlev\'e transcendent: the
significance of its role is discussed in an upcoming paper. In \cite{KitVar2023}, the asymptotic properties of
$\varphi(\tau)$ for a particular algebroid solution of equation \eqref{eq:dp3} were analysed; in fact, in
\cite{KitVar2023}, we coined the name \emph{mole function} for $\varphi(\tau)$.}
The pair of functions $(u(\tau), \me^{\mi\varphi(\tau)})$ can be uniquely parametrized via the co-ordinates of
the points of the monodromy manifold, so that the mole function, $\varphi(\tau)$,\footref{foot:comment-def-varphi}
is defined up to $2\pi n$, $n\in\mathbb{Z}$, rather than up to an arbitrary constant of integration, as follows from
equation~\eqref{eq:varphi}.

In \cite{KitVar2004}, we defined a monodromy manifold that can be presented in terms of the monodromy data associated
with a $2\times2$ matrix linear ODE. Consider $\mathbb{C}^{8}$ with
co-ordinates $(a,s_{0}^{0}, s_{0}^{\infty}, s_{1}^{\infty}, g_{11},\linebreak[1]g_{12},g_{21},g_{22})$,
where the parameter of formal monodromy, $a$, the Stokes multipliers, $s_{0}^{0}$, $s_{0}^{\infty}$, and $s_{1}^{\infty}$,
and the elements of the connection matrix, $(G)_{ij} \! =: \! g_{ij}$, $i,j \! \in \! \lbrace 1,2 \rbrace$, are
called the \emph{monodromy data}. These monodromy data are related by the set of algebraic equations
{}\footnote{In terms of the parameter $\me^{\pi a}$, these equations are algebraic.}
\begin{gather}
s_{0}^{\infty}s_{1}^{\infty} \! = \! -1 \! - \! \me^{-2 \pi a} \! - \! \mi s_{0}^{0}
\me^{-\pi a}, \label{eq:monodromy:s} \\
g_{21}g_{22} \! - \! g_{11}g_{12} \! + \! s_{0}^{0}g_{11}g_{22} \! = \! \mi \me^{-\pi a},
\label{eq:monodromy:main} \\
g_{11}^{2} \! - \! g_{21}^{2} \! - \! s_{0}^{0} g_{11} g_{21} \! = \! \mi s_{0}^{\infty}
\me^{-\pi a}, \label{eq:monodromy:s0} \\
g_{22}^{2} \! - \! g_{12}^{2} \! + \! s_{0}^{0} g_{12} g_{22} \! = \! \mi s_{1}^{\infty}
\me^{\pi a}, \label{eq:monodromy:s1} \\
g_{11}g_{22} \! - \! g_{12} g_{21} \! = \! 1. \label{eq:monodromy:detG}
\end{gather}
The system~\eqref{eq:monodromy:s}--\eqref{eq:monodromy:detG} defines an algebraic variety, which we call
the \emph{manifold of the monodromy data}, $\mathscr{M}$. The manifold
$\mathscr{M}$ provides a two-fold parametrization of the set of solutions of the system~\eqref{eq:dp3},
\eqref{eq:varphi},
namely, each solution $(u(\tau),\me^{\mi\varphi(\tau)})$ corresponds to two, and only two, points
$(a,s_0^0,s_0^{\infty},s_1^{\infty},g_{11},g_{12},g_{21},g_{22})$ and
$(a,s_0^0,s_0^{\infty},s_1^{\infty},-g_{11},-g_{12},-g_{21},-g_{22})$ on $\mathscr{M}$, and vice versa.
For the unique parametrization of solutions
of equation~\eqref{eq:dp3} in terms of the monodromy data, one has to identify (glue) points of $\mathscr{M}$ that
correspond to matrices $G$ and $-G$; by doing so, one arrives at the so-called \emph{contracted monodromy manifold}
(see \cite{KitVar2023,KitVarZapiski2023} for details). At the same time, no difficulties are encountered while
addressing the study of the functions $u(\tau)$ and $\me^{\mi\varphi(\tau)}$ with the help of $\mathscr{M}$ in
conjunction with the gluing condition $G\in\mathrm{PSL}(2,\mathbb{C})$. The latter condition is not, in fact,
consequential in those cases where one can fix any representative of $G$ in $\mathrm{PSL}(2,\mathbb{C})$;
for example, in obtaining the connection formulae for asymptotics of solutions. The advantage of the latter
representation for the monodromy manifold is that it simplifies the analysis of special solutions by making it easier
to observe the relationship between the corresponding pair of functions $(u(\tau),\me^{\mi\varphi(\tau)})$ and the
auxiliary $2\times2$ matrix linear ODE.

In the sections that follow, we address the application of B\"acklund transformations for equation~\eqref{eq:dp3}
to the construction of the small-$\tau$ asymptotics of the functions $u(\tau)$ and $\me^{\mi\varphi(\tau)}$.
The B\"acklund transformations for the degenerate third Painlev\'e equation (the formulae equivalent to
equations~\eqref{eq:u+} and \eqref{eq:u-} below) were obtained by Gromak~\cite{Gromak1973}. For our studies, not only do we need these formulae, but we also require the action of
the B\"acklund transformations on the manifold of the monodromy data. This action was obtained in \cite{KitVar2004} by
virtue of the realisation of the B\"acklund transformations via the Schlesinger transformations of an associated
$2\times2$ matrix linear ODE; more precisely, for any solution $(u(\tau),\varphi(\tau))$
of the system~\eqref{eq:dp3}, \eqref{eq:varphi} corresponding to the monodromy data
$(a,s_{0}^{0},s_{0}^{\infty},s_{1}^{\infty},g_{11},g_{12},g_{21},g_{22})$,
the functions (see \cite{KitVar2004}, p. 1198)
\begin{align}
u_+(\tau)&=-\frac{\mi\varepsilon b}{8(u(\tau))^2}\left(\tau(u^\prime(\tau)+\mi b)+(2a\mi-1)u(\tau)\right),
\label{eq:u+}\\
\varphi_+(\tau)&=\varphi(\tau)-\mi\ln\left(-\frac{u(\tau)u_+(\tau)}{\varepsilon b\tau^2}\right),\label{eq:varphi+}
\end{align}
and
\begin{align}
u_-(\tau)&=\frac{\mi\varepsilon b}{8(u(\tau))^2}\left(\tau(u^\prime(\tau)-\mi b)-(2a\mi+1)u(\tau)\right),
\label{eq:u-}\\
\varphi_-(\tau)&=\varphi(\tau)+\mi\ln\left(-\frac{u(\tau)u_-(\tau)}{\varepsilon b\tau^2}\right),\label{eq:varphi-}
\end{align}
solve the system~\eqref{eq:dp3}, \eqref{eq:varphi} for $a=a_{+}:= a+\mi$ and $a=a_{-}:= a-\mi$, respectively.
The monodromy data corresponding to these functions are
\begin{align}
(a,s_{0}^{0},s_{0}^{\infty},s_{1}^{\infty},g_{11},g_{12},g_{21},g_{22})_+=
(a+\mi,-s_{0}^{0},s_{0}^{\infty},s_{1}^{\infty},-\mi g_{11},-\mi g_{12},\mi g_{21},\mi g_{22}),\label{eq:data+}\\
(a,s_{0}^{0},s_{0}^{\infty},s_{1}^{\infty},g_{11},g_{12},g_{21},g_{22})_-=
(a-\mi,-s_{0}^{0},s_{0}^{\infty},s_{1}^{\infty},\mi g_{11},\mi g_{12},-\mi g_{21},-\mi g_{22}).\label{eq:data-}
\end{align}
The transformations \eqref{eq:data+} and \eqref{eq:data-} are mutually inverse, that is,
$(u_+)_-(\tau)=(u_-)_+(\tau)=u(\tau)$ and $(\varphi_+)_-(\tau)=(\varphi_-)_+(\tau)=\varphi(\tau)$. These facts
can be established either by a direct calculation with the help of equation~\eqref{eq:dp3}, or without any
calculations by simply referring to the transformations~\eqref{eq:data+} and \eqref{eq:data-} for the monodromy data.
Note that the direct and inverse B\"acklund transformations differ by a formal conjugation, i.e., the change
$\mi\to-\mi$ (cf. equations~\eqref{eq:u+}, \eqref{eq:varphi+}, and \eqref{eq:data+} with
equations~\eqref{eq:u-}, \eqref{eq:varphi-}, and \eqref{eq:data-}, respectively).

The function $f(\tau):=u_+(\tau)u(\tau)$ solves a second-order ODE that is quadratic with respect to the second
derivative, possesses the Painlev\'e property, and is equivalent to equation~\eqref{eq:dp3}. This ODE was discovered
by Bureau \cite{B1972} via a Painlev\'e-type analysis, and was subsequently rediscovered by Cosgrove and Scoufis
\cite{CosgroveScoufis1993} in the course of their studies on the classification of second-order ODEs that are
quadratic with respect to the second derivative and appeared in their classification scheme as SD-III.A.
Later, in \cite{KitVar2004}, an ODE equivalent to SD-III.A was independently derived via a Hamiltonian reformulation
of equation~\eqref{eq:dp3}.\footnote{\label{foot:SDIIIa} In \cite{KitVar2019}, we studied integrals taken along
the segment $[0,\tau]\in\mathbb{R}_+$ for the functions $\varphi(\tau)$ and $f(\tau)/\tau$ corresponding to
a meromorphic solution of equation~\eqref{eq:dp3} vanishing at the origin.}

Now, fixing some $a_0\in\mathbb{C}$ and denoting by $u_0(\tau)$ any solution of equation~\eqref{eq:dp3} with $a=a_0$,
define, for $n\in\mathbb{Z}$, the solution $u_n(\tau)$ as the $n$th iteration of the solution $u_0(\tau)$ under the
transformations \eqref{eq:u+} and \eqref{eq:u-}. By definition, $u_n(\tau)$ solves equation~\eqref{eq:dp3} for
$a=a_n:=a_0+\mi n$.

One can derive 2-node differential-difference  and 3-node difference and differential-difference
relations that are satisfied by the sequence $u_n=u_n(\tau)$ or related functions.

The B\"acklund transformations themselves can be considered as 2-node differential-difference relations. To see this,
one substitutes $u=u_n$, $a=a_n$, and $u_+=u_{n+1}$ into equation~\eqref{eq:u+} and $u=u_{n+1}$, $a=a_{n+1}$,
and $u_-=u_n$ into equation~\eqref{eq:u-}. The third relation can be obtained as a compatibility condition of the
two differential-difference equations obtained via a renotation argument as described in the previous sentence,
namely, multiply the first and second equations by $u_n$ and $u_{n+1}$, respectively, and then equate the resulting
right-hand sides to find that
\begin{equation*}\label{eq:2-node}
\frac{u_n^{\prime}(\tau)+\mi b}{u_n(\tau)}+\frac{u_{n+1}^{\prime}(\tau)-\mi b}{u_{n+1}(\tau)}=0\quad
\Rightarrow\quad
(u_n(\tau)u_{n+1}(\tau))^\prime+\mi b(u_{n+1}(\tau)-u_n(\tau))=0.
\end{equation*}

To write 3-node relations, it is convenient to introduce the function $v_n(\tau):=u_n(\tau)/\tau$; then,
equations~\eqref{eq:u+} and \eqref{eq:u-} imply that
\begin{align}
v_n^2\big(v_{n+1}-v_{n-1}\big)&=-\frac{\varepsilon b}{4\tau}\frac{\md v_n}{\md\tau},\label{eq:dif-discrete-v-n}\\
v_n^2\big(v_{n+1}+v_{n-1}\big)&=\frac{\varepsilon b}{4\tau^2}(b+2a_nv_n),\label{eq:discrete-v-n}
\end{align}
where $v_n=v_n(\tau)$.\footnote{\label{foot:comment:2004} In the corresponding equations on p. 1198 of \cite{KitVar2004},
slightly different definitions are used, namely, $n\to-n$ and $v_n\to v_{-n}$; moreover, the differential-difference equation
for $v_n(\tau)$ contains a misprint: its right-hand side should be divided by $v_n\tau$.}
The differential-difference equation \eqref{eq:dif-discrete-v-n} is related to the
Volterra chain \cite{Volterra1931} with free ends by the following transformation:
\begin{gather}
w_n(x):=v_n(\tau)v_{n+1}(\tau),\qquad x:=-2\tau^2/\varepsilon b,\label{eq:w-x-change-Volterra}\\
\frac{\md w_n}{\md x}= w_n(w_{n+1}-w_{n-1}).\label{eq:Volterra-chain}
\end{gather}
Equation~\eqref{eq:discrete-v-n} is equivalent to one of the so-called discrete Painlev\'e equations.
\begin{remark}\label{rem:Lambert}
Introducing the function $\alpha_n(x)=\sqrt{w_n(x)}$, one finds that $\alpha_n(x)$ defines a solution of the
Kac-van Moerbeke system of differential-difference equations \cite{Kac-vanMoerbeke1975};
solutions of the last system can be mapped, via a discrete Miura-type transformation, to solutions of the system of
Toda lattice equations.

For any sequence of functions $F_n(x)$, $n\in\mathbb{Z}$, define the difference operator
$\Delta$: $\Delta F_n(x)=F_{n+1}(x)-F_{n-1}(x)$; then, introducing the function $g_n(x)=w_{n}(x)w_{n+1}(x)$, where $w_n(x)$
solves equation~\eqref{eq:Volterra-chain}, one shows that
\begin{equation}\label{eq:log-differential-difference}
\frac{\md^2}{\md x^2}\ln g_n(x)=\Delta^2 g_n(x)= g_{n+2}(x)-2g_n(x)+g_{n-2}(x).
\end{equation}
\hfill$\blacksquare$\end{remark}
%%%%%%%%%%%%%%%%%%%%%%%%%%%%%%%%%%%%%%%%%%%%%%%%%%%%%%%%%%%%%%%%%%%%%%%%%%%%%%%%%%%%%%%%%%%%%%%%%%%%%%%%%%%%%
%%%%%%%%%%%%%%%%%%%%%%%%%%%%%%%%%%%%%%%%%%%%%%%%%%%%%%%%%%%%%%%%%%%%%%%%%%%%%%%%%%%%%%%%%%%%%%%%%%%%%%%%%%%%%
\section{Small-$\tau$ Asymptotics: The Generic Case}\label{sec:general}
Equation~\eqref{eq:dp3} has a regular singular point at $\tau=0$, so that the bulk of its solutions $u(\tau)$
are not single-valued. For the characterization of the asymptotic behaviour of such solutions at $\tau=0$, we
introduced the branching parameter $\rho$ in \cite{KitVar2004}, and proved that it satisfies the following equation
\begin{equation} \label{eq:rho-general}
\cos(2\pi\rho)=-\frac{\mi s_{0}^{0}}{2}=\cosh(\pi a)+\frac{1}{2}s_{0}^{\infty}s_{1}^{\infty}\me^{\pi a},
\end{equation}
where the right-hand side is defined in terms of the monodromy data associated with $u(\tau)$, and the inequality
\begin{equation}\label{eq:rho-restriction}
|\mathrm{Re}\,\rho|<1/2.
\end{equation}
The conditions \eqref{eq:rho-general} and \eqref{eq:rho-restriction} define the parameter $\rho$ up to a sign.
The asymptotic formulae for the general solution $u(\tau)$ and the corresponding mole function $\varphi(\tau)$
(cf. \cite{KitVar2004,KitVar2023}) are invariant with respect to the reflection $\rho\to-\rho$, so that the choice
of the sign of $\rho$ is a matter of convenience. The asymptotic formulae for some special/particular solutions
might be written with a predetermined choice for the sign of $\rho$ having been made, and if so, it is stated
as such.

For the formulation of the results obtained in this paper, we find it convenient to introduce another branching
parameter, $\varrho$, which solves equation~\eqref{eq:rho-general} with $\rho\to\varrho$, and whose real
part is fixed as follows:
\begin{equation}\label{eq:varrho-restriction}
\mathrm{Re}\,\varrho\in(0,1).
\end{equation}
Equation~\eqref{eq:rho-general} and the restriction~\eqref{eq:varrho-restriction} fix the parameter
$\varrho$ modulo the reflection $\varrho\to1-\varrho$, so that our asymptotic formulae for general solutions
should be invariant under this symmetry. For the asymptotic description of some particular solutions, we can impose
a special condition on $\varrho$ which may not be compatible with the symmetry.
\begin{remark}\label{rem:varrho-restriction}
The restriction~\eqref{eq:varrho-restriction} means that asymptotics of solutions with monodromy data
belonging to the ray $\mathrm{Im}(s_0^0)\in[2,+\infty)$ and $\mathrm{Re}(s_0^0)=0$ cannot be described in
terms of the parameter $\varrho$; furthermore, the asymptotic formulae in terms of $\varrho$, although formally
correct in a small neighbourhood of this ray, do not, in practice, allow one to achieve satisfactory approximation
of the corresponding solutions, and they can only be used with a sufficiently large number (depending on the
smallness of the neighbourhood) of correction terms. In this case, the reader should apply the asymptotic formulae
written in terms of the parameter $\rho\neq0$ (see the text subsequent to equation~\eqref{app:eq:Asympt0:u2004} in
Appendix~\ref{app:subsec:error-correction-terms}), or the uniform asymptotic formula~\eqref{app:eq:u:Asympt-uniform}
with $\sigma=4\rho\neq0$. The asymptotics for $\rho=0$ ($s_0^0=2\mi$) is constructed in
Theorem~\ref{th:ln-regular-a-non0} of Section~\ref{sec:logarithm}.
There is one more special meromorphic solution of equation~\eqref{eq:dp3} that corresponds to the case
$s_0^0=2\mi\cosh(\pi a)$ which is studied in Theorem~\ref{eq:mondata:s0s1EQ0} of Section~\ref{sec:Meromorphic}.
Our results obtained in \cite{KitVar2004,KitVar2023} were formulated in terms of $\rho$, with the restriction
$|\mathrm{Re}\,\rho|<1/2$, so that the asymptotic description of the solutions corresponding to the monodromy data
for $\mathrm{Re}\,\varrho=1/2$ was excluded. Furthermore, when compared with our previous results, the asymptotic formulae presented below provide a
much better approximation for solutions in a neighbourhood of the points with $\mathrm{Re}\,\rho=1/2$.
The results presented in \cite{KitVar2004,KitVar2023} are more effective for small values of
$|\rho|$.\footnote{\label{foot:supplement} In a follow-up paper, we are going to discuss the numerical
aspects of these issues.}
\hfill$\blacksquare$\end{remark}

In Theorems~\ref{th:altB1asympt0m} and \ref{th:alt2B1asympt0p}${}^{\mathbf\prime}$ below, we present
asymptotic formulae for the general solution $u(\tau)$ and its associated mole function $\varphi(\tau)$.
In Theorem~\ref{th:altB1asympt0m}, the asymptotics of $u(\tau)$ depends on the two complex parameters
(``constants of integration'') $\varrho$ and---the ratio---$w_1/w_2$, whilst the corresponding asymptotics of
the function $\varphi(\tau)$ contains the additional integration constant---the product---$w_1w_2$;
these constants of integration are parametrized via the monodromy data $s_0^0$, $g_{11}:g_{21}$, and
$s_0^{\infty}$, respectively.
In Theorem~\ref{th:alt2B1asympt0p}${}^{\mathbf\prime}$, the situation is similar; more precisely, the two
integration constants $w_1/w_2$ and $w_1w_2$ are replaced by $w_3/w_4$ and $w_3w_4$, respectively, and
the monodromy parametrization of the latter is obtained via $g_{22}:g_{12}$ and $s_1^{\infty}$, respectively.
\begin{remark}\label{rem:tau0plus}
Throughout the paper, we use the notation $\tau\to0^+$. It can be understood in the usual sense as $|\tau|\to0$ for
$\arg\tau=0$; however, the asymptotics are valid under more general assumptions, namely, $|\tau|\to0$ for
$-\pi<\arg\tau<\pi$. The latter definition will be clarified further in Section~\ref{sec:poles}, which is related
to the study of poles of the function $u(\tau)$.
\hfill$\blacksquare$\end{remark}
\begin{theorem}\label{th:altB1asympt0m}
Let $(u(\tau), \varphi(\tau))$ be a solution of the system~\eqref{eq:dp3}, \eqref{eq:varphi}
corresponding to the monodromy data $(a,s_{0}^{0},s_{0}^{\infty},s_{1}^{\infty},g_{11},g_{12},g_{21},g_{22})$.
Suppose that: {\rm(i)} $s_0^\infty\neq0$, which implies that
\begin{equation}\label{eq:cond-gik-Theo31}
g_{11}\me^{\pi\mi/4}\me^{\mp\pi\mi\varrho}+g_{21}\me^{-\pi\mi/4}\me^{\pm\pi\mi\varrho}\neq0,
\end{equation}
where either the upper signs or the lower signs, respectively, are taken;\footnote{\label{foot:reflection-varrho}
As discussed at the beginning of this section, the parameter $\varrho$ is defined modulo the reflection
$\varrho\to1-\varrho$, so that any one of these values for $\varrho$ can be chosen; in particular, both
restrictions~\eqref{eq:cond-gik-Theo31} hold simultaneously.}
{\rm(ii)} $s_1^\infty\neq0$, thus
\begin{equation}\label{eq:cond-varrho-a-k-Theo31}
\varrho\neq\pm\frac{\mi a}{2}+k,\quad
k\in\mathbb{Z};
\end{equation}
%where the parameter $\varrho$ {\rm(}see equation~\eqref{eq:rho-general} with $\rho\to\varrho$ {\rm)} is
%defined\,\footnote{\label{foot:reflection-varrho} As discussed in the preamble to this section $\varrho$ is defined
%modulo the reflection $\varrho\to1-\varrho$; anyone of these values of $\varrho$ can be chosen.}
%by the conditions~\eqref{eq:varrho-restriction} and
and {\rm(iii)} $s_0^0\neq-2\mi$, which implies
\begin{equation}\label{eq:varrhoneq1/2}
\varrho\neq1/2.
\end{equation}
%\begin{equation*} \label{eq:def-rho-Theo31}
%\varrho\neq1/2,\qquad
%0<\mathrm{Re}(\varrho)<1.
%\end{equation*}
Finally, impose the technical assumption $-2<\mathrm{Im}(a)<0$.\footnote{\label{foot:-2<Ima<0} This assumption is
removed in Theorem~\ref{th:restriction-on-Ima-removed}.}

Then,
\begin{align}
u(\tau)\underset{\tau\to0^{+}}{=}
&\frac{\varepsilon(1-2\varrho)^2w_1w_2\big(1+\mathcal{O}\big(\tau^{4\mathrm{Re}(\varrho)}\big)+
\mathcal{O}\big(\tau^{4-4\mathrm{Re}(\varrho)}\big)\big)}{\tau\big(w_1\tau^{1-2\varrho}+w_2\tau^{-1+2\varrho}\big)^2},
\label{eq:Asympt0u-}\\
\me^{\mi\varphi (\tau)} \underset{\tau \to 0^{+}}{=}
&\me^{\frac{3\pi\mi}{2}}\me^{-\frac{\pi a}{2}}\frac{2\pi}{w_1w_2}\big(2\tau^2\big)^{\mi a}
\big(1+\mathcal{O}\big(\tau^{4\mathrm{Re}(\varrho)}\big)+\mathcal{O}\big(\tau^{4-4\mathrm{Re}(\varrho)}\big)\big),
\label{eq:Asympt0varphi-}
\end{align}
where
\begin{align}
w_{1}=&\left(\tfrac{1}{2}(\varepsilon b)\me^{\pi\mi/2}\right)^{\frac{1}{2}
-\varrho}\frac{\Gamma(2\varrho)}{\Gamma(2-2\varrho)}
\Gamma(1-\varrho+\mi a/2)\left(g_{11}\me^{\pi\mi/4}\me^{-\pi\mi\varrho}
+g_{21}\me^{-\pi\mi/4}\me^{\pi\mi\varrho}\right),
\label{eq:w1}\\
w_{2}=&\left(\tfrac{1}{2}(\varepsilon b)\me^{\pi\mi/2}\right)^{\varrho-\frac{1}{2}}
\frac{\Gamma(2-2\varrho)}{\Gamma(2\varrho)}
\Gamma(\varrho+\mi a/2)\left(g_{11}\me^{\pi\mi/4}\me^{\pi\mi\varrho}
+g_{21}\me^{-\pi\mi/4}\me^{-\pi\mi\varrho}\right),
\label{eq:w2}
\end{align}
and $\Gamma (\ast)$ is the gamma function {\rm \cite{BE1}}.
%\footnote{\label{foot:tau0plus}Throught the paper,
%we use the notation $\tau\to0^+$. It can be understood in the usual sense as $|\tau|\to0$ for $\arg\tau=0$;
%however, asymptotics are valid under more general assumptions, namely, $|\tau|\to0$ for $-\pi<\arg\tau<\pi$.
%The latter definition will be clarified further in Section~\ref{sec:poles}, which is related to the study
%of poles of function $u(\tau)$.}
\end{theorem}
\begin{proof}
The definition of the parameter $\rho$ (and $\varrho$) via the first relation in equation~\eqref{eq:rho-general}
allows one to factorize the left-hand side of equation~\eqref{eq:monodromy:s0} defining $s_0^{\infty}$ as
\begin{equation}\label{eq:s0infty-factorization}
-\mi(g_{11}\me^{\pi\mi/4}\me^{-\pi\mi\varrho}+g_{21}\me^{-\pi\mi/4}\me^{\pi\mi\varrho})
(g_{11}\me^{\pi\mi/4}\me^{\pi\mi\varrho}+g_{21}\me^{-\pi\mi/4}\me^{-\pi\mi\varrho})=\mi s_0^{\infty}\me^{-\pi a}.
\end{equation}
The factorization~\eqref{eq:s0infty-factorization} proves the condition~\eqref{eq:cond-gik-Theo31}.
In order to prove condition ~\eqref{eq:cond-varrho-a-k-Theo31}, we assume that $\varrho=\pm\frac{\mi a}{2}+k$,
$k\in\mathbb{Z}$, hence $\cos(2\pi\varrho)=\cosh(\pi a)$; thus, comparing the left- and right-hand sides of the
relation~\eqref{eq:rho-general}, one arrives at $s_0^{\infty}s_1^{\infty}=0$, which implies
that $s_1^{\infty}=0$, since $s_0^{\infty}\neq0$ is assumed.

Consider the solution $(\tilde{u}(\tau),\tilde{\varphi}(\tau))$ of the system~\eqref{eq:dp3}, \eqref{eq:varphi} with
the monodromy parameters
$(\tilde{a},\tilde{s}_{0}^{0},\tilde{s}_{0}^{\infty},\tilde{s}_{1}^{\infty},\tilde{g}_{11},\tilde{g}_{12},
\tilde{g}_{21},\tilde{g}_{22})$. Assume that $|\mathrm{Im}(a)|<1$ and the corresponding branching parameter
$\tilde{\rho}$ satisfies the conditions $\tilde{\rho}\neq0$ and $|\mathrm{Re}(\tilde{\rho})|<1/2$; then, the
asymptotics as $\tau\to0$ of the functions $\tilde{u}(\tau)$ and $\tilde{\varphi}(\tau)$ are given by Theorem B.1 of
\cite{KitVar2023}.\footnote{\label{foot:condition-g11g22neq0} Theorem B.1 of
\cite{KitVar2023} contains the additional condition $g_{11}g_{22}\neq0$ on the monodromy data; this condition is
removed in \cite{KitVarZapiski2024}.}

Apply the B\"acklund transformations~\eqref{eq:u-} and \eqref{eq:varphi-} to the functions $\tilde{u}(\tau)$ and
$\tilde{\varphi}(\tau)$, respectively, by substituting for these functions their corresponding asymptotic expansions
given in Theorem B.1 of \cite{KitVar2023}. In so doing, one has to take into account that the asymptotic expansions
in Theorem B.1 of \cite{KitVar2023} are differentiable with respect to $\tau$, so that the asymptotics of the
function $\tilde{u}'(\tau)$ is equal to the derivative of the asymptotics of the function $\tilde{u}(\tau)$.
This fact follows from the derivation of the small-$\tau$ asymptotics presented in Section~5 of
\cite{KitVar2004}, where asymptotics of the functions $u(\tau)$ and $u'(\tau)$ are obtained independently in terms
of the asymptotics of the functions $A(\tau)$, $B(\tau)$, $C(\tau)$, and
$D(\tau)$.\footnote{\label{foot:ABCD} These functions define the coefficients of
the $2\times2$ matrix linear ODE whose isomonodromy deformations are governed by the functions $u(\tau)$ and
$\varphi(\tau)$; see Section 1 of \cite{KitVar2004} and Appendix A of \cite{KitVarZapiski2024}.}
Alternatively, the statement regarding the differentiability of the asymptotics of $u(\tau)$ can be confirmed with
the help of the complete asymptotic expansion as $\tau\to0$ given in Appendix~\ref{app:subsec:error-correction-terms}.
As a consequence of the application of the B\"acklund transformations~\eqref{eq:u-} and \eqref{eq:varphi-},
we obtain small-$\tau$ asymptotics for the pair of functions $\tilde{u}_-(\tau)$ and $\tilde{\varphi}_-(\tau)$,
respectively, that are parametrized in terms of the monodromy data corresponding to the original functions
$\tilde{u}(\tau)$ and $\tilde{\varphi}(\tau)$; in particular, for the parameter of formal monodromy,
$\tilde{a}_-=\tilde{a}-\mi$, thus $-2<\mathrm{Im}(\tilde{a}_-)<0$.

Simplify, now, the notation: let $(\tilde{u}_-(\tau),\tilde{\varphi}_-(\tau))\to(u(\tau),\varphi(\tau))$, denote
the corresponding monodromy data without ``tildes'' and ``minus subscripts'', and use equation~\eqref{eq:data-}
to present the small-$\tau$ asymptotics of the functions $u(\tau)$ and $\varphi(\tau)$ in terms of their monodromy
data. Note that, after this ``renotation'', $\tilde{\rho}$ is the only parameter with a ``tilde'' that remains
in the asymptotics of the functions $u(\tau)$ and $\varphi(\tau)$. Taking into account that $\tilde{s}_0^0=-s_0^0$,
one defines the branching parameter $\varrho=1/2-\tilde{\rho}$ so that it solves equation~\eqref{eq:rho-general} and
satisfies the conditions~\eqref{eq:varrho-restriction} and \eqref{eq:varrhoneq1/2}.\footnote{\label{foot:rhoneq0}
As a result of the conditions for the parameter $\tilde{\rho}$ stated at the beginning of the proof.}

After these rearrangements, we arrive at the formulae for the leading terms of the asymptotics stated in
equations~\eqref{eq:Asympt0u-} and \eqref{eq:Asympt0varphi-}. The correction terms in the
asymptotics~\eqref{eq:Asympt0u-} are obtained with the help of the local expansion~\eqref{app:eq:0-u-expansion}.
In \cite{KitVar2004,KitVar2023}, the error for the leading term of asymptotics of the functions $\tilde{u}(\tau)$
and $\me^{\mi\tilde{\varphi}(\tau)}$ is written as the multiplicative factor $(1+o(\tau^{\delta}))$, where the value
of $\delta>0$ was not specified.\footnote{\label{foot:delta}
In fact, the value of $\delta$ can be estimated via the method employed in \cite{KitVar2004}; even though it is
straightforward, it requires more elaborate and cumbersome calculations. It is much easier to find the precise value
for $\delta$ by referring to the local result (see Appendix~\ref{app:subsec:error-correction-terms} for details).}
It is not difficult to see that the B\"acklund transformations preserve the order of the correction term(s), that is,
the asymptotics of the functions $u(\tau)$ and $\me^{\mi\varphi(\tau)}$ can be presented as the product of their
leading terms and the factor $(1+o(\tau^{\delta}))$, where $\delta>0$ is the same as the corresponding one for the
functions with ``tildes''.

Assume that $\tilde\rho\in[0,1/2)$; then, comparing the asymptotics of the function $\tilde{u}(\tau)$ given in
Theorem B.1 of \cite{KitVar2023} (with the change of notation $u(\tau)\to\tilde{u}(\tau)$) to the
expansion~\eqref{app:eq:0-u-expansion} (once again with the change of notation $u(\tau)\to\tilde{u}(\tau)$), we find
that the $o(\tau^{\delta})$ term is, in fact, equal to $\mathcal{O}(\tau^{2-\sigma})$, where $\sigma=4\tilde\rho$.
Thus, taking into account that $\tilde\rho=1/2-\varrho$, we find that the correction term to the leading term of
asymptotics can be presented as the multiplicative factor
$\big(1+\mathcal{O}\big(\tau^{4\mathrm{Re}\,\varrho}\big)\big)$. In this case, i.e.,
$0\leqslant\mathrm{Re}\,\tilde{\rho}<1/2$, the term of order $\mathcal{O}\big(\tau^{4-4\mathrm{Re}\,\varrho}\big)$
can be omitted.
For the case $\tilde\rho\in(-1/2,0]$, similar considerations imply that the correction term can be presented in
the multiplicative form $\big(1+\mathcal{O}\big(\tau^{4-4\mathrm{Re}\,\varrho}\big)\big)$
(the reflection!).\footnote{\label{foot:corr-term-Section3} An alternative derivation for the correction term
is given in Appendix~\ref{app:sec:full0expansion}, Remark~\ref{app:eq:Theorem3.1error-corr-term}.}
Finally, to obtain the correction term for the asymptotics~\eqref{eq:Asympt0varphi-}, one integrates
equation~\eqref{eq:varphi} using the expansion~\eqref{app:eq:0-u-expansion}.
\end{proof}
\begin{remark}\label{rem:alt2B1asympt0m}
Instead of applying the B\"acklund transformations~\eqref{eq:u-} and \eqref{eq:varphi-} to the solution
$(\tilde{u}(\tau),\tilde{\varphi}(\tau))$ as done in the proof of Theorem~\ref{th:altB1asympt0m}, we can use,
instead, the B\"acklund transformations~\eqref{eq:u+} and \eqref{eq:varphi+}. Repeating, \emph{verbatim}, the
construction delineated in the proof of Theorem~\ref{th:altB1asympt0m}, with, of course, the obvious replacement
of the reference to equation~\eqref{eq:data-} by a reference to equation~\eqref{eq:data+}, we arrive at
Theorem~\ref{th:alt2B1asympt0p}${}^{\mathbf\prime}$ below.
\hfill$\blacksquare$\end{remark}
\begin{theorems}\hspace{-10pt}${}^{\mathbf\prime}$\label{th:alt2B1asympt0p}
Let $(u(\tau), \varphi(\tau))$ be a solution of the system~\eqref{eq:dp3}, \eqref{eq:varphi} corresponding to the
monodromy data $(a,s_{0}^{0},s_{0}^{\infty},s_{1}^{\infty},g_{11},g_{12},g_{21},g_{22})$.
Suppose that:
{\rm(i)} $s_1^\infty\neq0$, which implies that
\begin{equation}\label{eq:cond-gik-Theo31prime}
g_{12}\me^{\pi\mi/4}\me^{\mp\pi\mi\varrho}+g_{22}\me^{-\pi\mi/4}\me^{\pm\pi\mi\varrho}\neq0,
\end{equation}
where either the upper signs or the lower signs, respectively, are
taken;\footnote{\label{foot:gik-conditions-Th31prime} Both conditions~\eqref{eq:cond-gik-Theo31prime} hold
simultaneously (see footnote~\ref{foot:reflection-varrho}).}
{\rm(ii)} $s_0^\infty\neq0$, thus
\begin{equation*}\label{eq:cond-varrho-a-k-Theo31prime}
\varrho\neq\pm\frac{\mi a}{2}+k,\quad
k\in\mathbb{Z};
\end{equation*}
and {\rm(iii)} $s_0^0\neq-2\mi$, which implies
\begin{equation*}\label{eq:varrhoneq1/2prime}
\varrho\neq1/2.
\end{equation*}
Finally, impose the technical assumption $0<\mathrm{Im}(a)<2$.\footref{foot:-2<Ima<0}

Then,
\begin{align}
u(\tau)\underset{\tau\to0^{+}}{=}
&\frac{\varepsilon(1-2\varrho)^2w_3w_4\big(1+\mathcal{O}\big(\tau^{4\mathrm{Re}(\varrho)}\big)+
\mathcal{O}\big(\tau^{4-4\mathrm{Re}(\varrho)}\big)\big)}
{\tau\big(w_3\tau^{1-2\varrho}+w_4\tau^{-1+2\varrho}\big)^2},\label{eq:Asympt0u+}\\
\me^{\mi\varphi (\tau)} \underset{\tau \to 0^{+}}{=}
&\me^{\frac{3\pi\mi}{2}}\me^{\frac{\pi a}{2}}\frac{w_3w_4}{2\pi}\big(2\tau^2\big)^{\mi a}
\big(1+\mathcal{O}\big(\tau^{4\mathrm{Re}(\varrho)}\big)+
\mathcal{O}\big(\tau^{4-4\mathrm{Re}(\varrho)}\big)\big),\label{eq:Asympt0varphi+}
\end{align}
where
\begin{align}
w_{3}&=\left(\tfrac{1}{2}(\varepsilon b)\me^{-\pi\mi/2}\right)^{\frac{1}{2}-\varrho}\frac{\Gamma(2\varrho)}{\Gamma(2-2\varrho)}
\Gamma(1-\varrho-\mi a/2)\left(g_{12}\me^{\pi\mi/4}\me^{-\pi\mi\varrho}+g_{22}\me^{-\pi\mi/4}\me^{\pi\mi\varrho}\right),
\label{eq:w3}\\
w_{4}&=\left(\tfrac{1}{2}(\varepsilon b)\me^{-\pi\mi/2}\right)^{\varrho-\frac{1}{2}}\frac{\Gamma (2-2\varrho)}{\Gamma(2\varrho)}
\Gamma(\varrho-\mi a/2) \! \left(g_{12} \me^{\mi \pi/4}\me^{\pi\mi\varrho}+g_{22}\me^{-\pi\mi/4}\me^{-\pi\mi\varrho}\right).
\label{eq:w4}
\end{align}
\end{theorems}
\begin{remark}\label{rem:symmetry}
Note that
$w_k(\varrho)=-w_{k+1}(1-\varrho)$, $k=1,3$, which
manifests the invariance of the asymptotics of $u(\tau)$ and $\me^{\mi\varphi(\tau)}$ under the reflection
$\varrho\to1-\varrho$.
\hfill$\blacksquare$\end{remark}
\begin{proposition}\label{prop:equivalence Th31Th31prime}
The following identities hold:
\begin{equation}\label{eqs:w-identities}
w_1w_2w_3w_4=(2\pi)^2\me^{-\pi a},\qquad
\frac{w_1}{w_2}=\frac{w_3}{w_4}.
\end{equation}
\end{proposition}
\begin{proof}
Straightforward calculations using definitions \eqref{eq:w1}, \eqref{eq:w2}, \eqref{eq:w3},
and \eqref{eq:w4}.
\end{proof}
\begin{corollary}\label{cor:Th31=Th32}
Theorems~{\rm\ref{th:altB1asympt0m}} and {\rm\ref{th:alt2B1asympt0p}}${}^{\mathbf\prime}$ are valid for
monodromy data subject to the conditions $(s_0^0+2\mi)s_0^{\infty}s_1^{\infty}\neq0$ and
$\mathrm{Im}\,a\in(-2,0)\cup(0,2)${\rm;}
in particular, the asymptotics of the functions $u(\tau)$ and $\me^{\mi\varphi(\tau)}$ corresponding to the same set
of monodromy data defined by these theorems coincide.
\end{corollary}
\begin{proof}
Using the identities proved in Proposition~\ref{prop:equivalence Th31Th31prime}, we find that the asymptotic formulae
for the functions $u(\tau)$ and $\me^{\mi\varphi(\tau)}$ given in Theorem~\ref{th:altB1asympt0m} for
$\mathrm{Im}\,a\in(-2,0)$ coincide with the asymptotic formulae for these functions given in
Theorem~\ref{th:alt2B1asympt0p}${}^{\mathbf\prime}$ for $\mathrm{Im}\,a\in(0,2)$.
\end{proof}
\begin{theorem}\label{th:restriction-on-Ima-removed}
Theorems~{\rm\ref{th:altB1asympt0m}} and {\rm\ref{th:alt2B1asympt0p}}\hspace{-0pt}${}^{\mathbf\prime}$ are
valid for all $\mathrm{Im}(a)\in\mathbb{R}$ provided that all the other conditions stated therein hold.
\end{theorem}
\begin{proof}
The proof proceeds via the following bootstrap-type argument.

Firstly, reference to Corollary~\ref{cor:Th31=Th32} proves the statement of the theorem for
$\mathrm{Im}(a)\in(-2,0)\cup(0,2)$.

Secondly, we note that for the general solutions ($\mathrm{Re}\,\varrho,\mathrm{Re}\,\rho\in(0,1)$), the
asymptotic expansions presented in Theorems~\ref{th:altB1asympt0m} and
\ref{th:alt2B1asympt0p}\hspace{-0pt}${}^{\mathbf\prime}$ and in Theorem B.1 of \cite{KitVar2023} coincide, for all
values of $a$, modulo the corresponding correction terms up to which they are considered. Furthermore, these
correction terms, as well as the explicitly written leading terms, are holomorphic functions of $a$ at $a=0$;
therefore, all three asymptotic formulae at $a=0$ define asymptotics of the same functions $u(\tau)$ and
$\me^{\mi\varphi(\tau)}$, provided they are constructed in terms of the monodromy data corresponding to these
functions, as this fact is proved for one of these asymptotics, that is, the asymptotics obtained in Theorem B.1
of \cite{KitVar2023}. Thus, the theorem is valid for $\mathrm{Im}(a)\in(-2,2)$.

We now begin the bootstrap procedure.
Apply to the asymptotics of Theorem~\ref{th:altB1asympt0m} the B\"acklund transformations~\eqref{eq:u+} and
\eqref{eq:varphi+}, which are the inverses of the transformations used to obtain the asymptotics of this theorem;
consequently, one arrives at the original asymptotics stated in Theorem B.1 of \cite{KitVar2023}, but with the
interval of validity of the original asymptotics extended from $\mathrm{Im}\,a\in(-1,1)$ to $\mathrm{Im}\,a\in(-1,3)$.
If one applies, in an analogous manner, the B\"acklund transformations~\eqref{eq:u-} and \eqref{eq:varphi-} to
the asymptotics given in Theorem~\ref{th:alt2B1asympt0p}\hspace{-0pt}${}^{\mathbf\prime}$, then one concludes that
the original asymptotics in Theorem B.1 of \cite{KitVar2023} are valid in the interval $\mathrm{Im}\,a\in(-3,1)$.
Thus, the original asymptotics are valid for $\mathrm{Im}\,a\in(-3,3)$ instead of just for
$\mathrm{Im}\,a\in(-1,1)$. Subsequently, we revert back to the proof of Theorem~\ref{th:altB1asympt0m} and establish
that it is, in fact, true for $\mathrm{Im}\,a\in(-4,2)$, and that the corresponding asymptotics in
Theorem~\ref{th:alt2B1asympt0p}\hspace{-0pt}${}^{\mathbf\prime}$ are valid for $\mathrm{Im}\,a\in(-2,4)$, that is,
both theorems are applicable for $\mathrm{Im}\,a\in(-4,4)$.
This procedure can be repeated as many times as is necessary in order to arrive at the telescoping system of
intervals of validity of the theorem.

A mathematical induction argument completes the proof.
\end{proof}
%%%%%%%%%%%%%%%%%%%%%%%%%%%%%%%%%%%%%%%%%%%%%%%%%%%%%%%%%%%%%%%%%%%%%%%%%%%%%%%%%%%%%%%%%%%%%%%%%%%%%%%%%%%%%%%%%%%%
%%%%%%%%%%%%%%%%%%%%%%%%%%%%%%%%%%%%%%%%%%%%%%%%%%%%%%%%%%%%%%%%%%%%%%%%%%%%%%%%%%%%%%%%%%%%%%%%%%%%%%%%%%%%%%%%%%%%
\section{Power-Like Small-$\tau$ Asymptotics: Special Cases for $\rho\neq0$ and $\varrho\neq1/2$}
\label{sec:special cases}
The following theorems describe one-parameter families of solutions corresponding to special cases of the
monodromy data that were excluded from Theorems~\ref{th:alt2B1asympt0p} and
\ref{th:alt2B1asympt0p}\hspace{-0pt}${}^{\mathbf\prime}$.
Unlike Section~\ref{sec:general}, the branching parameters $\varrho$ and $\sigma$ in this section are fixed
in terms of $a$.
\begin{theorem}\label{th:Asympt0-rho-eq-1pn+ia2}
Let $(u(\tau), \varphi(\tau))$ be a solution of the system~\eqref{eq:dp3}, \eqref{eq:varphi} corresponding to the
monodromy data
$(a,s_{0}^{0},s_{0}^{\infty},s_{1}^{\infty},g_{11},g_{12},g_{21},g_{22})$.
Suppose that
\begin{equation}\label{eqs:conditions-th3.1-limit1}
a\neq\mi k,\quad
k\in\mathbb{Z},\qquad
s_0^{\infty}=0,\quad
\mathrm{and}\quad
s_1^{\infty}\neq0;
\end{equation}
then, $g_{21}\in\mathbb{C}\setminus\{0\}$, and the remaining monodromy data are given by the following equations:
\begin{equation}\label{eqs:spec-monodromy-th3.1-1}
s_0^0=2\mi\cosh(\pi a),\quad
g_{11}=\mi\me^{-\pi a}g_{21},\quad
g_{12}=-\frac{\me^{\pi a}+\mi s_1^{\infty}g_{21}^2}{2\sinh(\pi a)g_{21}},\quad
g_{22}=\frac{\mi-\me^{\pi a}s_1^{\infty}g_{21}^2}{2\sinh(\pi a)g_{21}}.
\end{equation}
\begin{enumerate}
\item[\pmb{$(1)$}]\label{Th4.1case1}
Assume that
$\mathrm{Im}\,a>0$, and define---uniquely---numbers $\varrho\in\mathbb{C}$, with $\mathrm{Re}\,\varrho\in(0,1)$, and
$n\in\mathbb{Z}_{\geqslant0}$ such that $\varrho=1+n+\mi a/2$, i.e.,
$\lfloor{\mathrm{Im}\,a/2}\rfloor=n$ and $\mathrm{Re}\,\varrho=1-\{\mathrm{Im}\,a/2\}$, where $\lfloor\cdot\rfloor$
and $\{\cdot\}$ denote, respectively, the floor and the fractional part of the real number; then,
\begin{align}
u(\tau)\underset{\tau\to0^{+}}{=}
&\frac{\varepsilon(1-2\varrho)^2\hat{w}_1\hat{w}_2\big(1+\mathcal{O}\big(\tau^{4\mathrm{Re}(\varrho)}\big)+
\mathcal{O}\big(\tau^{4(1-\mathrm{Re}(\varrho))}\big)\big)}
{\tau\big(\hat{w}_1\tau^{1-2\varrho}+\hat{w}_2\tau^{-1+2\varrho}\big)^2},\label{eq:Asympt0spec1u-}\\
\me^{\mi\varphi (\tau)} \underset{\tau \to 0^{+}}{=}
&\me^{\frac{3\pi\mi}{2}}\me^{-\frac{\pi a}{2}}\frac{2\pi}{\hat{w}_1\hat{w}_2}\big(2\tau^2\big)^{\mi a}
\big(1+\mathcal{O}\big(\tau^{4\mathrm{Re}(\varrho)}\big)+
\mathcal{O}\big(\tau^{4(1-\mathrm{Re}(\varrho))}\big)\big),\label{eq:Asympt0spec1varphi-}
\end{align}
where
\begin{align}
\hat{w}_{1}=&\left(\tfrac{1}{2}(\varepsilon b)\me^{\pi\mi/2}\right)^{\frac{1}{2}-\varrho}\frac{2\pi}{n!}
\frac{\Gamma(2\varrho)}{\Gamma(2-2\varrho)}\frac{\me^{3\pi\mi/4-3\pi a/2}}{s_1^{\infty}g_{21}},
\label{eq:omega1}\\
\hat{w}_{2}=&\left(\tfrac{1}{2}(\varepsilon b)\me^{\pi\mi/2}\right)^{\varrho-\frac{1}{2}}\me^{\pi\mi(\varrho-1/4)}
\frac{\Gamma(2-2\varrho)}{\Gamma(2\varrho)}
\Gamma(2\varrho-n-1)\,2\sinh(\pi a)\,g_{21}.
\label{eq:omega2}
\end{align}
\item[\pmb{$(2)$}]\label{Th4.1case2}
Assume that
$-1<\mathrm{Im}\,a<1$, and define
\begin{equation}\label{eqs:sigma-b1-1}
\sigma=-2\mi a,\quad
b_{1,-1}=-\mi\left(\frac{\varepsilon b}{2}\right)^{1+\mi a}\frac{\pi\me^{\pi a/2}}{\sinh(\pi a)}
\frac{s_1^\infty g_{21}^2}{\big(\Gamma(1+\mi a)\big)^3};
\end{equation}
then,
\begin{align}
u(\tau)\underset{\tau\to0^+}{=}&\;
\frac{\varepsilon b_{1,-1}\tau^{1-\sigma}}{\left(1+\frac{4 b_{1,-1}\tau^{2-\sigma}}{(\sigma-2)^2}\right)^2}
-\frac{b\tau}{2a}+\mathcal{O}\big(\tau^{3-\sigma}\big)+\mathcal{O}\big(\tau^{3}\big),
\label{eq:Asympt0-rhoEq+Ia:2u}\\
\me^{\mi\varphi(\tau)}\underset{\tau\to0^+}{=}&\;
\frac{\me^{\pi a}}{2\pi a g_{21}^2}
\left(\me^{\frac{\pi a}{2}}\Gamma(1-\mi a)s_1^{\infty}g_{21}^2\left(2\tau^2\right)^{\mi a}-
\mi\big(\Gamma(1+\mi a)\big)^2\left(\frac{4}{\varepsilon b}\right)^{\mi a}\right)\nonumber\\
&\times\left(1+\mathcal{O}\big(\tau^{2}\big)+\mathcal{O}\big(\tau^{2+2\mi a}\big)\right).
\label{eq:Asympt0-rhoEq+Ia:2phi}
\end{align}
\item[\pmb{$(3)$}]\label{Th4.1case3}
Assume that
$n-1<-\mathrm{Im}\,a<n$, $n\in\mathbb{N}$, or $\mathrm{Im}\,a=-(n-1)$ and $\mathrm{Re}\,a\neq0$.
Let $\sigma$ and $b_{1,-1}$ be defined by equation~\eqref{eqs:sigma-b1-1}, in particular,
$2(n-1)<-\mathrm{Re}\,\sigma<2n${\rm;} then,
\begin{align}
\varepsilon u(\tau)\underset{\tau\to0^+}{=}&\,
\sum_{k=1}^n b_{2k-1,0}\tau^{2k-1}+ b_{1,-1}\tau^{1-\sigma}+\mathcal{O}\big(\tau^{2n+1}\big),
\label{eq:Asympt0-u-long-b1-1}\\
\me^{\mi\varphi(\tau)}\underset{\tau\to0^+}{=}&\,
-\frac{\mi\me^{\pi a}\big(\Gamma(1+\mi a)\big)^2}{2\pi ag_{21}^2}\left(\frac{\varepsilon b}{4}\right)^{-\mi a}
\exp\left(-\mi\left(P_n(\tau)-\frac{4a^2}{\varepsilon b}b_{1,-1}\frac{\tau^{-\sigma}}{\sigma}
+\mathcal{O}\big(\tau^{2n}\big)\right)\right),
\label{eq:Asympt0-phi-long-b1-1}
\end{align}
where $b_{1,0}=-\frac{\varepsilon b}{2a}$,\footnote{\label{foot:sigmaEQ-4a2} See the second
equation in~\eqref{app:eq:sigma-b0-b1pm1}, where $\sigma^2=-4a^2$.}
the numbers $b_{2l+1,0}$ are the ``middle terms'' of the asymptotic expansion for $u(\tau)$ defined in
Appendix~{\rm\ref{app:subsec:error-correction-terms}}, and the polynomials $P_n(\tau)$ are given by
\begin{gather}
P_n(\tau)=\sum_{N=1}^{n-1}p_N\tau^{2N},\label{eq:Pn-def}\\
p_N=\frac{a}{N}\sum_{k=1}^N\left(\frac{2a}{\varepsilon b}\right)^k\sum_{\{m_1,\ldots,m_{N}\}\in M_{k,N}}
\frac{(m_1+\ldots+m_{N})!}{m_1!\cdot\ldots\cdot m_{N}!}\prod_{l=1}^{N}(b_{2l+1,0})^{m_l},\label{eq:pN}
%P_n(\tau)&=2a\sum_{k=1}^{n-1}\left(\frac{2a}{\varepsilon b}\right)^k\sum_{\{m_1,m_2,\ldots,m_{n-1}\}\in M_{k,n}}
%\frac{(m_1+m_2+\ldots+m_{n-1})!}{m_1!m_2!\cdots m_{n-1}!}\prod_{l=1}^{n-1}(b_{2l+1,0})^{m_l}\\
%&\times\frac{\tau^{2m_1+4m_2+\ldots+2(n-1)m_{n-1}}}{2m_1+4m_2+\ldots+2(n-1)m_{n-1}},
\end{gather}
where the summation set $M_{k,N}$ consists of the sets of numbers $m_i\in\mathbb{Z}_{\geqslant0}$, $i=1,2,\ldots,N$, that
solve the system
\begin{equation}\label{eqs:MkN}
\begin{gathered}
m_1+\ldots+m_i+\ldots+m_{N}=k,\\
m_1+\ldots+im_i+\ldots+Nm_{N}= N.
\end{gathered}
\end{equation}
\end{enumerate}
\end{theorem}
\begin{proof}
If $s_0^{\infty}=0$, then the relation~\eqref{eq:rho-general} implies the equation for $s_0^0$ given in the
list~\eqref{eqs:spec-monodromy-th3.1-1}.
Equation~\eqref{eq:monodromy:s0} is equivalent to the condition $g_{11}=\mi\me^{\pm\pi a}g_{21}$; but,
equations~\eqref{eq:monodromy:main} and \eqref{eq:monodromy:detG}, together with the assumption $a\neq\mi k$,
$k\in\mathbb{Z}$, exclude the possibility $g_{11}=\mi\me^{\pi a}g_{21}$, so one arrives at the second equation in the
list~\eqref{eqs:spec-monodromy-th3.1-1}.
Note that, because of the aforementioned condition on $g_{11}$ and equation~\eqref{eq:monodromy:detG},
$g_{11}g_{21}\neq0$.
Choose $g_{21}\neq0$ and $s_1^{\infty}g_{21}^2$ as the parameters defining the solution
$(u(\tau),\me^{\mi\varphi(\tau)})$. Substituting $g_{11}=\mi\me^{-\pi a}g_{21}$ into
equation~\eqref{eq:monodromy:detG} and dividing both sides of the resulting equation by $g_{21}$, one gets a linear
equation with respect to $g_{12}$ and $g_{22}$. The second linear equation with respect to these co-ordinates is
obtained analogously via equation~\eqref{eq:monodromy:s1}: using the formula $s_0^0=2\mi\cosh(\pi a)$, the left-hand
side of equation~\eqref{eq:monodromy:s1} can be factorised as two linear forms with respect to $g_{12}$ and $g_{22}$,
wherein one of these forms coincides with the left-hand side of the linear equation with respect to $g_{12}$ and
$g_{22}$ already derived using equation~\eqref{eq:monodromy:detG} and, therefore, can be removed with the help of
this linear equation.
Thus, solving the obtained linear system (its discriminant is non-vanishing provided $a\neq\mi k$, $k\in\mathbb{Z}$),
one arrives at the last two equations in the list~\eqref{eqs:spec-monodromy-th3.1-1}.

For $s_0^{\infty}=0$, it follows that $\varrho=n+1\pm\mi a/2$, $n\in\mathbb{Z}$
(cf. equation~\eqref{eq:rho-general}). As a consequence of the symmetry discussed in Remark~\ref{rem:symmetry},
we can, without loss of generality, assume that
$\varrho=n+1+\mi a/2$.\footnote{\label{foot:varrho-reflection-lost} Note that, because of this choice for $\varrho$,
the reflection symmetry $\varrho\to1-\varrho$ in the asymptotics~\eqref{eq:Asympt0spec1u-} and
\eqref{eq:Asympt0spec1varphi-} (cf. equations~\eqref{eq:omega1} and \eqref{eq:omega2}) is lost.}
Also, note that $\varrho\neq1/2$ because $a\notin\mi\mathbb{Z}$.

We now proceed to the proof of item {\bf(1)} of the theorem, that is, assume $\mathrm{Im}\,a>0$ and
$n\in\mathbb{Z}_{\geqslant0}$.
Substituting $\varrho=n+1+\mi a/2$ and the expression for $g_{11}$ in the list~\eqref{eqs:spec-monodromy-th3.1-1}
into equation~\eqref{eq:w2} for $w_2$, one verifies equation~\eqref{eq:omega2} for $\hat{w}_2$. The derivation of
equation~\eqref{eq:omega1} for $\hat{w}_1$ is not as straightforward, because substituting the same expressions
for $\varrho$ and $g_{11}$ into equation~\eqref{eq:w1} for $w_1$ gives rise to the appearance of the term
$\Gamma(-n)$ as the value of the right-most $\Gamma$-function, which, for $n\in\mathbb{Z}_{\geqslant0}$, is its
valuation at the pole $-n$.
At the same time, though,  the right-most
(parenthetical) term in equation~\eqref{eq:w1}, which consists of the linear combination of the monodromy data,
vanishes for the monodromy data~\eqref{eqs:spec-monodromy-th3.1-1}; thus, we get an indeterminate
expression. This indeterminacy can be resolved in several ways:

\textbf{(i)} the simplest way is to use Theorem~\ref{th:alt2B1asympt0p}\hspace{-1pt}${}^{\mathbf\prime}$.
The key point here is to note that the coefficients $w_3$ and $w_4$ are finite for the monodromy
data~\eqref{eqs:spec-monodromy-th3.1-1}; therefore, for this set of monodromy data and $\mathrm{Im}\,a>0$, one
can use the asymptotics~\eqref{eq:Asympt0u+} and \eqref{eq:Asympt0varphi+} without any modifications. These
asymptotics resemble (as functions of $\tau$) those stated in equations~\eqref{eq:Asympt0spec1u-} and
\eqref{eq:Asympt0spec1varphi-}; however, the coefficients $w_3$ and $w_4$, calculated for the monodromy
data~\eqref{eqs:spec-monodromy-th3.1-1}, look different than the coefficients $\hat{w}_1$ and $\hat{w}_2$ presented
in equations~\eqref{eq:omega1} and \eqref{eq:omega2}. To get the exact correspondence, one has to use the
relations~\eqref{eqs:w-identities}, which have an algebraic nature and hold for any parametrization(s) of the
coefficients $w_k$, $k=1,2,3,4$, in terms of the monodromy data;

\textbf{(ii)} a modification of the approach suggested in item \textbf{(i)}, but using
Theorem~\ref{th:altB1asympt0m} in lieu of Theorem~\ref{th:alt2B1asympt0p}\hspace{-1pt}${}^{\mathbf\prime}$.
As discussed above, the parameter $w_2$ in the asymptotics being calculated for the
monodromy data~\eqref{eqs:spec-monodromy-th3.1-1} coincides with $\hat{w}_2$ given in equation~\eqref{eq:omega2},
while the formula for the parameter $w_1$ for the monodromy data~\eqref{eqs:spec-monodromy-th3.1-1} does not give
rise to a definite result. To find the value of
$w_1$ for the said monodromy data, which is denoted as $\hat{w}_1$ in equation~\eqref{eq:omega1}, one can use,
again, the results of Proposition~\ref{prop:equivalence Th31Th31prime}, where any one of the
relations~\eqref{eqs:w-identities} can be employed for this purpose; and

\textbf{(iii)} a direct resolution of the indeterminacy problem that provides an alternative
proof of equation~\eqref{eq:omega1}. The standard method for
resolving such indeterminacies is to consider a proper limiting procedure, which we now proceed to outline.
Define a small parameter $\delta$ via the equation $\varrho=1+n+\mi a/2-\delta$. Substituting this expression for
$\varrho$ into the argument of the right-most $\Gamma$-function in equation~\eqref{eq:w1}, we get
\begin{equation}\label{eq:Gamma-delta-At-pole}
\Gamma(-n+\delta)=\frac{(-1)^n\pi}{\Gamma(n+1-\delta)\sin(\pi\delta)}.
\end{equation}
We have to take a limit along a curve on the monodromy manifold; therefore, we have to find  infinitesimal
$\mathcal{O}(\delta)$-corrections to the monodromy data at the point~\eqref{eqs:spec-monodromy-th3.1-1}
parametrized by $s_1^{\infty}$ and $g_{21}$. Before doing so, however, we consider the right-most (parenthetical)
term in equation~\eqref{eq:w1}, where we denote by $\kappa$ an infinitesimal correction to the parameter $g_{11}$,
namely, $g_{11}=\mi\me^{-\pi a}g_{21}+\kappa$; then, after a straightforward calculation, we find that
\begin{equation}\label{eq:last()w1-delta-kappa}
g_{11}\me^{\pi\mi/4}\me^{-\pi\mi\varrho}+g_{21}\me^{-\pi\mi/4}\me^{\pi\mi\varrho}=
\kappa\me^{\pi\mi/4+\pi a/2+\pi\mi(n+1)}-2\pi\mi\delta g_{21}\me^{-\pi\mi/4-\pi a/2-\pi\mi(n+1)}+\mathcal{O}(\delta^2).
\end{equation}
To find the $\delta$-dependence of $\kappa$, we have to use
equations~\eqref{eq:monodromy:s}--\eqref{eq:monodromy:detG} defining the monodromy manifold and
equation~\eqref{eq:rho-general} for $\varrho$. First, we find the perturbation of
the Stokes multipliers:
\begin{equation}\label{eqs:stokes-perturbations}
s_0^0=2\mi\cosh(\pi a)-4\pi\delta\sinh(\pi a)+\mathcal{O}(\delta^2),\qquad
s_0^{\infty}=4\pi\mi\sinh(\pi a)\me^{-\pi a}\delta/s_1^{\infty}+\mathcal{O}(\delta^2).
\end{equation}
With the help of equations~\eqref{eqs:stokes-perturbations}, one finds
\begin{equation}\label{eq:kappa-general}
\kappa=\frac{4\pi\mi\delta\sinh(\pi a)\me^{-\pi a}g_{21}^2g_{22}}{1+2\mi\sinh(\pi a)g_{21}g_{22}}+\mathcal{O}(\delta^2).
\end{equation}
Substituting into equation~\eqref{eq:kappa-general} the formula for $g_{22}$ given in
the list~\eqref{eqs:spec-monodromy-th3.1-1}, one gets
\begin{equation}\label{eq:kappa-spec4.1}
\kappa=-\frac{2\pi\mi\delta}{s_1^{\infty}g_{21}}\me^{-2\pi a}+2\pi\delta g_{21}\me^{-\pi a}+\mathcal{O}(\delta^2).
\end{equation}
Substituting the expression for $\kappa$ given in equation~\eqref{eq:kappa-spec4.1} into
equation~\eqref{eq:last()w1-delta-kappa}, we observe that the terms without $s_1^{\infty}$ cancel! Taking this into
account together with equation~\eqref{eq:Gamma-delta-At-pole},
we simplify equation~\eqref{eq:w1}, and, denoting by $\hat{w}_1$ the special value of $w_1$ corresponding to the
monodromy data~\eqref{eqs:spec-monodromy-th3.1-1}, arrive at equation~\eqref{eq:omega1}.

We now turn our attention to the proof of item \pmb{$(2)$} of the theorem. Here, we rely upon our basic result
as formulated in Appendix B of \cite{KitVar2023} and the local expansion studied in
Appendix~\ref{app:sec:full0expansion} below.

Note that the case $0<\mathrm{Im}\,a<1$ has already been considered in item \pmb{$(1)$} of the theorem;
but, the leading terms of the corresponding asymptotics look different, and
the correction term in item \pmb{$(2)$} is more precise than the one in item \pmb{$(1)$}.
So, our goal is to prove that the leading terms coincide, and to justify the correction term stated in
item \pmb{$(2)$}.

The case under consideration corresponds to $n=0$ in the formula for the parameter $\varrho$ given in
item \pmb{$(1)$}, so that $\varrho=1+\mi a/2$; this formula implies that $\mathrm{Re}\,\varrho\in(1/2,1)$, which
means that $\tau^{1-2\varrho}>\tau^{-1+2\varrho}$.
Taking the last fact into account, we compare the asymptotics~\eqref{eq:Asympt0spec1u-} and
\eqref{eq:Asympt0spec1varphi-} with the asymptotics~\eqref{eq:Asympt0-rhoEq+Ia:2u} and
\eqref{eq:Asympt0-rhoEq+Ia:2phi}, respectively. Comparing these formulae, we find that the following relations hold:
$\sigma=4(1-\varrho)$ and $(1-2\varrho)^2\hat{w}_2/\hat{w}_1=b_{1,-1}$.
Both formulae can be validated with the help of equations~\eqref{eq:omega1}--\eqref{eqs:sigma-b1-1}.

The term proportional to $\tau$ in the asymptotics~\eqref{eq:Asympt0spec1u-} and the $\mathcal{O}(1)$ term in
the asymptotics~\eqref{eq:Asympt0spec1varphi-}, which are explicitly written in the
asymptotics~\eqref{eq:Asympt0-rhoEq+Ia:2u} and \eqref{eq:Asympt0-rhoEq+Ia:2phi}, respectively, are concealed in the
correction terms. To justify the correction term in the
asymptotics~\eqref{eq:Asympt0-rhoEq+Ia:2u}, we refer to the expansion~\eqref{app:eq:0-u-expansion}, wherein, due to
the last relation in the list of equations~\eqref{app:eq:sigma-b0-b1pm1} and the fact that $\sigma=-2\mi a$, one has
to set $b_{1,1}=0$. This relation
does not hold for the solutions in item \pmb{$(1)$} with $n\geqslant1$. Since $\mathrm{Re}\,\sigma>0$, the
$\mathcal{O}\big(\tau^3\big)$ correction term in equation~\eqref{eq:Asympt0-rhoEq+Ia:2u} can be omitted; in fact,
the largest correction term, i.e., $\mathcal{O}\big(\tau^{3-2\sigma})$, in the expansion~\eqref{app:eq:0-u-expansion}
is hidden in the denominator of the leading term (cf. the function $A_0(x)$ in
Appendix~\ref{app:subsec:super-generating-function}). For the function
$\me^{\mi\varphi(\tau)}$, we also have a more precise formula for the correction term than the corresponding one
in item \pmb{$(1)$}: the leading term of this formula is special case of the asymptotic formula (B.5) in Theorem B.1
of \cite{KitVar2023}. The error corrections for $\varphi(\tau)$ are obtained by substituting the
expansion~\eqref{app:eq:0-u-expansion} into equation~\eqref{eq:varphi} and integrating the resulting expansion.

Turning to the case $-1<\mathrm{Im}\,a\leqslant0$, we find it convenient to refer to the description of the
asymptotics in terms of the parameter $\rho$. Since both $\rho$ and $\varrho$ solve the same
equation~\eqref{eq:rho-general}, we can present $\rho$ as $\rho=1+n+\mi a/2$ for some $n\in\mathbb{Z}$.
The integer $n=-1$ because $|\mathrm{Re}\,\rho|<1/2$. Comparing the asymptotics given in Theorem B.1 of
\cite{KitVar2023} with the expansion~\eqref{app:eq:0-u-expansion} and taking into account the symmetry of this
expansion with respect to the transformation $\sigma\to-\sigma$, we put $\sigma=-4\rho$ and once again obtain the
relation $\sigma=-2\mi a$. Then, according to the last equation in the list~\eqref{app:eq:sigma-b0-b1pm1}, one finds
that $b_{1,1}b_{1,-1}=0$. Equations (B.4) and (B.6)--(B.8) in Appendix B of \cite{KitVar2023} show that
$b_{1,-1}$ is given by equation~\eqref{eqs:sigma-b1-1}, so that $b_{1,1}=0$. The asymptotic formula for
$\me^{\mi\varphi(\tau)}$ is a special case of the asymptotics (B.5) in Appendix B of \cite{KitVar2023}.
The corrections written in equation~\eqref{eq:Asympt0-rhoEq+Ia:2u}
are based on the local expansion \eqref{app:eq:0-u-expansion}, where we have taken into account that $b_{1,1}=0$, and
therefore $b_{3,k}=0$ for $k=1,2$. We now consider the derivation of these corrections more carefully.

If $\mathrm{Im}\,a=0$, then the first two explicitly written $\mathcal{O}(\tau)$ terms represent the leading term of
asymptotics, while the correction is of the order $\mathcal{O}\big(\tau^3\big)$. In this case, the denominator of
the first fraction in the asymptotics~\eqref{eq:Asympt0-rhoEq+Ia:2u} can be omitted because its contribution is of the order
$\mathcal{O}\big(\tau^3\big)$.

If $-1<\mathrm{Im}\,a<0$, then $\mathrm{Re}\,\sigma<0$; this case, however, is more complicated.
The problem here is related with the error estimate, which is presented as the factor
$\big(1+\mathcal{O}\big(\tau^{\delta}\big)\big)$ that multiplies the leading term of asymptotics of the function
$u(\tau)$ (cf. equation~(45) in Theorem 3.4 of \cite{KitVar2004}
or equation (B.5) in Theorem B.1 of \cite{KitVar2023}), where the parameter $\delta>0$ is not specified. In this case,
the $\mathcal{O}\big(\tau^{1-\sigma}\big)$ term of the asymptotics, which contains the monodromy parameters, may be
vying with the $\mathcal{O}\big(\tau^{1+\delta}\big)$ term (depending on the values of $\delta$ and
$|\mathrm{Re}\,\sigma|$),
so that it is not at all apparent as to whether or not the leading term of asymptotics contains the monodromy
parameters or they are hidden in the correction term. In fact, the $\mathcal{O}\big(\tau^{1-\sigma}\big)$ term
continues to contribute to the leading term of asymptotics, and its parametrization via the monodromy data
given in \cite{KitVar2004,KitVar2023} is correct; however, this requires a separate justification.
There are three approaches for establishing this result: (i) to perform calculations similar to those in Section~5
of \cite{KitVar2004} for correspondingly modified assumptions on the coefficients of the
associated Fuchs-Garnier pair;\footnote{\label{foot:corrections:zapiski2024} See the comments and corrections to
this calculation given in Appendix A of \cite{KitVarZapiski2024}.} (ii) use of B\"acklund transformations; or
(iii) analytic continuation with
respect to the parameter $a$. The calculational scheme of the proof delineated in item (i) requires lengthy
calculations, together with a fairly large array of auxiliary constructions parallel to those carried out
in \cite{KitVar2004}. In the
present proof (see below), we refer to the methodology of item (ii), while the proof proposed in item (iii) is
outlined in Remark~\ref{rem:analytic Continuation-a}. As a matter of fact, we have already used a proof based
on B\"acklund transformations (cf. item (ii) above) in \cite{KitVarZapiski2023} to find asymptotics of $u(\tau)$
for $a=-\mi/2$: the underlying idea of this proof works without modification for the more general
situation $-1<\mathrm{Im}\,a<0$.

The proof consists of the following steps: first, note that if $0<\mathrm{Im}\,a<1$, then $-1<\mathrm{Im}(a-\mi)<0$;
second, one verifies that the monodromy data~\eqref{eqs:conditions-th3.1-limit1}, \eqref{eqs:spec-monodromy-th3.1-1}
are invariant under the action of B\"acklund transformations (cf. equations~\eqref{eq:data+} and \eqref{eq:data-});
and third, apply the B\"acklund transformations~\eqref{eq:u-} and \eqref{eq:varphi-} to the asymptotics~\eqref{eq:Asympt0-rhoEq+Ia:2u} and
\eqref{eq:Asympt0-rhoEq+Ia:2phi} of the functions $u(\tau)$ and $\varphi(\tau)$, respectively, for
$\mathrm{Im}\,a\in(0,1)$. These asymptotics are differentiable, so that, after elementary calculations and a
renotation for the monodromy variables, one arrives at the asymptotics~\eqref{eq:Asympt0-rhoEq+Ia:2u} and
\eqref{eq:Asympt0-rhoEq+Ia:2phi}. Note that, if $\mathrm{Im}\,a\in(-1,0)$, then $\mathrm{Re}\,\sigma<0$, so that
the denominator in the first term of the asymptotics~\eqref{eq:Asympt0-rhoEq+Ia:2u} can be neglected since its
contribution for small enough values of $\tau$ is smaller than the $\mathcal{O}(\tau^3)$ correction term.
The corrections for the asymptotics of the function $\me^{\mi\varphi(\tau)}$ are obtained by integrating
equation~\eqref{eq:varphi} with the help of the expansion~\eqref{app:eq:0-u-expansion}.

The proof for the asymptotics presented in item~\pmb{$(3)$} of the theorem for $n-1<-\mathrm{Im}\,a<n$ is also based
on the application of B\"acklund transformations. It is very similar to the proof of item \pmb{$(3)$} in
Theorem~\ref{th:Asympt0-rho-eq-1pn-ia2} below; therefore, the reader familiar with the proof for the asymptotics
stated in item~\pmb{$(3)$} of Theorem~\ref{th:Asympt0-rho-eq-1pn-ia2} should not expect to encounter any complications
with the proof of item~\pmb{$(3)$} of Theorem~\ref{th:Asympt0-rho-eq-1pn+ia2}.
\end{proof}
\begin{remark}\label{rem:analytic Continuation-a}
In this remark, we outline another proof for the asymptotics of the solution stated in item~\pmb{$(2)$} of
Theorem~\ref{th:Asympt0-rho-eq-1pn+ia2} for $\mathrm{Im}\,a\in(-1,0)$. This proof is based on the analytic
continuation of the solution with respect to the parameter $a$.

In the complex $a$-plane, we denote by $\mathcal{D}$ the strip $|\mathrm{Im}\,a|\leqslant1$ punctured at $a=0$;
in fact, we will work with the compactified strip
\begin{equation*}\label{eq:DepsilonDEF}
\mathcal{D}_{\epsilon}:=\left\{
a\in\mathbb{C}: |\mathrm{Im}\,a|\leqslant1-\epsilon,\,
|a|\geqslant\epsilon,\,
\epsilon\in(0,1/2)
\right\}.
\end{equation*}
Then, for any fixed parameter $s_1^{\infty}g_{21}^2$, we define, with the help of the convergent
series~\eqref{app:eq:0-u-expansion}, where $b_{1,-1}$ and $\sigma$ are given in equations~\eqref{eqs:sigma-b1-1},
the function $u_a(\tau)$.
The compactified domain $\mathcal{D}_{\epsilon}$ is necessary in order to guarantee that all functions
$u_a(\tau)$ (considered as functions of $\tau$) for $a\in\mathcal{D}_{\epsilon}$ have a non-empty common domain
of definition in some cut (along the negative real semi-axis) neighbourhood of $\tau=0$.
Note that the functions $u_a(\tau)$ are single-valued for $a\in\mathcal{D}_{\epsilon}$ because the coefficients of
the expansion~\eqref{app:eq:0-u-expansion} are single-valued in $\mathcal{D}_{\epsilon}$. We denote by
$u_a^{\pm}(\tau)$ the functions $u_a(\tau)$ for $\pm\mathrm{Im}\,a>0$.
As explained in the proof of Theorem~\ref{th:Asympt0-rho-eq-1pn+ia2}, the expansion~\eqref{app:eq:0-u-expansion} is
different for the functions $u_a^{\pm}(\tau)$, but, for $\mathrm{Im}\,a=0$, these expansions coincide. According to
the Principle of Analytic Continuation, the functions $u_a^{\pm}(\tau)$ are analytic continuations of one another;
however, for the function $u_a^{+}(\tau)$, we proved that it corresponds to the monodromy data
\eqref{eqs:conditions-th3.1-limit1}, \eqref{eqs:spec-monodromy-th3.1-1}; therefore, the same conclusion follows for
the function $u_a^{-}(\tau)$.
\hfill$\blacksquare$\end{remark}
\begin{theorem}\label{th:Asympt0-rho-eq-1pn-ia2}
Let $(u(\tau), \varphi(\tau))$ be a solution of the system~\eqref{eq:dp3}, \eqref{eq:varphi}
corresponding to the monodromy data $(a,s_{0}^{0},s_{0}^{\infty},s_{1}^{\infty},g_{11},g_{12},g_{21},g_{22})$.
Suppose that
\begin{equation}\label{eqs:conditions-th31p-limit1}
a\neq\mi k,\quad
k\in\mathbb{Z},\qquad
s_1^{\infty}=0,\quad
\mathrm{and}\quad
s_0^{\infty}\neq0;
\end{equation}
then, $g_{12}\in\mathbb{C}\setminus\{0\}$, and the remaining monodromy data are given by the following equations:
\begin{equation}\label{eqs:spec-monodromy-th3.1p-1}
s_0^0=2\mi\cosh(\pi a),\quad
g_{11}=\frac{s_0^{\infty}g_{12}^2\me^{-\pi a}-\mi}{2\sinh(\pi a) g_{12}},\quad
g_{21}=-\frac{\me^{\pi a}+\mi s_0^{\infty}g_{12}^2\me^{-2\pi a}}{2\sinh(\pi a)g_{12}},\quad
g_{22}=-\mi\me^{-\pi a}g_{12}.
\end{equation}
\begin{enumerate}
\item[\pmb{$(1)$}]
Assume that
$\mathrm{Im}\,a<0$, and define---uniquely---numbers $\varrho\in\mathbb{C}$, with $\mathrm{Re}\,\varrho\in(0,1)$, and
$n\in\mathbb{Z}_{\geqslant0}$ such that $\varrho=1+n-\mi a/2$, i.e.,
$\lfloor{\mathrm{Im}\,a/2}\rfloor=-n-1$ and $\mathrm{Re}\,\varrho=\{\mathrm{Im}\,a/2\}$, where $\lfloor\cdot\rfloor$ and
$\{\cdot\}$ denote, respectively, the floor and the fractional part of the real number; then,
\begin{align}
u(\tau)\underset{\tau\to0^{+}}{=}
&\frac{\varepsilon(1-2\varrho)^2\hat{w}_3\hat{w}_4\big(1+\mathcal{O}\big(\tau^{4\mathrm{Re}(\varrho)}\big)
+\mathcal{O}\big(\tau^{4(1-\mathrm{Re}(\varrho))}\big)\big)}
{\tau\big(\hat{w}_3\tau^{1-2\varrho}+\hat{w}_4\tau^{-1+2\varrho}\big)^2},\label{eq:Asympt0spec1u+}\\
\me^{\mi\varphi (\tau)} \underset{\tau \to 0^{+}}{=}
&\me^{\frac{3\pi\mi}{2}}\me^{\frac{\pi a}{2}}\frac{\hat{w}_3\hat{w}_4}{2\pi}\big(2\tau^2\big)^{\mi a}
\big(1+\mathcal{O}\big(\tau^{4\mathrm{Re}(\varrho)}\big)+\mathcal{O}\big(\tau^{4(1-\mathrm{Re}(\varrho))}\big)\big),\label{eq:Asympt0spec1varphi+}
\end{align}
where
\begin{align}
\hat{w}_{3}=&\left(\tfrac{1}{2}(\varepsilon b)\me^{-\pi\mi/2}\right)^{\frac{1}{2}-\varrho}\frac{2\pi}{n!}
\frac{\Gamma(2\varrho)}{\Gamma(2-2\varrho)}\frac{\me^{\pi\mi/4+\pi a/2}}{s_0^{\infty}g_{12}},
\label{eq:omega3}\\
\hat{w}_{4}=&\left(\tfrac{1}{2}(\varepsilon b)\me^{-\pi\mi/2}\right)^{\varrho-\frac{1}{2}}\me^{\pi\mi(1/4-\varrho)}
\frac{\Gamma(2-2\varrho)}{\Gamma(2\varrho)}
\Gamma(2\varrho-n-1)\,2\sinh(\pi a)\,g_{12};
\label{eq:omega4}
\end{align}
\item[\pmb{$(2)$}]
Assume that
$-1<\mathrm{Im}\,a<1$, and define
\begin{equation}\label{eqs:sigma-b11}
\sigma=-2\mi a,\quad
b_{1,1}=-\mi\left(\frac{\varepsilon b}{2}\right)^{1-\mi a}\frac{\pi\me^{-3\pi a/2}}{\sinh(\pi a)}
\frac{s_0^\infty g_{12}^2}{\big(\Gamma(1-\mi a)\big)^3};
\end{equation}
then,
\begin{align}
u(\tau)\underset{\tau\to0^+}{=}&\;
\frac{\varepsilon b_{1,1}\tau^{1+\sigma}}{\left(1+\frac{4 b_{1,1}\tau^{2+\sigma}}{(\sigma+2)^2}\right)^2}
-\frac{b\tau}{2a}+\mathcal{O}\big(\tau^{3+\sigma}\big)+\mathcal{O}\big(\tau^{3}\big),\label{eq:Asympt0-rhoEq-Ia:2u}\\
\me^{-\mi\varphi(\tau)}\underset{\tau\to0^+}{=}&\;
-\frac{\me^{\pi a}}{2\pi a g_{12}^2}
\left(\!\me^{-\frac{3\pi a}{2}}\Gamma(1+\mi a)s_0^{\infty}g_{12}^2
\left(2\tau^2\right)^{-\mi a}
-\mi\big(\Gamma(1-\mi a)\big)^2\left(\frac{4}{\varepsilon b}\right)^{-\mi a}\right)\nonumber\\
&\times\left(1+\mathcal{O}\big(\tau^{2}\big)+\mathcal{O}\big(\tau^{2-2\mi a}\big)\right).
\label{eq:Asympt0-rhoEQ-Ia:2phi}
\end{align}
\item[\pmb{$(3)$}]\label{Th4.2case3}
Assume that
$n-1<\mathrm{Im}\,a<n$, $n\in\mathbb{N}$, or $\mathrm{Im}\,a=(n-1)$ and $\mathrm{Re}\,a\neq0$.
Let $\sigma$ and $b_{1,1}$ be defined by equation~\eqref{eqs:sigma-b11}, in particular,
$2(n-1)\leqslant\mathrm{Re}\,\sigma<2n${\rm;} then,
\begin{align}
\varepsilon u(\tau)\underset{\tau\to0^+}{=}&\,
\sum_{k=1}^n b_{2k-1,0}\tau^{2k-1}+ b_{1,1}\tau^{1+\sigma}+\mathcal{O}\big(\tau^{2n+1}\big),
\label{eq:Asympt0-u-long-b11}\\
\me^{-\mi\varphi(\tau)}\underset{\tau\to0^+}{=}&\,
\frac{\mi\me^{\pi a}\big(\Gamma(1-\mi a)\big)^2}{2\pi ag_{12}^2}\left(\frac{\varepsilon b}{4}\right)^{\mi a}
\exp\left(\mi\left(P_n(\tau)+\frac{4a^2}{\varepsilon b}b_{1,1}
\frac{\tau^{\sigma}}{\sigma}+\mathcal{O}\big(\tau^{2n}\big)\right)\right),
\label{eq:Asympt0-phi-long-b11}
\end{align}
where the coefficients $b_{2k-1,0}$, $k=1,\ldots,n$, are defined in
Appendix~{\rm\ref{app:subsec:error-correction-terms}},\footnote{\label{foot:Th42AppA} See the expansion
\eqref{app:eq:0-u-expansion} and footnote~\ref{foot:sigmaEQ-4a2}.}
and the polynomials $P_n(\tau)$ are given in equations~\eqref{eq:Pn-def}--\eqref{eqs:MkN}.
%\begin{equation}\label{eq:Pn-th4.2}
%\begin{aligned}
%P_n(\tau)&=2a\sum_{k=1}^{n-1}\left(\frac{2a}{\varepsilon b}\right)^k\sum_{\{m_1,m_2,\ldots,m_{n-1}\}\in M_{k,n}}
%\frac{(m_1+m_2+\ldots+m_{n-1})!}{m_1!m_2!\cdots m_{n-1}!}\prod_{l=1}^{n-1}(b_{2l+1,0})^{m_l}\\
%&\times\frac{\tau^{2m_1+4m_2+\ldots+2(n-1)m_{n-1}}}{2m_1+4m_2+\ldots+2(n-1)m_{n-1}},
%\end{aligned}
%\end{equation}
%where the summation set $M_{k,n}$ consists of the numbers $m_k\in\mathbb{Z}_{\geqslant0}$, $k=1,2,\ldots,n-1$, that solve the
%following system,
%\begin{equation}\label{eqs:Mkn-th4.2}
%\begin{gathered}
%m_1+m_2+\ldots+m_{n-1}=k,\\
%m_1+2m_2+\ldots+(n-1)m_{n-1}\leqslant n-1.
%\end{gathered}
%\end{equation}
\end{enumerate}
\end{theorem}
\begin{proof}
The direct proof of this theorem is very similar to the proof of Theorem~\ref{th:Asympt0-rho-eq-1pn+ia2}; moreover,
there is a symmetry between the cases considered in these theorems, namely, the transformation
\begin{equation}\label{eqs:mondata-power-symmetry}
\begin{gathered}
a\to-a,\quad
s_0^0\to s_0^0,\quad
s_0^{\infty}\to-s_1^{\infty},\quad
s_1^{\infty}\to-s_0^{\infty},\\
g_{11}\to g_{12},\quad
g_{12}\to-g_{11},\quad
g_{21}\to g_{22},\quad
g_{22}\to-g_{21},
\end{gathered}
\end{equation}
maps the monodromy data satisfying the conditions~\eqref{eqs:conditions-th3.1-limit1} and
\eqref{eqs:spec-monodromy-th3.1-1} to the monodromy data satisfying the
conditions~\eqref{eqs:conditions-th31p-limit1} and  \eqref{eqs:spec-monodromy-th3.1p-1}.
The asymptotic results of Theorems~\ref{th:Asympt0-rho-eq-1pn+ia2} and \ref{th:Asympt0-rho-eq-1pn-ia2} can be
derived from one another with the help of the symmetry~\eqref{eqs:mondata-power-symmetry};
therefore, for the proof of the asymptotic results presented in items~\pmb{$(1)$} and \pmb{$(2)$} of this theorem,
we refer to the corresponding proof of Theorem~\ref{th:Asympt0-rho-eq-1pn+ia2}. Here, we prove the asymptotics given
in item~\pmb{$(3)$}, which also completes the proof of item~\pmb{$(3)$} of Theorem~\ref{th:Asympt0-rho-eq-1pn+ia2}.

There are two schemes for the proof of the asymptotics presented in item~\pmb{$(3)$}; the first proof is based on
B\"acklund transformations, whilst the second is based on analytic continuation: here, we consider a proof that
uses B\"acklund transformations, and the second proof is outlined in Remark~\ref{rem:altproof-b11} below.

Firstly, note that the action of the B\"acklund transformations on the monodromy manifold
(cf. equations~\eqref{eq:data+} and \eqref{eq:data-}) preserves the conditions~\eqref{eqs:conditions-th31p-limit1} and
\eqref{eqs:spec-monodromy-th3.1p-1}.

Secondly, the solutions corresponding to $n=1$ in item~\pmb{$(3)$} for $\mathrm{Im}\,a\in[0,1)$ coincide with the
solutions considered in item~\pmb{$(2)$} (cf. equations~\eqref{eq:Asympt0-rhoEq-Ia:2u} and
\eqref{eq:Asympt0-rhoEQ-Ia:2phi}) for $\mathrm{Im}\,a\in[0,1)$.
For the solutions corresponding to $\mathrm{Im}\,a\in[0,1)$, $\sigma>0$, so that  we can expand the denominator of
the first term in the asymptotics~\eqref{eq:Asympt0-rhoEq-Ia:2u} into a Taylor series with respect to
$\tau^{2+\sigma}$, and deduce that the correction provided by this expansion is smaller than
$\mathcal{O}(\tau^{3+\sigma})$; thus, the contribution to the error resulting from the denominator can be neglected,
and one arrives at the asymptotics~\eqref{eq:Asympt0-u-long-b11} for $n=1$.
To verify the asymptotics \eqref{eq:Asympt0-phi-long-b11}, one has to expand the exponential function with
$P_1(\tau)=0$ and one non-trivial explicit term, plus the corrections, and then multiply this expansion by the first
coefficient; then, after a straightforward calculation, one finds that the formula obtained coincides with the
asymptotics~\eqref{eq:Asympt0-rhoEQ-Ia:2phi}.

Having in mind an induction-based proof for item~\pmb{$(3)$}, denote, for $n=1$,
$(u_0(\tau), \me^{\mi\varphi_0(\tau)})$ the corresponding pair of
functions considered in the previous paragraph, and conclude that the base of the mathematical induction is
established.

To make the inductive step, we define the sequence of functions $u_k(\tau)$ and $\varphi_k(\tau)$ for $k\in\mathbb{N}$
by successively applying $k$ B\"acklund transformations~\eqref{eq:u+} and \eqref{eq:varphi+}, respectively, to the
pair of functions $u(\tau)=u_0(\tau)$ and $\varphi(\tau)=\varphi_0(\tau)$. In this context, we define $a_0:=a$,
so that $a_k$, the parameter of formal monodromy corresponding to the functions $u_k(\tau)$ and $\varphi_k(\tau)$,
satisfies the recurrence relation $a_k=a_{k-1}+\mi$ (cf. equation~\eqref{eq:data+}). The corresponding parameter
$\sigma_k$ (cf. equation~\eqref{eqs:sigma-b11}) is not bounded and varies with $k$, namely, $\sigma_k=\sigma_{k-1}+2$,
with $\sigma_0=\sigma\geqslant0$.

Our induction hypothesis is that all the coefficients of the terms $\tau^{k-l\sigma_{n-1}}$, $k,l\in\mathbb{N}$, in
the expansion~\eqref{app:eq:0-u-expansion} for $u_{n-1}(\tau)$ vanish, and that the asymptotics stated in
item~~\pmb{$(3)$} are valid for the functions $u_{n-1}(\tau)$ and $\me^{\mi\varphi_{n-1}(\tau)}$.

To take the inductive step, consider the B\"acklund transformations~\eqref{eq:u+} and \eqref{eq:varphi+}, in which
we put $u_+(\tau)=u_n(\tau)$, $u(\tau)=u_{n-1}(\tau)$ and $\varphi_+(\tau)=\varphi_n(\tau)$,
$\varphi(\tau)=\varphi_{n-1}(\tau)$, $n\in\mathbb{N}$, respectively, and observe that these transformations are
covariant mappings of the expansion~\eqref{app:eq:0-u-expansion}, with $a=a_{n-1}$ and $\sigma=\sigma_{n-1}$, to an
expansion of the same form, but with $a=a_n$ and $\sigma=\sigma_n$. Actually,
substituting the expansion~\eqref{app:eq:0-u-expansion} for the function $u_{n-1}(\tau)$ into the formula for the
B\"acklund transformation \eqref{eq:u+} and re-expanding this expression as $\tau\to0$, one sees that the resulting
expansion contains terms with positive powers of $\tau$ and $\tau^{\sigma}$, and the largest term of the expansion
is $b_{1,0}(n)\tau$, where the coefficient $b_{1,0}(n)=b\,(a_{n-1}-\mi)b_{3,0}(n-1)/(4b_{1,0}^2(n-1))$. Here and
below, we use the notation $b_{i,k}(m)$ for $m=n-1$ or $m=n$ to denote the terms of the
expansion~\eqref{app:eq:0-u-expansion} for the functions $u_{m}(\tau)$.
Using the second equation in the list~\eqref{app:eq:sigma-b0-b1pm1} and equation~\eqref{app:eq:b30} with $a=a_{n-1}$
and $\sigma=\sigma_{n-1}=-2\mi a_{n-1}$, one proves that $b_{1,0}(n)=2a_n b/\sigma_n^2$, which coincides with the
second equation in the list~\eqref{app:eq:sigma-b0-b1pm1} for $a=a_n$ and $\sigma=\sigma_n=-2\mi a_n$. Now, we
have to check that the coefficient of the term $\tau^{1+\sigma_{n-1}}$ vanishes. This fact is equivalent to the
relation
\begin{equation*}\label{eq:vanishing-b31}
b_{3,1}(n-1)=\frac{2\,(2\mi a_{n-1}+2)\,b_{1,1}(n-1)\,b_{3,0}(n-1)}{(2\mi a_{n-1}+2+\sigma_{n-1})\,b_{1,0}(n-1)},
\end{equation*}
which can be verified with the help of equations~\eqref{app:eq:sigma-b0-b1pm1}, \eqref{app:eq:b2k+1+k+1+k}
(for $k=1$), and \eqref{app:eq:b30}. Thus, the largest term in powers of $\tau$ that contains the parameter
$\sigma_{n-1}$ is of order $\tau^{3+\sigma_{n-1}}$: we denote this term as $b_{1,1}(n)\,\tau^{1+\sigma_n}$. The two terms
$b_{1,0}(n)\tau$ and $b_{1,1}(n)\tau^{1+\sigma_n}$, together with the fact that the
expansion~\eqref{app:eq:0-u-expansion} contains only
$\mathcal{O}\big(\tau^{k+l\sigma_n}\big)$, $k\in\mathbb{N}$, $l\in\mathbb{Z}_{\geqslant0}$, terms, completely define
this expansion via substitution into the degenerate third Painlev\'e equation~\eqref{eq:dp3} with $a=a_n$.

As long as the form of the small-$\tau$ expansion for the function $u_{n}(\tau)$ is established, we have to prove that
$b_{1,1}(n)$ is given by equation~\eqref{eqs:sigma-b11} with monodromy data corresponding to the $n$th iteration
of $u_0(\tau)$ by the B\"acklund transformations. We can certainly continue to use equation~\eqref{eq:u+} and study
the largest powers of $\tau$ containing $\sigma_{n-1}$, but to do so, it is convenient to multiply both sides by
$u_{n-1}^2(\tau)$; then, we find that the $\mathcal{O}\big(\tau^{1+\sigma_{n-1}}\big)$ terms cancel identically.
The $\mathcal{O}\big(\tau^{3+\sigma_{n-1}}\big)$ terms also cancel, but to see this, one has to prove the relation
\begin{equation*}\label{eq:relation:b31:b11:b10}
-\frac{\mi b}{8} b_{3,1}(n-1)=b_{1,0}(n-1)\,b_{1,0}(n)\,b_{1,1}(n-1),
\qquad
\varepsilon=1,
\end{equation*}
which can be done with the help of equations~\eqref{app:eq:b2k+1+k+1+k} for $k=1$, and the second equation in the
list~\eqref{app:eq:sigma-b0-b1pm1} for $a=a_{n-1}$ and $a=a_n$. In order to actually establish the induction
hypothesis, one has to equate the $\mathcal{O}\big(\tau^{5+\sigma_{n-1}}\big)$ terms to zero, which is possible
with the help of equation~\eqref{app:eq:b51} for $b_{5,1}(n-1)$. There is, however, an easier way to find the
corresponding recurrence relation, namely, to use the inverse B\"acklund transformation or either one of
equations~\eqref{eq:dif-discrete-v-n} or \eqref{eq:discrete-v-n}.

Consider equation~\eqref{eq:discrete-v-n}, say, and recall that $v_n(\tau)=u_n(\tau)/\tau$, $n\in\mathbb{N}$;
substitute into this equation the expansion~\eqref{app:eq:0-u-expansion} for $u_n(\tau)$ and equate coefficients
of like powers of $\tau$ on both sides of the resulting equation.
On the right-hand side of this equation, there is a term of the order $\tau^{\sigma_{n}-2}=\tau^{\sigma_{n-1}}$,
with coefficient equal to $\tfrac{\varepsilon b}{2}a_n\,b_{1,1}(n)$, while on the left-hand side of this equation,
there is also a term of the order $\tau^{\sigma_{n-1}}$, but with coefficient equal to
$b_{1,0}^2(n)\,b_{1,1}(n-1)$; so, equating these coefficients, we arrive at the following recurrence relation:
\begin{equation}\label{eq:recurrence-b11}
b_{1,1}(n)=\frac{2b_{1,0}^2(n)}{\varepsilon b\,a_n}b_{1,1}(n-1)=\frac{\varepsilon b}{2a_n^3}b_{1,1}(n-1).
\end{equation}
Now, according to the induction hypothesis, $b_{1,1}(n-1)$ is given by equation~\eqref{eqs:sigma-b11} with $a=a_{n-1}$.
Equation~\eqref{eq:recurrence-b11} implies that $b_{1,1}(n)$ is given by the same equation~\eqref{eqs:sigma-b11} but
with $a=a_n$. In verifying this fact, it is imperative to take into account the change of the monodromy data under the
B\"acklund transformations, that is, $s_0^{\infty}\to s_0^{\infty}$ and $g_{12}^2\rightarrow -g_{12}^2$.

The corresponding expansion for $\me^{\mi\varphi(\tau)}$ (cf. equation~\eqref{eq:Asympt0-rhoEQ-Ia:2phi}) is obtained
with the help of equation~\eqref{eq:varphi}, and the multiplicative constant is verified via
equation~\eqref{eq:varphi+}.
\end{proof}
\begin{remark}\label{rem:altproof-b11}
It is instructive to provide an alternative proof for the expansion~\eqref{eq:Asympt0-rhoEq-Ia:2u} without having to
resort to B\"acklund transformations. The first observation is that, for the monodromy
data~\eqref{eqs:spec-monodromy-th3.1p-1}, the coefficient $b_{1,-1}=0$ (see
Appendix~\ref{app:subsec:error-correction-terms}, equation~\eqref{app:eqs:betas}, and the text subsequent to
equation~\eqref{app:eqs:betas}).
The second observation is that, in this case, the expansion~\eqref{app:eq:0-u-expansion} contains powers of $\tau$
with exponents having positive real part, which, after a rearrangement of terms, manifests as the asymptotic
expansion valid for all $\sigma\in\mathbb{C}$ with $\mathrm{Re}\,\sigma\geqslant0$
(the restriction $|\mathrm{Re}\,\sigma|<2$ is not imposed).
The expansion is convergent in some neighbourhood of $\tau=0$ with a branch cut along the negative real semi-axis;
in particular, these solutions do not have poles in some neighbourhood of the origin.
The expansion \eqref{eq:Asympt0-rhoEq-Ia:2u} defines, therefore, the analytic continuation of the
solution for all values of the monodromy parameters for which the coefficients of the expansion are defined.
Thus, the coefficient $b_{1,1}$ is given by the same formula for all $a\neq\mi k$, $k\in\mathbb{N}$.
\hfill$\blacksquare$\end{remark}
\begin{remark}\label{rem:Pn1-4}
Since the definition of the polynomials $P_n=P_n(\tau)$ appearing in the asymptotics of the function
$\varphi(\tau)$ (cf. equations~\eqref{eq:Asympt0-phi-long-b1-1} and \eqref{eq:Asympt0-phi-long-b11}) are cumbersome
(cf. equations~\eqref{eq:Pn-def}--\eqref{eqs:MkN}), we present explicit expressions for the first four polynomials:
\begin{gather*}\label{eqs:P1-P4}
P_1=0,\quad
P_2=\frac{4a^2}{\varepsilon b}b_{3,0}\frac{\tau^2}{2},\quad
P_3=P_2+\frac{4a^2}{\varepsilon b}\left(b_{5,0}+\frac{2a}{\varepsilon b}b_{3,0}^2\right)\frac{\tau^4}{4},\\
P_4=P_3+\frac{4a^2}{\varepsilon b}\left(b_{7,0}+\frac{4a}{\varepsilon b}b_{3,0}b_{5,0}
+\frac{4a^2}{b^2}b_{3,0}^2\right)\frac{\tau^6}{6}.
\end{gather*}
Note that $\deg\,P_n(\tau)=2(n-1)$.
\hfill$\blacksquare$\end{remark}
\begin{remark}\label{rem:uas-pq1q2}
Assume that the reader has an asymptotic expansion of the type
\begin{equation} \label{eq:uas-pq1q2}
u(\tau) \underset{\tau \to 0^+}{\sim} \frac{p}{\tau\big(q_1\tau^{\alpha}+q_2\tau^{-\alpha}\big)^2}, \quad
\alpha,p,q_1,q_2 \in \mathbb{C}\setminus \{0\}, \quad |\mathrm{Re}\,\alpha|<1,
\end{equation}
and would like to get the monodromy parametrization for the asymptotics~\eqref{eq:uas-pq1q2}.
How does one distinguish between the parametrizations given in Theorems~\ref{th:altB1asympt0m}
(\ref{th:alt2B1asympt0p}\hspace{-0pt}${}^{\mathbf\prime}$),
\ref{th:Asympt0-rho-eq-1pn+ia2}, and \ref{th:Asympt0-rho-eq-1pn-ia2}?

In order to choose which of these theorems is suitable for parametrizing the asymptotics~\eqref{eq:uas-pq1q2}
via the monodromy data, the reader should complete the following steps:
(1) set $\alpha=1-2\varrho_1$ and $-\alpha=1-2\varrho_2$ to obtain two possible values, $\varrho_1$ and $\varrho_2$,
for the parameter $\varrho$ so that $\varrho_1+\varrho_2=1$, $0<\mathrm{Re}\,\varrho_k<1$, $k=1,2$, and
$\varrho_1\neq\varrho_2 \neq 1/2$;
(2) normalize the asymptotics~\eqref{eq:uas-pq1q2}, that is, multiply both the numerator and the denominator of
the asymptotics~\eqref{eq:uas-pq1q2} by $\lambda^2$ and choose $\lambda^2$ such that $\tilde{q}_1\tilde{q}_2=1$,
where $\tilde{q}_k=q_k\lambda$, $k=1,2$;
(3) solve the equation $(1-2\varrho)^2=\varepsilon p\lambda^2$, and denote the roots as $\varrho_1$ and $\varrho_2$,
where, clearly, $\varrho_1+\varrho_2=1$, $\varrho_1\neq\varrho_2\neq1/2$, and, if the asymptotics~\eqref{eq:uas-pq1q2}
is correct, then the roots obtained in steps (1) and (2) coincide;
(4) find $a$ by transforming the degenerate third Painlev\'e equation under investigation into the form of
equation~\eqref{eq:dp3};
(5) if $\mathrm{Im}\, a=0$, then, as follows from equation~\eqref{eq:rho-general}, $s_0^{\infty}s_1^{\infty} \neq 0$,
so that one has to use, for any one of the roots $\varrho_k$, $k=1,2$,\footref{foot:reflection-varrho}
either Theorem~\ref{th:altB1asympt0m} or Theorem~\ref{th:alt2B1asympt0p}\hspace{-0pt}${}^{\mathbf\prime}$;
(6) if $\mathrm{Im}\,a>0$, then, check whether any of the roots $\varrho_1$ or $\varrho_2$ satisfy the conditions
\begin{equation} \label{eqs:varrho:Im-positive}
\mathrm{Re}\,\varrho=1-\left\{\mathrm{Im}\,a/2\right\},\qquad
\mathrm{Im}\,\varrho=\mathrm{Re}\,a/2,
\end{equation}
and, in the event that none of the roots satisfy the relations~\eqref{eqs:varrho:Im-positive}, then, again,
for either one of the roots, the parametrizations of Theorems~\ref{th:altB1asympt0m} or
\ref{th:alt2B1asympt0p}\hspace{-0pt}${}^{\mathbf\prime}$ are valid, whereas if one of
the roots does, in fact, satisfy the relations~\eqref{eqs:varrho:Im-positive}, then all the conditions enumerated
in item {\bf(1)} of Theorem~\ref{th:Asympt0-rho-eq-1pn+ia2} are satisfied and one is in a position to use, for
this root, the monodromy parametrization for the asymptotics~\eqref{eq:uas-pq1q2} given in item {\bf(1)} of
Theorem~\ref{th:Asympt0-rho-eq-1pn+ia2} with the other root being inapplicable for the construction of
the asymptotics; and
(7) if $\mathrm{Im}\,a<0$, then, check whether any of the roots $\varrho_1$ or $\varrho_2$ satisfy the conditions
\begin{equation} \label{eqs:varrho:Im-negative}
\mathrm{Re}\,\varrho=\left\{\mathrm{Im}\,a/2\right\}, \qquad
\mathrm{Im}\,\varrho=-\mathrm{Re}\,a/2,
\end{equation}
and, in the event that none of the roots satisfy the relations \eqref{eqs:varrho:Im-negative}, then, again, for
either one of the roots, the parametrizations of Theorems~\ref{th:altB1asympt0m} or
\ref{th:alt2B1asympt0p}\hspace{-0pt}${}^{\mathbf\prime}$ are valid, whereas if one of the roots does, in fact,
satisfy the relations~\eqref{eqs:varrho:Im-negative}, then all the conditions enumerated in item {\bf(1)} of
Theorem~\ref{th:Asympt0-rho-eq-1pn-ia2} are satisfied and one is in a position to use, for this root,
the monodromy parametrization for the asymptotics~\eqref{eq:uas-pq1q2} given in item {\bf(1)} of
Theorem~\ref{th:Asympt0-rho-eq-1pn-ia2} with the other root being inapplicable for the construction of the
asymptotics.

We conclude this remark with a brief explanation of how one should proceed in order to find the monodromy
parametrization of the asymptotics~\eqref{eq:uas-pq1q2} once the identification of the proper theorem has been made.
Consider, say, Theorem~\ref{th:altB1asympt0m}. Depending on the root which is chosen in
the procedure delineated above, one finds a relation of the form $\tilde{q}_1^2=w_1/w_2$ for $\varrho_1$ and
$\tilde{q}_2^2=w_2/w_1$ for $\varrho_2$. Either one of these equations have, depending on the values of the parameters
$\varrho_k$ and $\tilde{q}_k$, $k=1,2$, three types of solutions:
(1) $g_{11}=Cg_{21} \neq 0$, where $C=C(\varrho_k,\tilde{q}_k)$;
(2) $g_{11}=0$ and $g_{21}\in\mathbb{C}\setminus\{0\}$; and
(3) $g_{21}=0$ and $g_{11}\in\mathbb{C}\setminus\{0\}$.
For the sake of example, consider case (1), and recall equation~\eqref{eq:monodromy:main}. The Stokes multiplier
$s_0^0$ that appears in equation~\eqref{eq:monodromy:main} can be calculated via equation~\eqref{eq:rho-general}
provided the root $\varrho_k$ is chosen. The first and second terms of equation~\eqref{eq:monodromy:main}
can be re-written, respectively, as $g_{21}g_{22}=\tfrac{1}{C}g_{11}g_{22}$
and $g_{11}g_{12}=Cg_{21}g_{12}=C(g_{11}g_{22}-1)$, where, in the derivation of the last equation, we have taken
equation~\eqref{eq:monodromy:detG} into account. Consequently, one obtains a linear equation for the determination
of $g_{11}g_{22}$; the latter product is a key parameter defining the asymptotics at the point at infinity
(see Appendix C of \cite{KitVar2023}). Another parameter that is necessary for constructing the large-$\tau$
asymptotics of $u(\tau)$ is $g_{11}g_{12}$ (see equations (C.29) and (C.31) in \cite{KitVar2023}), which has been
addressed above.
\hfill$\blacksquare$\end{remark}
%%%%%%%%%%%%%%%%%%%%%%%%%%%%%%%%%%%%%%%%%%%%%%%%%%%%%%%%%%%%%%%%%%%%%%%%%%%%%%%%%%%%%%%%%%%%%%%%%%%%%%%%%%%%%%%%%%%%%%%%%%%%%%%
%%%%%%%%%%%%%%%%%%%%%%%%%%%%%%%%%%%%%%%%%%%%%%%%%%%%%%%%%%%%%%%%%%%%%%%%%%%%%%%%%%%%%%%%%%%%%%%%%%%%%%%%%%%%%%%%%%%%%%%%%%%%%%%
\section{Special Solutions with Logarithmic Behaviour as $\tau\to0$: $\rho=0$ and $\varrho=1/2$}\label{sec:logarithm}
The values $\rho=0$ and $\varrho=1/2$ for the respective branching parameters were excluded from
the formulations of the theorems in Sections~\ref{sec:general}--\ref{sec:poles}
because these, and only these, values correspond to solutions of equation~\eqref{eq:dp3}
that exhibit logarithmic behaviour. All solutions of equation~\eqref{eq:dp3} for $a\neq\mi k$, $k\in\mathbb{Z}$,
possessing logarithmic behaviour as $\tau\to0$ are members of two ($\rho=0$ and $\varrho=1/2$) one-parameter families
of solutions:
the asymptotics for the family corresponding to $\rho=0$, with the restriction $|\mathrm{Im}\,a|<1$,
was obtained in \cite{KitVar2004}.
In the recent paper~\cite{KitVarZapiski2024}, we: (i) rewrote the result of \cite{KitVar2004} in more convenient
form (in terms of simplified notation); (ii) obtained the corresponding asymptotics for the function
$\me^{\mi\varphi(\tau)}$; and (iii) distinguished the special case $a=0$. Here, this result is extended to all
$a\notin\mi2\mathbb{Z}$, and a refined estimate for the error-correction term is obtained.
\begin{theorem} \label{th:ln-regular-a-non0}
Let $(u(\tau),\varphi(\tau))$ be a solution of the system \eqref{eq:dp3}, \eqref{eq:varphi} corresponding to the
monodromy data $(a,s^{0}_{0},s^{\infty}_{0},s^{\infty}_{1},g_{11},g_{12},g_{21},g_{22})$. Suppose that
\begin{equation}\label{eqs:ln-as00-conditions}
a\in\mathbb{C}\setminus\mi2\mathbb{Z},\qquad
s_{0}^{0}=2\mi;
\end{equation}
then,
\begin{gather}
(g_{11}-\mi g_{21})(g_{12}-\mi g_{22})=\mi(1-\me^{-\pi a})\neq0,\label{eq:nonzero-gik-relation}\\
(g_{11}-\mi g_{21})^2=\mi s_0^{\infty}\me^{-\pi a}\neq0,\label{eq:nonzero-s0infty-rho0}\nonumber\\
(g_{12}-\mi g_{22})^2=-\mi s_1^{\infty}\me^{\pi a}\neq0.\label{eq:nonzero-s1infty-rho0}\nonumber
\end{gather}
Define
\begin{equation} \label{eq:def-c-log}
c:=4\gamma+\frac{\mi}{a}+\psi(-\mi a/2)-\frac{\pi\mi}{2}+\frac{\pi\mi\,(g_{12}+\mi g_{22})}{g_{12}-\mi g_{22}}+
\ln(\varepsilon b/2);
\end{equation}
then,
\begin{align}
u(\tau)\underset{\tau\to0^{+}}{=}&-ab\tau\left(\ln\tau+\frac{1}{2}(c-\mi/a)\right)
\left(\ln\tau+\frac{1}{2}(c+\mi/a)\right)\left(1+\mathcal{O}\big(\tau^2\ln^2\tau\big)\right)
\label{eq:ln-regular-u-factor} \\
\underset{\tau \to 0^{+}}{=}&-ab\tau\!\left(\ln^2\tau+c\ln\tau+\frac{1}{4}\!\left(c^{2}+\frac{1}{a^{2}}\right)\!\right)\!
\left(1+\mathcal{O}\big(\tau^2\ln^2\tau\big)\right),
\label{eq:ln-regular-u-combined} \\
\me^{\mi\varphi(\tau)}\underset{\tau\to 0^{+}}{=}&\frac{\me^{\frac{\pi}{2}(a+\mi)}}{\pi a}(g_{12} \! - \! \mi g_{22})^{2}
\left(\Gamma \big(1 \! - \! \tfrac{\mi a}{2} \big) \right)^{2}(2 \tau^{2})^{\mi a}\! \left(\frac{\ln \tau \! + \! \frac{1}{2}(c \! + \!
\mi/a)}{\ln \tau \! + \! \frac{1}{2}(c \! - \! \mi/a)} \right) \! \left(1 \! + \! \mathcal{O} \big(\tau^2 \big) \right),
\label{eq:ln-regular-varphi}
\end{align}
where $\psi (z) \! := \! \tfrac{\md \ln \Gamma (z)}{\md z}$ is the digamma function, and
$\gamma \! = \! -\psi (1) \! = \!0.577215664901532860606512 \dotsc$ is the Euler-Mascheroni constant.
\end{theorem}
\begin{proof}
Substituting $s_0^0=2\mi$ into equation~\eqref{eq:monodromy:main} and using equation~\eqref{eq:monodromy:detG},
we obtain equation~\eqref{eq:nonzero-gik-relation}; then, the first condition of \eqref{eqs:ln-as00-conditions}
implies the inequality in \eqref{eq:nonzero-gik-relation}. The two subsequent equalities/inequalities that include
the Stokes multipliers $s_0^{\infty}$ and $s_1^{\infty}$ are derived similarly, but, instead of using
equation~\eqref{eq:monodromy:main}, one makes use of equations~\eqref{eq:monodromy:s0}
and \eqref{eq:monodromy:s1}, respectively. Thus, the parameter $c$ is correctly defined by
equation~\eqref{eq:def-c-log}.

For $|\mathrm{Im}\,a|<1$ and $a\neq0$, the asymptotics~\eqref{eq:ln-regular-u-combined} and
\eqref{eq:ln-regular-varphi} are proved in \cite{KitVarZapiski2024}
(see Theorem 3.1 and Remark 3.1 in \cite{KitVarZapiski2024}); however, in the asymptotic
formulae~\eqref{eq:ln-regular-u-combined} and \eqref{eq:ln-regular-varphi}, more precise error estimates for
the correction terms are obtained by employing the complete local asymptotic expansion given in
Appendix~\ref{app:eq:0-u-expansionLog} (see Remark~\ref{rem:errorTh61}).

The restriction for the parameter of formal monodromy $a$ remains, however: the extension of the
asymptotics~\eqref{eq:ln-regular-u-factor}--\eqref{eq:ln-regular-varphi} to all values of
$a\in\mathbb{C}\setminus\mi2\mathbb{Z}$ is done below (see Lemma~\ref{th:bootstrap-log}) using
a bootstrap-type argument as in Section~\ref{sec:general}
(cf. Theorem~\ref{th:restriction-on-Ima-removed}), but, in the present case, with the
help of Theorems~\ref{th:Asympt0-ln-rho12-} and \ref{th:Asympt0-ln-rho12+}${}^{\mathbf\prime}$,
which will be proved below. For the proofs of these theorems, we use the
asymptotics~\eqref{eq:ln-regular-u-factor}--\eqref{eq:ln-regular-varphi}, with the updated error-correction term(s),
but in which the parameter $a$ is still subject to the restrictions $|\mathrm{Im}\,a|<1$ and $a\neq0$.
\end{proof}
\begin{theorem}\label{th:Asympt0-ln-rho12-}
Let $(u(\tau), \varphi(\tau))$ be a solution of the system~\eqref{eq:dp3}, \eqref{eq:varphi} corresponding to the
monodromy data $(a,s_{0}^{0},s_{0}^{\infty},s_{1}^{\infty},g_{11},g_{12},g_{21},g_{22})$.
Suppose that
\begin{equation} \label{eqs:Asympt0-rho12:mondataConditions}
a\in\mathbb{C}\setminus\{\mi(2m-1),m\in\mathbb{Z}\},\qquad
s_0^0=-2\mi;
%\lvert\mathrm{Im}(a)\rvert<1,\qquad
%\rho=1/2,\qquad
%g_{11}g_{22}\neq0,
\end{equation}
%where
%\begin{equation} \label{eq:Asympt0-rho12:rho-s}
%\cos(2\pi\rho):=-\frac{\mi s_{0}^{0}}{2}=\cosh(\pi a)+\frac{1}{2}s_{0}^{\infty}s_{1}^{\infty}\me^{\pi a}.
%\end{equation}
%In this case $s_0^0=-2\mi$ and the parameter $s:=1+\mi s_0^0=3$ {\rm(}see Sections $4$ and $5$ of {\rm\cite{KitVar2023})}.
then,
\begin{gather}
(g_{11}+\mi g_{21})(g_{12}+\mi g_{22})=-\mi(1+\me^{-\pi a})\neq0,\label{eq:nonzero-gik-relation-rho12}\\
(g_{11}+\mi g_{21})^2=\mi s_0^{\infty}\me^{-\pi a}\neq0,\label{eq:nonzero-s0infty-rho12}\nonumber\\
(g_{12}+\mi g_{22})^2=-\mi s_1^{\infty}\me^{\pi a}\neq0.\label{eq:nonzero-s1infty-rho12}\nonumber
\end{gather}
Define
\begin{equation}\label{eq:def-cm-log}
c_-:=4\gamma+\psi(1/2+\mi a/2)+\frac{\pi\mi}{2}+\frac{\pi\mi(g_{11}-\mi g_{21})}{g_{11}+\mi g_{21}}+
\ln(\varepsilon b/2);
\end{equation}
then,
\begin{align}
u(\tau)\underset{\tau\to0^{+}}{=}&-\frac{\varepsilon\left(1+\mathcal{O}\big(\tau^2\ln^2\tau\big)\right)}
{4\tau\big(\ln\tau+c_-/2\big)^2}\label{eq:u-asympt0-rho12-},\\
%u(\tau)\underset{\tau\to0^{+}}{=}&-\frac{\varepsilon\left(1+\mathcal{O}\big(((\tau\ln\tau)^2\big)\right)}
%{\tau\big(\ln\left(\varepsilon b\tau^2/2\right)+\psi\left(\frac12+\frac{\mi a}{2}\right)
%+4\gamma+\frac{\pi\mi}{2}+\pi\mi(g_{11}-\mi g_{21})/(g_{11}+\mi g_{21})\big)^2},\label{eq:u-asympt0-rho12-}\\
%u(\tau)\underset{\tau\to0^{+}}{=}\frac{\varepsilon\big(1+\me^{-\pi a}\big)\,\mathfrak{f}_{1}(\tau)}
%{\tau\,\mathfrak{f}_{3}(\tau)\,\big(\mathfrak{f}_{2}(\tau)\big)^{2}}\left(1+\mathcal{O}\big((\tau\ln\tau)^2\big)\right),
%\label{eq:u-asympt0-rho12}\\
\me^{\mi\varphi(\tau)}\underset{\tau\to0^{+}}{=}&-\frac{2\pi\me^{-\pi a/2}\big(2\tau^2\big)^{\mi a}
\exp\left(-2\mi\varepsilon b\tau^2\left(\left(\ln\tau+\frac{c_-}{2}-\frac12\right)^2+\frac14\right)
+\mathcal{O}\big(\tau^4\ln^4\tau\big)\right)}{\left(\Gamma\left(1/2+\mi a/2\right)(g_{11}+\mi g_{21})\right)^2}.
\label{eq:varphi-asympt0-rho12-}
%\me^{\mi\varphi(\tau)}\underset{\tau\to0^{+}}{=}\frac{\pi \me^{3\pi\mi/2}
%2^{\mi a} \tau^{2\mi a}}{\big(\Gamma (\tfrac{1}{2} \! + \! \tfrac{\mi a}{2}) \big)^{2} \cosh
%(\pi a/2)} \frac{\mathfrak{f}_{3}(\tau)}{\mathfrak{f}_{1}(\tau)} \left(1+\mathcal{O}\big((\tau\ln\tau)^{2}\big)\right),
%\label{eq:varphi-asympy0-rho12}
\end{align}
%where $\psi(z)=\frac{\md\ln\Gamma(z)}{\md z}$ is the digamma function and $\gamma=-\psi(1)=0.5772156649\ldots$ is
%the Euler-Mascheroni constant.
%where for $k=1,2$:
%\begin{align}
%\mathfrak{f}_{2k-1}(\tau)&=2w_{k}(1)+(1-\mi a )\mathfrak{f}_{2k}(\tau),\label{eq:f1f3}\\
%\mathfrak{f}_{2k}(\tau)&=w_{k}(1)\ln(\tau^2)+\pi\mi w_{k}(-1)+w_{k}(1)\Phi_{k}(a),\label{eq:f2f4}
%\end{align}
%with
%\begin{equation} \label{eq:Asympt0-rho12-wkPhik}
%\begin{aligned}
%w_{k}(l)&:=g_{1k}\me^{-\mi\pi/4}+l g_{2k}\me^{\mi\pi/4}, \quad
%l=\pm 1,\\
%\Phi_{k}(a)&:=-4\psi(1)+\psi\left(\tfrac{1}{2}-(-1)^k\tfrac{\mi a}{2}\right)+\ln\left(\tfrac{\varepsilon b}{2}\right)
%-(-1)^k\tfrac{\pi\mi}{2}.
%\end{aligned}
%\end{equation}
\end{theorem}
\begin{proof}
Substituting $s_0^0=-2\mi$ into equation~\eqref{eq:monodromy:main} and using equation~\eqref{eq:monodromy:detG},
we obtain equation~\eqref{eq:nonzero-gik-relation-rho12}; then, the first
condition of~\eqref{eqs:Asympt0-rho12:mondataConditions} implies the inequality
in~\eqref{eq:nonzero-gik-relation-rho12}.
The derivation of the two subsequent equalities/inequalities that include the parameters $s_0^{\infty}$ and
$s_1^{\infty}$ is done in a similar way, but, instead of using equation~\eqref{eq:monodromy:main}, one
employs equations~\eqref{eq:monodromy:s0} and \eqref{eq:monodromy:s1}, respectively.
Thus, the parameter $c_-$ is correctly defined by equation~\eqref{eq:def-cm-log}.

The asymptotics~\eqref{eq:u-asympt0-rho12-} and \eqref{eq:varphi-asympt0-rho12-} are obtained by applying
the B\"acklund transformations \eqref{eq:u-} and \eqref{eq:varphi-} to the
asymptotics~\eqref{eq:ln-regular-u-factor} and \eqref{eq:ln-regular-varphi},
respectively.\footnote{\label{foot:alt-derivation-asympt-varphi} Alternatively, the $\tau$-dependent
part of the asymptotics~\eqref{eq:varphi-asympt0-rho12-} can be obtained by integrating
equation~\eqref{eq:varphi}, and leads to the exponential form of the asymptotics presented in
equations~\eqref{eq:varphi-asympt0-rho12-} and \eqref{eq:varphi-asympt0-rho12+}.}
Recall that, thus far, the latter asymptotics are proved for $|\mathrm{Im}\,a|<1$;
therefore, at this stage of the proof, the asymptotics~\eqref{eq:u-asympt0-rho12-} and
\eqref{eq:varphi-asympt0-rho12-} are established for values of the parameter $a$ in the strip $-2<\mathrm{Im}\,a<0$
(cf. the monodromy data transformation~\eqref{eq:data-}).
The extension of the asymptotics to all values of $a\in\mathbb{C}\setminus\{\mi(2m-1),m\in\mathbb{Z}\}$
is accomplished via Lemma~\ref{th:bootstrap-log}.
\end{proof}
\setcounter{theorems}1
\begin{theorems}\hspace{-9pt}${}^{\mathbf\prime}$\label{th:Asympt0-ln-rho12+}
Let $(u(\tau), \varphi(\tau))$ be a solution of the system~\eqref{eq:dp3}, \eqref{eq:varphi} corresponding to
the monodromy data $(a,s_{0}^{0},s_{0}^{\infty},s_{1}^{\infty},g_{11},g_{12},g_{21},g_{22})$.
Suppose that the conditions stated in Theorem~{\rm\ref{th:Asympt0-ln-rho12-}} are valid
{\rm(}cf. equations~\eqref{eqs:Asympt0-rho12:mondataConditions} and \eqref{eq:nonzero-gik-relation-rho12}{\rm)}.
Define
\begin{equation}\label{eq:def-cp-log}
c_+:=4\gamma+\psi(1/2-\mi a/2)-\frac{\pi\mi}{2}+\frac{\pi\mi(g_{12}-\mi g_{22})}{g_{12}+\mi g_{22}}+
\ln(\varepsilon b/2);
\end{equation}
then,
\begin{align}
u(\tau)\underset{\tau\to0^{+}}{=}&-\frac{\varepsilon\left(1+\mathcal{O}\big(\tau^2\ln^2\tau\big)\right)}
{4\tau\big(\ln\tau+c_+/2\big)^2},
\label{eq:u-asympt0-rho12+}\\
\me^{\mi\varphi(\tau)}\underset{\tau\to0^{+}}{=}&\frac{\me^{\pi a/2}}{2\pi}
\left(\Gamma\left(1/2-\mi a/2\right)\right)^2(g_{12}+\mi g_{22})^2\nonumber\\
&\times\big(2\tau^2\big)^{\mi a}
\exp\left(-2\mi\varepsilon b\tau^2\left(\left(\ln\tau+\frac{c_+}{2}-\frac12\right)^2+\frac14\right)
+\mathcal{O}\big(\tau^4\ln^4\tau\big)\right).\label{eq:varphi-asympt0-rho12+}
\end{align}
%\begin{gather}
%u(\tau)\underset{\tau\to0^{+}}{=}\frac{\varepsilon\big(1+\me^{-\pi a}\big)\,\mathfrak{g}_{3}(\tau)}
%{\tau\,\mathfrak{g}_{1}(\tau)\,\big(\mathfrak{f}_{4}(\tau)\big)^{2}}\left(1+\mathcal{O}\big((\tau\ln\tau)^2\big)\right),
%\label{eq:alt-u-asympt0-rho12}\\
%\me^{\mi\varphi(\tau)}\underset{\tau\to0^{+}}{=}\frac{\pi \me^{3\pi\mi/2}
%2^{\mi a} \tau^{2\mi a}}{\big(\Gamma (\tfrac{1}{2} \! + \! \tfrac{\mi a}{2}) \big)^{2} \cosh
%(\pi a/2)} \frac{\mathfrak{g}_{3}(\tau)}{\mathfrak{g}_{1}(\tau)}\left(1+\mathcal{O}\big((\tau\ln\tau)^2\big)\right),
%\label{eq:alt-varphi-asympy0-rho12}
%\end{gather}
%where for $k=1,2$:
%\begin{equation}\label{eq:g1g3}
%\mathfrak{g}_{2k-1}(\tau)=2w_{k}(1)+(1+\mi a )\mathfrak{f}_{2k}(\tau),
%\end{equation}
%with the functions $\mathfrak{f}_{2k}(\tau)$, $w_k(l)$, and $\Phi_k(a)$ defined in equations~\eqref{eq:f2f4} and
%\eqref{eq:Asympt0-rho12-wkPhik}.
\end{theorems}
\begin{proof}
The proof is similar to that given for Theorem~\ref{th:Asympt0-ln-rho12-}.
Since $s_0^0=-2\mi$, the relation~\eqref{eq:nonzero-gik-relation-rho12} also holds for the monodromy data corresponding
to the solutions studied in this theorem; thus, the parameter $c_+$ is correctly
defined by equation~\eqref{eq:def-cp-log}.

In this case, we apply to the asymptotics stated in Theorem~\ref{th:ln-regular-a-non0} the B\"acklund
transformations~\eqref{eq:u+} and \eqref{eq:varphi+} in order to arrive at the asymptotics \eqref{eq:u-asympt0-rho12+}
and \eqref{eq:varphi-asympt0-rho12+}, respectively.\footref{foot:alt-derivation-asympt-varphi}
It is important to note that the B\"acklund-transformation argument provides us with the proof of the
asymptotics~\eqref{eq:u-asympt0-rho12+} and \eqref{eq:varphi-asympt0-rho12+} for values of the parameter $a$
restricted to the strip $0<\mathrm{Im}\,a<2$. As in Theorems~\ref{th:ln-regular-a-non0} and
\ref{th:Asympt0-ln-rho12-}, the extension of these asymptotics to all values of
$a\in\mathbb{C}\setminus\{\mi(2m-1),m\in\mathbb{Z}\}$ is completed upon invoking Lemma~\ref{th:bootstrap-log}.
\end{proof}
\begin{remark}\label{rem:rho0-varrho1/2}
Theorem~\ref{th:ln-regular-a-non0} corresponds to the value $\rho=0$ (cf. equations~\eqref{eq:rho-general} and
\eqref{eqs:ln-as00-conditions}), whilst Theorems~\ref{th:Asympt0-ln-rho12-} and
\ref{th:Asympt0-ln-rho12-}${}^{\mathbf\prime}$
are related to the value $\varrho=1/2$ (cf. equations~\eqref{eq:rho-general} and
\eqref{eqs:Asympt0-rho12:mondataConditions}).
\hfill$\blacksquare$\end{remark}
\begin{lemma}\label{th:bootstrap-log}
{\bf Completion of the proofs of Theorems~{\rm{\bf\ref{th:ln-regular-a-non0}}, {\bf\ref{th:Asympt0-ln-rho12-}}},
and {\rm{\bf\ref{th:Asympt0-ln-rho12-}}${}^{\mathbf\prime}$}}. These theorems
are valid for all values of the parameter $a$ stated therein.
\end{lemma}
\begin{proof}
The results presented in Theorems~\ref{th:Asympt0-ln-rho12-} and
\ref{th:Asympt0-ln-rho12+}\hspace{-0pt}${}^{\mathbf\prime}$ are obtained for the parameter of formal monodromy, $a$,
belonging to the disjoint strips $-2<\mathrm{Im}\,a<0$ and  $0<\mathrm{Im}\,a<2$, respectively, of the complex plane.
The formulae for the respective asymptotics are similar, but the coefficients seem to be different, which is not
surprising, since, by construction, the parameter $a$ belongs to different strips.
In fact, if we assume that $\mathrm{Im}\,a\in(-2,2)$ and $a\neq\pm\mi$, then both results coincide. To prove this for
the asymptotics of the function $u(\tau)$, consider the difference $c_--c_+$ (cf. equations~\eqref{eq:def-cm-log}
and \eqref{eq:def-cp-log}):
\begin{equation}\label{eq:cm-cp}
\begin{gathered}
c_--c_+=\psi(1/2+\mi a/2)-\psi(1/2-\mi a/2)+\pi\mi+\pi\mi\left(\frac{g_{11}-\mi g_{21}}{g_{11}+\mi g_{21}}
-\frac{g_{12}-\mi g_{22}}{g_{12}+\mi g_{22}}\right)\\
=\pi\mi\left(\tanh(\pi a/2)+1+\frac{2\mi(g_{11}g_{22}-g_{12}g_{21}}{-\mi(1+\me^{-\pi a})}\right)
=\pi\mi\left(\frac{2\me^{\pi a/2}}{\me^{\pi a/2}+\me^{-\pi a/2}}-\frac{2}{1+\me^{-\pi a}}\right)=0.
\end{gathered}
\end{equation}
In the calculation~\eqref{eq:cm-cp}, the identity $\psi(1/2+z)-\psi(1/2-z)=\pi\tan(\pi z)$ and
equations~\eqref{eq:nonzero-gik-relation-rho12} and \eqref{eq:monodromy:detG} were used.

To confirm the coincidence of the asymptotics~\eqref{eq:varphi-asympt0-rho12-} and \eqref{eq:varphi-asympt0-rho12+}
for the function $\me^{\mi\varphi(\tau)}$, we must, in addition, verify that the corresponding multiplicative
constants are equal; this is done by considering their ratio:
\begin{equation*}\label{eq:ratio-multconst-log}
\begin{gathered}
\left(\frac{-2\pi\me^{-\pi a/2}}{\Gamma^2\left(\frac12+\frac{\mi a}{2}\right)(g_{11}+\mi g_{21})^2}\right):
\left(\frac{\me^{\pi a/2}}{2\pi}\Gamma^2\left(\frac12-\frac{\mi a}{2}\right)(g_{12}+\mi g_{22})^2\right)\\
=-\left(\frac{2\pi\me^{-\pi a/2}}{\Gamma\left(\frac12+\frac{\mi a}{2}\right)\Gamma\left(\frac12-\frac{\mi a}{2}\right)
(g_{11}+\mi g_{21})(g_{12}+\mi g_{22})}
\right)^2
=-\left(\frac{\cosh(\pi a/2)\,2\pi\,\me^{-\pi a/2}}{\pi(-\mi\,(1+\me^{-\pi a}))}\right)^2=1.
\end{gathered}
\end{equation*}
The line $\mathrm{Im}\,a=0$ requires separate consideration. Firstly, note that the proof presented above does not
sense the presence of this line, that is, the leading terms of asymptotics of the functions $u(\tau)$ and
$\me^{\mi\varphi(\tau)}$,
which can be defined by any one of the pair of formulae~\eqref{eq:u-asympt0-rho12-} and
\eqref{eq:varphi-asympt0-rho12-}
or \eqref{eq:u-asympt0-rho12+} and \eqref{eq:varphi-asympt0-rho12+}, and denoted henceforth as $u_{as}(\tau)$ and
$\me^{\mi\varphi_{as}(\tau)}$, respectively, are analytic functions of the monodromy data and,
in particular, the formal monodromy parameter $a$ in the strip $|\mathrm{Im}\,a|<1$. As a solution of
the differential equation~\eqref{eq:dp3}, the function $u(\tau)$ is also an analytic function of $a$; correspondingly,
$\me^{\mi\varphi(\tau)}$ is also an analytic function of $a$.
Secondly, note that the difference $u(\tau)-u_{as}(\tau):=E(\tau)$ is an analytic function of $a$, since it is
the difference of the analytic functions. Its small-$\tau$ asymptotics is studied in
Appendix~\ref{app:subsec:coefficients-expansionLog2}, where it is proved that the asymptotics of $E(\tau)$ does not
have any singularities for any values of $a$, so that the asymptotics of the function $u(\tau)$ stated in
Theorems~\ref{th:Asympt0-ln-rho12-} and \ref{th:Asympt0-ln-rho12+}\hspace{-0pt}${}^{\mathbf\prime}$ is true in the
entire strip $\mathrm{Im}\,a\in(-2,2)$ punctured at $a=\pm\mi$. This fact implies the validity of the same statement
for the asymptotics of the function $\me^{\mi\varphi(\tau)}$.

We are now in a position to employ the bootstrap argument. In the previous paragraph, the validity of each of the
Theorems~\ref{th:Asympt0-ln-rho12-} and \ref{th:Asympt0-ln-rho12+}\hspace{-0pt}${}^{\mathbf\prime}$
was extended to the strip $\mathrm{Im}\,a\in(-2,2)$ punctured at the two points $a=\pm\mi$.
One applies to the solution and the corresponding asymptotics stated in
Theorems~\ref{th:Asympt0-ln-rho12-} and \ref{th:Asympt0-ln-rho12+}\hspace{-0pt}${}^{\mathbf\prime}$
the inverses of the B\"acklund transformations that were used to obtain these asymptotic results;
then, one arrives at the asymptotics formulated in Theorem~\ref{th:ln-regular-a-non0}, but now with the validity of
these results extended to the wider strip $\mathrm{Im}\,a\in(-3,3)=(-3,1)\cup(-1,3)$ punctured at the three points
$a=-2\mi,0,2\mi$.

At the next stage of the bootstrap argument, we apply the B\"acklund transformations~\eqref{eq:u-} and
\eqref{eq:varphi-} and arrive at the asymptotics stated in Theorem~\ref{th:Asympt0-ln-rho12-}, and apply the
B\"acklund transformations~\eqref{eq:u+} and \eqref{eq:varphi+} to obtain the results stated in
Theorem~\ref{th:Asympt0-ln-rho12+}\hspace{-0pt}${}^{\mathbf\prime}$, but now with the validity of both theorems
extended to the strip $\mathrm{Im}\,a\in(-4,4)=(-4,2)\cup(-2,4)$ punctured at the four points
$a=\pm3\mi,\pm\mi$.

Finally, an inductive argument completes the proof that the asymptotics stated in Theorem~\ref{th:ln-regular-a-non0}
are valid for $a\in\mathbb{C}\setminus\mi2\mathbb{Z}$, and the results of Theorems~\ref{th:Asympt0-ln-rho12-} and
\ref{th:Asympt0-ln-rho12+}\hspace{-0pt}${}^{\mathbf\prime}$ hold for
$a\in\mathbb{C}\setminus\lbrace\mi(2m-1),m\in\mathbb{Z}\rbrace$.
\end{proof}
\begin{remark}\label{rem:c+c-=0}
{}From the qualitative point of view, the asymptotic behaviours of the solutions described in
Theorems~\ref{th:Asympt0-ln-rho12-} and \ref{th:Asympt0-ln-rho12+}\hspace{-0pt}${}^{\mathbf\prime}$, which
correspond to the case $c_+=c_-=:\tilde{c}_{-1,3}=0$, do not exibit any special lineaments when compared to the
case for non-vanishing values of this parameter. A study of the complete asymptotic expansion
(see Appendix~\ref{app:sec:full0expansionLog2}), however, shows that solutions of equation~\eqref{eq:dp3}
corresponding to $\tilde{c}_{-1,3}=0$ are the only ones for which the \emph{levels} (see
Appendix~\ref{app:subsec:coefficients-expansionLog2} for the definition of levels) of the complete asymptotic
expansions are represented by truncated logarithmic series; in Corollary~\ref{cor:c+=c-=0} below, we specify
such solutions in terms of the monodromy data.
\hfill$\blacksquare$\end{remark}
\begin{corollary}\label{cor:c+=c-=0}
Let the pair of functions $(u(\tau),\me^{\mi\varphi(\tau)})$ correspond to the monodromy data specified in
Theorem~{\rm\ref{th:Asympt0-ln-rho12-};} then, the parameter $c_-=c_+=0$ in the asymptotic formulae
\eqref{eq:u-asympt0-rho12-}, \eqref{eq:varphi-asympt0-rho12-} and \eqref{eq:u-asympt0-rho12+},
\eqref{eq:varphi-asympt0-rho12+} iff the monodromy data satisfy, \pmb{in addition}, one of the following three
conditions:
\begin{enumerate}
\item[\pmb{$(1)$}]
$G_+G_-\neq0$,
\begin{gather}
g_{11}g_{22}=-\frac{G_+G_-}{4(1+\me^{-\pi a})},\quad
g_{11}g_{12}=\frac{(2+\mi G_-)G_-}{4(1+\me^{-\pi a})},\quad
g_{22}g_{21}=\frac{(2-\mi G_+)G_+}{4(1+\me^{-\pi a})},\label{eq:gikG+G-generic}\\
g_{11}(g_{12}+\mi g_{22})=\frac{G_-}2,\qquad
g_{22}(g_{11}+\mi g_{21})=\frac{\mi G_+}2,\label{eq:gikG+G-forStokes}
\end{gather}
\item[\pmb{$(2)$}]
$G_-=0,\quad G_+=2\mi\me^{-\pi a}$,
\begin{equation}\label{eq:G-=0}
g_{11}=0,\qquad
g_{22}\in\mathbb{C}\setminus\{0\},\qquad
g_{12}g_{21}=-1,\qquad
g_{22}g_{21}=\mi\me^{-\pi a},
\end{equation}
\item[\pmb{$(3)$}]
$G_+=0,\quad G_-=-2\mi\me^{-\pi a}$,
\begin{equation}\label{eq:G+=0}
g_{22}=0,\qquad
g_{11}\in\mathbb{C}\setminus\{0\},\qquad
g_{12}g_{21}=-1,\qquad
g_{11}g_{12}=-\mi\me^{-\pi a},
\end{equation}
\end{enumerate}
where
\begin{equation}\label{eq:defG+-}
G_{\pm}=\frac{1}{\pi}\big(1+\me^{-\pi a}\big)\Big(\ln(\varepsilon b/2)+4\gamma
+\frac12\big(\psi(1/2+\mi a/2)+\psi(1/2-\mi a/2)\big)\Big)\pm\mi\me^{-\pi a}.
\end{equation}
\end{corollary}
\begin{proof}
We outline the proof for item~\pmb{$(1)$}. The derivations of the formulae in items~\pmb{$(2)$} and \pmb{$(3)$}
are much simpler, and are thus left to the interested reader.

Assume that $c_-=c_+=0$, then $c_-+c_+=0$, consequently, adding equations~\eqref{eq:def-cm-log} and
\eqref{eq:def-cp-log}, taking the two fractions consisting of the monodromy data $g_{ij}$ to a common denominator,
and substituting the identity~\eqref{eq:nonzero-gik-relation-rho12} for the common denominator, one shows that the
resulting equation can be written as
\begin{equation}\label{eq:gik-T}
g_{11}g_{12}+g_{21}g_{22}=(1+\me^{-\pi a})\Psi,\quad
\Psi=\frac1\pi\Big(4\gamma+\ln(\varepsilon b/2)+\frac12\big(\psi(1/2+\mi a/2)+\psi(1/2-\mi a/2)\Big).
\end{equation}
Since $s_0^0=-2\mi$, it follows from equation~\eqref{eq:monodromy:main} that
\begin{equation}\label{eq:monodromy:main:s00eq-2i}
g_{21}g_{22}-g_{11}g_{12}=\mi(2g_{11}g_{22}+\me^{-\pi a});
\end{equation}
thus, adding and subtracting equations~\eqref{eq:gik-T} and \eqref{eq:monodromy:main:s00eq-2i}, one arrives at the
system
\begin{equation}\label{sys:gik}
\begin{gathered}
2g_{21}g_{22}=(1+\me^{-\pi a})\Psi+\mi(2g_{11}g_{22}+\me^{-\pi a}),\\
2g_{12}g_{11}=(1+\me^{-\pi a})\Psi-\mi(2g_{11}g_{22}+\me^{-\pi a}).
\end{gathered}
\end{equation}
Multiplying equations~\eqref{sys:gik} and taking equation~\eqref{eq:monodromy:detG} into account, one gets
\begin{equation}\label{eq:g11g22-quadratic}
4g_{11}g_{22}(g_{11}g_{22}-1)=(1+\me^{-\pi a})^2\Psi^2+(2g_{11}g_{22}+\me^{-\pi a})^2.
\end{equation}
Expanding the right-most (parenthetical) term in equation~\eqref{eq:g11g22-quadratic}, cancelling the quadratic
terms $4(g_{11}g_{22})^2$, and introducing $G_{\pm}$ as in equation~\eqref{eq:defG+-}, we obtain the first
equation in the list~\eqref{eq:gikG+G-generic}; the remaining two equations in the list~\eqref{eq:gikG+G-generic}
are obtained upon substituting the first equation into the expressions on the right-hand sides of the
system~\eqref{sys:gik}.
The equations in \eqref{eq:gikG+G-forStokes} are obtained by taking linear combinations of the
equations in \eqref{eq:gikG+G-generic}. Conversely, substitute equations~\eqref{eq:gikG+G-generic} into, say,
the definition of $c_-$ (cf. equation~\eqref{eq:def-cm-log}) to prove that $c_-=0$.
\end{proof}
\begin{remark}\label{rem:G+/-=0}
In Corollary~\ref{cor:c+=c-=0}, the quantities $G_{\pm}$ are functions of the formal monodromy,
$a$, and the scaling parameter $\varepsilon b$; therefore, if $a\neq\mi(2m-1)$, $m\in\mathbb{Z}$, then
one, and only one, solution $u(\tau)$ with the asymptotics given in Theorem~\ref{th:Asympt0-ln-rho12-} for $c_-=0$
exists. One of the monodromy parameters, either $g_{11}\neq0$ or $g_{22}\neq0$, defines the
``constant of integration'' of the function $\varphi(\tau)$ in equation~\eqref{eq:varphi}, so that the function
$\me^{\mi\varphi(\tau)}$ is also unique modulo this multiplicative $\tau$-independent parameter.
\hfill$\blacksquare$\end{remark}
\begin{remark}\label{rem:truncated}
Items~\pmb{$(2)$} and \pmb{$(3)$} of Corollary~\ref{cor:c+=c-=0} correspond to the case when one of the monodromy
parameters, either $g_{11}$ or $g_{22}$, vanishes. According to Theorem~\ref{th:Asympt0-ln-rho12-},
such special values for $g_{11}$ or $g_{22}$ give rise to
small-$\tau$ asymptotics that are qualitatively similar to the small-$\tau$ asymptotics of solutions with
monodromy data described in item~\pmb{$(1)$} of Corollary~\ref{cor:c+=c-=0}. The large-$\tau$ asymptotics
of solutions with monodromy data given in items~\pmb{$(2)$} and \pmb{$(3)$} of Corollary~\ref{cor:c+=c-=0},
however, are more interesting, because they are the only solutions that have the small-$\tau$
logarithmic behaviour presented in Theorem~\ref{th:Asympt0-ln-rho12-} having truncated asymptotics as
$\tau\to+\infty$.
In this context, the following intriguing question manifests:
do there exist, for any scaling $\varepsilon b$, values of the formal monodromy parameter $a$ that solve the equation
$G_{\pm}=0$, and, if so, how many solutions exist? We haven't yet studied this question. For $\varepsilon b=2$,
numerical studies using \textsc{Maple} give a series (infinite?) of solutions for each equation $G_{\pm}=0$,
namely,
$a=a^1_{\pm}=0.2381378288\ldots\mp\mi0.6358442252\ldots$, $a=a^2_{\pm}=0.1144878083\ldots\mp\mi1.714583576\ldots$,
$a=a^3_{\pm}=0.09349464758\ldots\mp\mi2.744016682\ldots$, etc. If the monodromy parameter $a$ is a solution of
the equation $G_{\pm}=0$, then, for this parameter value, there exists a unique solution $u(\tau)$ of
equation~\eqref{eq:dp3} and a corresponding function $\me^{\mi\varphi(\tau)}$ which is uniquely defined modulo a
multiplicative $\tau$-independent non-vanishing parameter that is defined in terms of $g_{12}$ or $g_{21}$.
\hfill$\blacksquare$\end{remark}
%%%%%%%%%%%%%%%%%%%%%%%%%%%%%%%%%%%%%%%%%%%%%%%%%%%%%%%%%%%%%%%%%%%%%%%%%%%%%%%%%%%%%%%%%%%%%%%%%%%%%%%%%%%%%%%%%%%%
%%%%%%%%%%%%%%%%%%%%%%%%%%%%%%%%%%%%%%%%%%%%%%%%%%%%%%%%%%%%%%%%%%%%%%%%%%%%%%%%%%%%%%%%%%%%%%%%%%%%%%%%%%%%%%%%%%%%
\section{Solutions with Poles Accumulating at the Origin: $\mathrm{Re}(\varrho)=1/2$}\label{sec:poles}
The reader may have noted that, although the respective denominators of the asymptotic formulae for $u(\tau)$
given in equations~\eqref{eq:Asympt0u-} and \eqref{eq:Asympt0u+} vanish for an infinite sequence of points
$\tau=\tau_p\to0$, $p\in\mathbb{N}$, for $\varrho=1/2+\mi\varkappa$, $\varkappa\in\mathbb{R}\setminus\{0\}$, such
values of the branching parameter were not excluded from the formulations of
Theorems~\ref{th:alt2B1asympt0p} and \ref{th:alt2B1asympt0p}\hspace{-1pt}${}^{\mathbf\prime}$; more precisely, since
\begin{equation*}\label{eq:asymptotic-poles-derivation}
w_k\tau^{-2\mi\varkappa}+w_{k+1}\tau^{2\mi\varkappa}=
2\sqrt{w_k}\,\sqrt{w_{k+1}}\cos\Big(2\varkappa\ln\tau+\mi\ln\big(\sqrt{w_k}/\sqrt{w_{k+1}}\big)\Big),\qquad
k=1,3,
\end{equation*}
the asymptotic formulae have second-order poles at the points
\begin{equation}\label{eq:asymptotic-poles}
\tau_p=\exp\left(-\frac{\pi p}{2|\varkappa|}+\frac{\pi}{4\varkappa}
-\frac{\mi}{4\varkappa}\ln\frac{w_k}{w_{k+1}}\right),\quad
k=1,3.
\end{equation}
Note that the value of $k$ (equal to $1$ or $3$) is not important because of the second identity in
\eqref{eqs:w-identities}, and the choice of the branch of the $\ln$-function in equation~\eqref{eq:asymptotic-poles}
is also not essential because its selection is a mere redefinition (shift) of $p$, which has the sense
of a parameter tending to $+\infty$. Hereafter, we assume that the branch of the $\ln$-function is fixed.

Note that all the points $\tau_p$ belong to the ray in $\mathbb{C}$ that is defined by the complex number $\tau_0$;
therefore, if $\arg\tau_0=0$ and the notation $\tau\to0^+$ is understood in the standard sense, then the asymptotics
of the function $u(\tau)$ is considered for real positive values of $\tau$ approaching the origin, and we encounter
a problem related with an infinite number of poles of the asymptotics located along the way as the origin is
approached. One can surmise that, somewhere in a neighbourhood of the poles of the asymptotics, are located poles
of $u(\tau)$; if, however, we understand $\tau\to0^+$ in the standard way, then we are restrained from going around
the poles $\tau_p$, so that we stop at the first pole, and the sense of such asymptotics is unclear.

In fact, such a “problem” with the asymptotic formulae occurs for all the Painlev\'e equations, and has a standard
solution. In our case, for example, the notation $\tau\to0^+$ means that we take $\arg\,\tau=0$ on the positive real
semi-axis, and approach the origin in a wider domain $\mathcal{S}\in\mathbb{C}$.
For the regular singular point of $u(\tau)$ at $\tau=0$, the domain $\mathcal{S}$
is a full neighbourhood of the origin cut along the negative real semi-axis; however, in this case, we must also
take into account those points $\tau_p$ for which $\arg\tau_p$ is arbitrary, which implies that the definition
of $\mathcal{S}$ should be supplemented by deleting from it an infinite number of discs centred at the points $\tau_p$.
Below, we complete the definition of the
discs, and note here that, due to the Painlev\'e property of the function $u(\tau)$, one can take an arbitrary path
to the origin in the multiply-connected domain $\mathcal{S}$ along which the
asymptotics of the function $u(\tau)$ is considered, since both $u(\tau)$ and its asymptotics are uniquely defined
in $\mathcal{S}$.

We now proceed with the definition of the discs.
The points $\tau_p$ are located on the ray with the origin at $\tau=0$; therefore, the distance between the
neighbouring points $\tau_p$ is $| \tau_{p+1}-\tau_{p}|=J|\tau_p|$, where
$J=1-\exp\left(-\pi/(2|\varkappa|)\right)$.
Next, we consider the discs $\mathcal{D}_p$ centred at $\tau_p$ with radius $R_p=J|\tau_p|^{1+\delta_d}$,
where $\delta_d\in[0,2]$ is the same for all $p\in\mathbb{N}$ and will be specified later.
Clearly, for $\delta_d\in(0,2]$ and small enough $\tau$,
$\mathcal{D}_p\cap\mathcal{D}_{p^{\prime}}=\varnothing$ $\forall$ $p\neq p^{\prime}\in\mathbb{N}$. To ensure this
property for $\delta_d=0$, we have to reduce the coefficient $J$ to, say, $J/3$, because the connectedness of the
intersection of the domain $\mathcal{S}$ with any small enough neighbourhood of the origin is important.
The formal definition of the domain $\mathcal{S}$ reads:
\begin{equation}\label{eq:S-D-def}
\mathcal{S}:=\{\tau\in\mathbb{C}:\;|\arg\tau|<\pi\}
\setminus\underset{\scriptscriptstyle p\in\mathbb{N}}\cup\mathcal{D}_p,\quad
\mathcal{D}_p:=\left\{\tau\in\mathbb{C}:\;|\tau-\tau_p|<R_p=J|\tau_p|^{1+\delta_d}\right\},
\end{equation}
with $\delta_d$ and $J$ defined above.
\begin{remark}\label{rem:tau0plusM}
We are now ready to clarify the notation $\tau\to0^+$ (cf. Remark~\ref{rem:tau0plus}) appearing
in the asymptotic formulae of solutions that have sequences of poles accumulating at the origin; for such solutions,
the notation $\tau\to0^+$ is equivalent to $\tau\in\mathcal{S}$ and $\tau\to0$: for brevity, we write
$\mathcal{S}\ni\tau\to0$.
\hfill$\blacksquare$\end{remark}
Our calculations of the monodromy data in \cite{KitVar2004,KitVarZapiski2024} uphold this modification
of the sense of the notation $\tau\to0^+$ because they deal with estimates of functions with power-like
behaviour. This fact implies that Theorems~\ref{th:alt2B1asympt0p} and \ref{th:alt2B1asympt0p}${}^{\mathbf\prime}$
also sustain this modification. The error estimates in these theorems remain unchanged as the origin is approached
``far'' from the boundaries of the discs $\mathcal{D}_{p}$, or, when $\delta_{d}=0$; it is clear, however, that the
error of the approximations increases as the points $\tau_{p}$ are approached, that is, when the parameter
$\delta_{d}>0$.
Here, we formulate the special case of Theorems~\ref{th:alt2B1asympt0p} and
\ref{th:alt2B1asympt0p}${}^{\mathbf\prime}$ for $\varrho=1/2+\mi\varkappa$ that is applicable to situations related
with the existence of sequences of poles accumulating at $\tau=0$.
\begin{theorem}\label{th:through-poles}
Let $(u(\tau),\varphi(\tau))$ be a solution of the system \eqref{eq:dp3}, \eqref{eq:varphi}
corresponding to the monodromy data $(a,s_{0}^{0},s_{0}^{\infty},s_{1}^{\infty},g_{11},g_{12},g_{21},g_{22})$.
Suppose that
\begin{equation}\label{eqs:s00s0infs1inf-poles-gen}
s_0^0=-2\mi\cosh(2\pi\varkappa),\quad
\varkappa\in\mathbb{R}\setminus\{0\},\qquad
s_0^\infty s_1^\infty\neq0.
\end{equation}
The conditions~\eqref{eqs:s00s0infs1inf-poles-gen} imply that
\begin{equation}\label{eq:conditions-a-varkappa}
a\neq 2\varkappa+\mi(2k-1),\qquad
k\in\mathbb{Z},
\end{equation}
\begin{equation}\label{eqs:conditions-gik-varkappa}
g_{11}\me^{-\pi\mi/4}\me^{\pm\pi\varkappa}+g_{21}\me^{\pi\mi/4}\me^{\mp\pi\varkappa}\neq0,\quad
g_{12}\me^{-\pi\mi/4}\me^{\pm\pi\varkappa}+g_{22}\me^{\pi\mi/4}\me^{\mp\pi\varkappa}\neq0,
\end{equation}
where either the upper signs or the lower signs, respectively, are taken.

Define
\begin{equation}\label{eq:Avarkappa}
\widehat{\mathscr{A}}(\varkappa):=\me^{\frac{\pi\varkappa}{2}}\left(\frac{\varepsilon b}{2}\right)^{-\mi\varkappa}
\frac{\Gamma(1+2\mi\varkappa)}{\Gamma(1-2\mi\varkappa)}\,\Gamma\!\left(\frac12+\frac{\mi a}{2}-\mi\varkappa\right)
\left(g_{11}\me^{-\frac{\pi\mi}{4}}\me^{\pi\varkappa}+g_{21}\me^{\frac{\pi\mi}{4}}\me^{-\pi\varkappa}\right),
\end{equation}
and assume that $\mathcal{S}$ is defined by \eqref{eq:S-D-def} with $\delta_d\in[0,2)${\rm;}
then,
\begin{align}
u(\tau)\underset{\scriptscriptstyle\mathcal{S}\ni\tau\to0}{=}&\,\frac{4\varepsilon\varkappa^{2}
\widehat{\mathscr{A}}(\varkappa)\widehat{\mathscr{A}}(-\varkappa)\big(1+\mathcal{O}(\tau^{2-\delta_d})\big)}
{\tau\big(\widehat{\mathscr{A}}(\varkappa)\tau^{-2\mi\varkappa}-\widehat{\mathscr{A}}(-\varkappa)
\tau^{2\mi\varkappa}\big)^2}, \label{eq:A-varkappa-asympt-u} \\
\me^{\mi\varphi(\tau)}\underset{\scriptscriptstyle\mathcal{S}\ni\tau\to0}{=}&\,\frac{2\pi\me^{-\frac{3\pi\mi}{2}}
\me^{-\frac{\pi a}{2}}}{\widehat{\mathscr{A}}(\varkappa)\widehat{\mathscr{A}}(-\varkappa)}
\left(2\tau^{2}\right)^{\mi a}\big(1+\mathcal{O}(\tau^{2-\delta_d})\big). \label{eq:A-varkappa-asympt-varphi}
\end{align}
\end{theorem}
\begin{proof}
This theorem is a reformulation of Theorem~\ref{th:alt2B1asympt0p} for
$\varrho=1/2+\mi\varkappa$, where we use the relations $w_1=\widehat{\mathscr{A}}(\varkappa)$ and
$w_2=-\widehat{\mathscr{A}}(-\varkappa)$, which are valid for this value of $\varrho$;
moreover, the restriction $-2<\mathrm{Im}\,a<0$ is removed due to Theorem~\ref{th:restriction-on-Ima-removed}.
The error estimate, however, requires further commentary: if $\delta_d=0$, which means that the discs $\mathcal{D}_p$
are large enough and $\tau$ is far from the pole $\tau_p$ of the leading term of asymptotics, then the
error estimate is the same as in Theorem~\ref{th:alt2B1asympt0p}, but the situation changes when we consider discs
with smaller radii $R_p=\mathcal{O}\big(|\tau_p|^{1+\delta_d}\big)$ as $\tau_p\to0$ and
$\tau\to\partial\mathcal{D}_p$.
To evaluate the error of the approximation of the function $u(\tau)$ by its leading term of asymptotics, we have
to invoke the first correction term $yA_1(x)/\tau$ (see Appendix~\ref{app:subsec:super-generating-function} and
equation~\eqref{eq:A1-pole-estimate} below), which is of the order
$\tau\tau_p^3/(\tau-\tau_p)^3$ for $\tau$ near $\partial\mathcal{D}_p$,
while the leading term in this domain is of the order $\tau_p^2/(\tau(\tau-\tau_p)^2)$; thus, the error of the
approximation differs from the leading term by the factor
$\tau^2\tau_p/(\tau-\tau_p)=\mathcal{O}\big(\tau^{2-\delta_d}\big)$. The error estimate for the function
$\me^{\mi\varphi(\tau)}$ is obtained via equation~\eqref{eq:varphi} by integrating the corresponding asymptotics
of $u(\tau)$.
\end{proof}
\begin{remark}\label{rem:th:through-poles-alt}
As a consequence of Theorem \ref{th:restriction-on-Ima-removed}, the conditions of Theorems \ref{th:alt2B1asympt0p}
and  \ref{th:alt2B1asympt0p}\hspace{-1pt}${}^{\mathbf\prime}$ coincide; subsequently, we can obtain
Theorem \ref{th:through-poles} with the help of Theorem \ref{th:alt2B1asympt0p}\hspace{-1pt}${}^{\mathbf\prime}$.
This leads to seemingly different, yet equivalent, asymptotics for the functions $u(\tau)$ and
$\me^{\mi\varphi(\tau)}$: we formulate these equivalent results in
Theorem \ref{th:alt-through-poles}\hspace{-1pt}${}^{\mathbf\prime}$.
\hfill$\blacksquare$\end{remark}
\begin{theorems}\hspace{-10pt}${}^{\mathbf\prime}$\label{th:alt-through-poles}
Let $(u(\tau),\varphi(\tau))$ be a solution of the system \eqref{eq:dp3}, \eqref{eq:varphi}
corresponding to the monodromy data $(a,s_{0}^{0},s_{0}^{\infty},s_{1}^{\infty},g_{11},g_{12},g_{21},g_{22})$.
Suppose that the conditions~\eqref{eqs:s00s0infs1inf-poles-gen}--\eqref{eqs:conditions-gik-varkappa}
of Theorem~{\rm\ref{th:through-poles}} are satisfied.

Define
\begin{equation}\label{eq:Bvarkappa}
\widehat{\mathscr{B}}(\varkappa):=\me^{-\frac{\pi\varkappa}{2}}\left(\frac{\varepsilon b}{2}\right)^{-\mi\varkappa}
\frac{\Gamma(1+2\mi\varkappa)}{\Gamma(1-2\mi\varkappa)}\,\Gamma\!\left(\frac12-\frac{\mi a}{2}-\mi\varkappa\right)
\left(g_{12}\me^{-\frac{\pi\mi}{4}}\me^{\pi\varkappa}+g_{22}\me^{\frac{\pi\mi}{4}}\me^{-\pi\varkappa}\right),
\end{equation}
and assume that $\mathcal{S}$ is defined by \eqref{eq:S-D-def} with $\delta_d\in[0,2);$
then,
\begin{align}
u(\tau)\underset{\scriptscriptstyle\mathcal{S}\ni\tau\to0}{=}&\;\frac{4\varepsilon\varkappa^{2}
\widehat{\mathscr{B}}(\varkappa)\widehat{\mathscr{B}}(-\varkappa)\big(1+\mathcal{O}(\tau^{2-\delta_d})\big)}
{\tau\big(\widehat{\mathscr{B}}(\varkappa)\tau^{-2\mi\varkappa}-\widehat{\mathscr{B}}(-\varkappa)
\tau^{2\mi\varkappa}\big)^2}, \label{eq:B-varkappa-asympt-u-alt} \\
\me^{\mi\varphi(\tau)}\underset{\scriptscriptstyle\mathcal{S}\ni\tau\to0}{=}&\;\frac{\me^{-\frac{3\pi\mi}{2}}
\me^{\frac{\pi a}{2}}\widehat{\mathscr{B}}(\varkappa)\widehat{\mathscr{B}}(-\varkappa)}{2\pi}
\left(2\tau^{2}\right)^{\mi a}\big(1+\mathcal{O}(\tau^{2-\delta_d})\big). \label{eq:B-varkappa-asympt-varphi-alt}
\end{align}
\end{theorems}
\begin{proof}
The results of this theorem are a reformulation of those in
Theorem~\ref{th:alt2B1asympt0p}\hspace{-1pt}${}^{\mathbf\prime}$ for $\varrho=1/2+\mi\varkappa$, where, under this
substitution, $w_3=\widehat{\mathscr{B}}(\varkappa)$ and $w_4=-\widehat{\mathscr{B}}(-\varkappa)$; moreover,
the restriction $0<\mathrm{Im}\,a<2$ is removed due to Theorem~\ref{th:restriction-on-Ima-removed}. The justification
for the error estimates is literally the same as that given in the proof of Theorem~\ref{th:through-poles}.
\end{proof}
\begin{remark}\label{rem:equivalence-theorems-trough-poles}
The asymptotics of the functions $u(\tau)$ and $\me^{\mi\varphi(\tau)}$ given in
Theorems~\ref{th:through-poles} and \ref{th:alt-through-poles}\hspace{-1pt}${}^{\mathbf\prime}$ coincide:
this is a consequence of the relations~\eqref{eqs:w-identities}; in particular,
equations~\eqref{eq:A-varkappa-asympt-u} and \eqref{eq:B-varkappa-asympt-u-alt} imply that
equation~\eqref{eq:asymptotic-poles} for the poles of the leading term of asymptotics can be simplified as
\begin{equation}\label{eq:tau-p-A-B}
\tau_p=\exp\left(-\frac{\pi p}{2|\varkappa|}
+\frac{\mi}{4\varkappa}\ln\frac{\widehat{\mathscr{A}}(-\varkappa)}{\widehat{\mathscr{A}}(\varkappa)}\right)
=\exp\left(-\frac{\pi p}{2|\varkappa|}
+\frac{\mi}{4\varkappa}\ln\frac{\widehat{\mathscr{B}}(-\varkappa)}{\widehat{\mathscr{B}}(\varkappa)}\right).
\end{equation}
\hfill$\blacksquare$\end{remark}
\begin{corollary}\label{cor:nozeros-nopoles}
Let $(u(\tau),\varphi(\tau))$ be a solution of the system \eqref{eq:dp3}, \eqref{eq:varphi}
corresponding to the monodromy data $(a,s_{0}^{0},s_{0}^{\infty},s_{1}^{\infty},g_{11},g_{12},g_{21},g_{22})$.
Suppose that the conditions~\eqref{eqs:s00s0infs1inf-poles-gen}--\eqref{eqs:conditions-gik-varkappa}
of Theorem~{\rm\ref{th:through-poles}} are satisfied.

For $\epsilon>0$, define
\begin{equation*}\label{eq:Sepsilon-def}
\mathcal{S}_{\epsilon}:=\{\tau\in\mathcal{S}: |\tau|<\epsilon\}.
\end{equation*}
Then, there exists $\epsilon>0$ such that $u(\tau)$ and $\me^{\mi\varphi(\tau)}$ have neither zeros
nor poles in $\mathcal{S}_{\epsilon}$.
\end{corollary}
\begin{proof}
The absence of poles is apparent because finite-valued asymptotics at a point means that an approximated
function has a finite value at such a point.

The proof for the zeros proceeds by contradiction. If such an $\epsilon>0$ does not exist, then there exists
a sequence of zeros accumulating at the origin, which contradicts the asymptotics~\eqref{eq:A-varkappa-asympt-u}
for the function $u(\tau)$. Actually, in the case of zeros $\tau^0_k\to0$, $k\in\mathbb{N}$, we find, upon
substituting $u(\tau^0_k)=0$ into equation~\eqref{eq:A-varkappa-asympt-u}, that $0=\mathcal{O}(1/\tau^0_k)$ as
$k\to\infty$, so that the approximation of $u(\tau)$ in a neighbourhood of $\tau^0_k$
becomes worse as $\tau\to0$, which is a contradiction.

According to equation~\eqref{eq:varphi}, the function
$\me^{\mi\varphi(\tau)}$ has first-order zeros or poles only at the zeros of the function $u(\tau)$, depending
on the expansion~\eqref{eq:Taylor-zero-u+} or \eqref{eq:Taylor-zero-u-}, respectively.
\end{proof}

It follows from Theorem~\ref{th:through-poles} and Corollary~\ref{cor:nozeros-nopoles} that those zeros or poles
accumulating at the origin, if any, of the function $u(\tau)$ can be located only in the discs $\mathcal{D}_p$
for large enough $p$. As noted in the last sentence of the proof of Corollary~\ref{cor:nozeros-nopoles},
the zeros and poles of the function $\me^{\mi\varphi(\tau)}$ are located at the zeros of the corresponding function
$u(\tau)$, and, therefore, do not require further attention. Our main goal in this section is to establish
the following theorem.\footnote{\label{foot:sec5:varepsilon} In the proof of
Theorem~\ref{th:poles-u-one-to-one-poles-asympt} and in all constructions up to, and including, the
inequality~\eqref{ineq:u-uas-integral-above-gen}, we set $\varepsilon=+1$ in equation~\eqref{eq:dp3}; all
statements remains valid for $\varepsilon=-1$. To get the $\varepsilon$-dependent variant of the aforementioned
constructions, one has to make the changes $u(\tau)\to\varepsilon u(\tau)$,
$u_{as}(\tau)\to\varepsilon u_{as}(\tau)$, and $b\to\varepsilon b$.}
\begin{theorem}\label{th:poles-u-one-to-one-poles-asympt}
For large enough $p\in\mathbb{N}$, each disc $\mathcal{D}_p$ contains one, and only one, pole and no zeros of
the solution $u(\tau)$ corresponding to the monodromy data defined in Theorem~{\rm\ref{th:through-poles}}.
\end{theorem}
\begin{proof}
It is straightforward to establish that if $\tau_{\hat{p}}\in\mathbb{C}$ is a pole of some solution $u(\tau)$,
then $\tau_{\hat{p}}\neq0$, it is a second-order pole, and the corresponding Laurent-series expansion of
$u(\tau)$ at $\tau_{\hat{p}}$ is
\begin{equation}\label{eq:Laurent-pole-u}
u(\tau)=-\frac{\tau_{\hat{p}}}{4(\tau-\tau_{\hat{p}})^2}+u_0-\frac{u_0}{\tau_{\hat{p}}}(\tau-\tau_{\hat{p}})
+\frac{2ab\tau_{\hat{p}}-24\tau_{\hat{p}}u_0^2+9u_0}{10\tau_{\hat{p}}^2}(\tau-\tau_{\hat{p}})^2+
\mathcal{O}\big((\tau-\tau_{\hat{p}})^3\big),
\end{equation}
where $u_0$ is a complex parameter. Thus, both the function $u(\tau)$ and its leading term of asymptotics only have
poles of the second order.

Since our asymptotics are obtained with the help of the B\"acklund transformations~\eqref{eq:u+} and \eqref{eq:u-},
it is clear that the second-order poles are the images of the first-order zeros of the solutions
that are used as the ``seed solutions'' for these transformations: we now consider this statement more precisely.

The direct substitution of the general form of the Taylor-series expansion for $u(\tau)$ at its zero,
$\tau_0\in\mathbb{C}\setminus\{0\}$,
into equation~\eqref{eq:dp3} shows that all zeros are of the first order, and two possible expansions can be found:
\begin{equation}\label{eq:Taylor-zero-u+}
u(\tau)=+\mi b(\tau-\tau_0)-\frac{(2a-\mi)b}{2\tau_0}(\tau-\tau_0)^2\!+u_3(\tau-\tau_0)^3\!+
\frac{4b^2\!+(\mi a-1)u_3}{2\tau_0}(\tau-\tau_0)^4\!+\mathcal{O}\big((\tau-\tau_0)^5\big),
\end{equation}
\begin{equation}\label{eq:Taylor-zero-u-}
u(\tau)=-\mi b(\tau-\tau_0)-\frac{(2a+\mi)b}{2\tau_0}(\tau-\tau_0)^2\!+u_3(\tau-\tau_0)^3\!+
\frac{4b^2\!-(\mi a+1)u_3}{2\tau_0}(\tau-\tau_0)^4\!+\mathcal{O}\big((\tau-\tau_0)^5\big),
\end{equation}
where $u_3$ is a complex parameter. Hence, there are two types of zeros, and their expansions
differ by ``formal conjugation'', i.e., the change $\mi\to-\mi$ in all terms. Denote by $\tau_0^{\pm}$,
respectively, the zeros of $u(\tau)$ such that the first coefficient of the Taylor-series expansion of $u(\tau)$
at these zeros is equal to $\pm\mi b$.

Now, with the above information, we can check that the transformation~\eqref{eq:u+} ($u\to u_+$) sends the
$\tau_0^+$ zeros to the poles $\tau_p=\tau_0^+$, the $\tau_0^-$ zeros to holomorphic points, and the poles $\tau_p$
to the zeros $\tau_0^-=\tau_p$, whilst the transformation~\eqref{eq:u-} ($u\to u_-$) sends the $\tau_0^+$ zeros to
holomorphic points, the $\tau_0^-$ zeros to the poles $\tau_p=\tau_0^-$, and the poles $\tau_p$ to the zeros
$\tau_0^+=\tau_p$.

The zeros of $u(\tau)$ require, in fact, an analysis similar to the poles, because, in the asymptotic study of the
degenerate third Painlev\'e equation via isomondromy deformations, the coefficients of the associated linear matrix
ODE are parametrized by the functions $u(\tau)$ and $1/u(\tau)$. Therefore, in Section 4 of \cite{KitVarZapiski2024},
we distinguished and studied the solutions of equation~\eqref{eq:dp3} which have sequences of zeros accumulating at
$\tau=0$: these solutions are characterized by the values of the branching parameter $\rho=\mi\varkappa$,
$\varkappa\in\mathbb{R}\setminus\{0\}$. The monodromy data of these solutions are defined in Theorem B.1
of \cite{KitVar2023}. These monodromy data contain the restrictions $g_{11}g_{22}\neq0$ and $|\mathrm{Im}\,a|<1$,
which are removed in \cite{KitVarZapiski2024} and in Theorem~\ref{th:restriction-on-Ima-removed}, respectively.
The result obtained in \cite{KitVarZapiski2024} states that there are two sequences of zeros of $u(\tau)$
accumulating at $\tau=0$: one sequence corresponds to the expansion~\eqref{eq:Taylor-zero-u+}, whilst the other
corresponds to the expansion~\eqref{eq:Taylor-zero-u-}. The location of the members of these sequences is described
in terms of the location of the zeros of the leading term of asymptotics of these solutions, namely, the zeros
of the leading term of asymptotics are surrounded by $\mathcal{D}_p$-like discs, and Theorem 4.2 of
\cite{KitVarZapiski2024} states that, for small enough $\tau$, each disc contains one, and only one,
zero and no poles of our solution. Equipped with this information regarding the zeros accumulating at the origin,
and the fact that the solutions that possess such sequences of zeros are the ``seed solutions'' of the
B\"acklund transformation(s) for the solutions studied in this section, we employ a one-to-one correspondence
argument between poles and zeros of the solutions related via the B\"acklund transformations, and arrive at
Theorem~\ref{th:poles-u-one-to-one-poles-asympt} formulated above.

As explained above, the discs containing the poles are precisely the same discs containing the zeros in
Theorem 4.2 of \cite{KitVarZapiski2024}; in particular, the parameter $\delta_d<2$ coincides for both discs.
In \cite{KitVarZapiski2024}, we stated that the largest possible value of $\delta_d$ could be $2$, but it
necessitates increasing the value of the $J$-like parameter in the definition of the radius of the discs,
depending on the monodromy data of the solution $u(\tau)$.
\end{proof}
\begin{remark}\label{rem:disk-radius}
In the case of the zeros, the estimate for $\delta_d$ is easy enough to obtain. Denote by $\tau_0$ and
$\tau_{\hat{0}}$, respectively, the zeros of $u(\tau)$ and $u_{as}(\tau)$,\footnote{\label{foot:uas} We denote by
$u_{as}(\tau)$ the level-one terms of the expansion~\eqref{app:eq:0-u-expansion} that coincide with the leading
term of asymptotics obtained in Theorem B.1 of \cite{KitVar2023}.} which belong to an open disc of radius
$R_d=\mathcal{O}\big(\tau_{\hat{0}}^{1+\delta_d}\big)$ centred at $\tau_{\hat{0}}$;
then, the expansion~\eqref{app:eq:0-u-expansion} can be presented in the form
$u(\tau)-u_{as}(\tau)=\mathcal{O}\big(\tau^3\big)$. Thus, we find that
$u(\tau_{\hat{0}})=\mathcal{O}\big(\tau_{\hat{0}}^3\big)$. The expansions~\eqref{eq:Taylor-zero-u+} and
\eqref{eq:Taylor-zero-u-} imply that
$\tau_{\hat{0}}-\tau_0=\mathcal{O}\big(\tau_{\hat{0}}^3\big)$ as $\tau_{\hat{0}}\to0$. Writing
$3=1+\delta_d$, it follows that $\delta_d=2$; however, we consider an open disc, so that, in fact, $\delta_d<2$.
The proof that each such disc, for small enough $\tau_{\hat{0}}$, actually contains one, and only one, zero of
$u(\tau)$ is given in \cite{KitVarZapiski2024}.
\hfill$\blacksquare$\end{remark}
\begin{remark}[\bf{Direct Proof of Theorem~\ref{th:poles-u-one-to-one-poles-asympt}}]\label{rem:directproof}
A direct proof that $\delta_d<2$ in the case of the poles is more involved, because one cannot simply
substitute poles of the solution or its asymptotics into the expansion~\eqref{app:eq:u-A-relation}
(see Appendix~\ref{app:subsec:super-generating-function}), because it would lead to a contradiction. In the case
of the poles, therefore, the estimate $\delta_d<2$, without reference to the corresponding estimate for the zeros,
manifests  differently.

We now consider a direct proof that the disc $\mathcal{D}_p$  contains, for large enough $p\in\mathbb{N}$,
one, and only one, pole of the solution $u(\tau)$, together with the corresponding estimate for $\delta_d$.
Consider the same integral that was used in the proof of Lemma 4.1 in \cite{KitVarZapiski2024}, and calculate
it using the Residue Theorem:
\begin{equation}\label{ineq:u-uas-integral}
I:=\frac{1}{2\pi\mi}\varointctrclockwise\limits_{\partial\mathcal{D}_p}
\tau\big(u(\tau)-u_{as}(\tau)\big)\,\md\tau=-\frac14\sum_{k=1}^n\tau_{\hat{p}_k}+\frac14\tau_p=
-\frac14\sum_{k=1}^n(\tau_{\hat{p}_k}-\tau_p)-(n-1)\frac14\tau_p,
\end{equation}
where $\tau_{\hat{p}_k}$ are the proposed poles, if any,\footnote{\label{foot:n=0} The number of poles, $n$, can be
equal to $0$, in which case, the sum vanishes.} of the solution $u(\tau)$, and $\tau_p$ is the pole of the
leading term of asymptotics $u_{as}(\tau)$ (cf. equation~\eqref{eq:A-varkappa-asympt-u}).
Then, we can estimate this integral from below as follows:
\begin{equation}\label{ineq:u-uas-integral-below}
|I|>\frac14(|n-1||\tau_p|-nR_p)=\frac{|\tau_p|}{4}\big(|n-1|-n|\tau_p|^{\delta_d}\big).
\end{equation}
In the event that we continue to follow the scheme delineated in \cite{KitVarZapiski2024}, we have to estimate
$I$ from above by using the correction term for the function $\tau\big(u(\tau)-u_{as}(\tau)\big)$. According to
equations~\eqref{app:eq:supergenfuncton} and \eqref{app:eq:u-A-relation} in
Appendix~\ref{app:subsec:super-generating-function}, this term equals $yA_1(x)$ (cf. equation~\eqref{app:eq:A1}),
where $y=\tau^{-\sigma}$ and $x=\tau^{2+\sigma}$, and, as a consequence of the symmetry $\sigma\to-\sigma$,
we can take $\sigma=-4\varrho=-2-4\mi\varkappa$.
The function $A_1(x)$ has a third-order pole; therefore, expanding it in a neigbourhood of $\tau=\tau_p\to0$,
one finds that
\begin{equation}\label{eq:A1-pole-estimate}
yA_1(x)\underset{\substack{\tau,\tau_p\in\mathcal{D}_p\\\tau\to0}}=
\mathcal{O}\left(\frac{\tau_p^5}{(\tau-\tau_p)^3}\right),
\quad
y=\tau^{2+4\mi\varkappa},\quad
x=\tau^{-4\mi\varkappa},\quad
\varkappa\in\mathbb{R}\setminus\{0\}.
\end{equation}
Using the estimate~\eqref{eq:A1-pole-estimate} and the fact that $\tau\in\partial\mathcal{D}_p$, so that
$|\tau-\tau_p|=R_p$, one finds
\begin{equation}\label{ineq:u-uas-integral-above}
|I|\underset{\substack{\tau,\tau_p\in\mathcal{D}_p\\\tau\to0}}<
\left|\mathcal{O}\left(\frac{|\tau_p|^5}{R_p^2}\right)\right|
\underset{\substack{\tau,\tau_p\in\mathcal{D}_p\\\tau\to0}}=
\left|\mathcal{O}\left(|\tau_p|^{3-2\delta_d}\right)\right|.
\end{equation}
Comparing the inequalities~\eqref{ineq:u-uas-integral-below} and \eqref{ineq:u-uas-integral-above}, we see that, in
order to prove that $n=1$, we must impose the condition $1<3-2\delta_d$, i.e., $\delta_d<1$, rather than $\delta_d<2$!
To reconcile this situation, we have to take for $u_{as}(\tau)$ the sequence of the first $m-1$ terms of the
asymptotic expansion (cf. Appendix~\ref{app:subsec:super-generating-function},
equations~\eqref{app:eq:supergenfuncton} and \eqref{app:eq:u-A-relation}). In this case,
the last omitted term is of the order
\begin{equation}\label{eq:An-pole-estimate}
y^mA_m(x)\underset{\substack{\tau,\tau_p\in\mathcal{D}_p\\\tau\to0}}
=\mathcal{O}\left(\frac{\tau_p^{3m+2}}{(\tau-\tau_p)^{m+2}}\right),
\end{equation}
so that, repeating \emph{verbatim} the estimate of $I$ for the case $m=1$, one arrives at
\begin{equation}\label{ineq:u-uas-integral-above-gen}
|I|<\left|\mathcal{O}\big(|\tau_p|^{2m+1-(m+1)\delta_d}\big)\right|\quad
\Rightarrow\quad
\delta_d<\frac{2m}{m+1}.
\end{equation}
Thus, for the poles, $\delta_d$ can be taken equal to any positive number less than $2$. The last fact implies that
the pole of the function $u(\tau)$ is located in the closed disc with radius
$R_p=J_p|\tau_p|^{1+\delta_d}$, with $\delta_d=2$; increasing, if necessary, $J_p$, we can assume that the pole of
$u(\tau)$ is located in some open disc $\mathcal{D}_p$.

The absence of zeros in $\mathcal{D}_p$ can be proved by using the Argument Principle; here, we should exploit
the fact that the leading term of asymptotics does not have zeros in some small enough cut neighbourhood of the
origin.
\hfill$\blacksquare$\end{remark}

There are special cases of the parameter of formal monodromy $a$ which are excluded in the formulation of
Theorem~\ref{th:through-poles}, and, subsequently, Theorem~\ref{th:poles-u-one-to-one-poles-asympt}
(cf. condition~\eqref{eq:conditions-a-varkappa}); however,
for these values of $a$, equation~\eqref{eq:dp3} still possesses solutions with poles accumulating at the origin.

The analogue of the asymptotic results stated in Theorem~\ref{th:through-poles} for the
condition~\eqref{eq:conditions-a-varkappa} is formulated in Theorems~\ref{th:through-poles-s0infty0} and
\ref{th:through-poles-s1infty0} below, because the monodromy data of the
corresponding solutions are different. At the same time, it is possible to combine both cases and formulate the
analogue of---the ``disc''---Theorem~\ref{th:poles-u-one-to-one-poles-asympt} as a single theorem, namely,
Theorem~\ref{th:disc-s0infty0s1infty0} (see below).

The problem with the vanishing denominators that occurs in the leading terms of the asymptotics
derived in Theorems~\ref{th:alt2B1asympt0p} and \ref{th:alt2B1asympt0p}\hspace{-1pt}${}^{\mathbf\prime}$,
the consideration of which is the starting point of this section, also takes place with the denominators of the
leading terms of the asymptotics presented in the respective items \pmb{$(1)$} of
Theorems~\ref{th:Asympt0-rho-eq-1pn+ia2} and \ref{th:Asympt0-rho-eq-1pn-ia2}. The analysis of the vanishing
denominators in the latter theorems completes the description of solutions with the poles accumulating at the origin,
and corresponds to the values of the parameter of formal monodromy excluded in Theorem~\ref{th:through-poles}.

Define the poles (cf. equation~\eqref{eq:asymptotic-poles} for $\tau_p$ with $k=1$)
\begin{equation}\label{eq:def-hat-poles}
\hat{\tau}_p=\exp\left(-\frac{\pi p}{2|\varkappa|}+\frac{\pi}{4\varkappa}
-\frac{\mi}{4\varkappa}\ln\frac{\hat{\omega}_1}{\hat{\omega}_{2}}\right),\qquad
\hat{\omega}_j=\hat{w}_j\vert_{\varrho=1/2+\mi\varkappa},\quad
j=1,2,
\end{equation}
where $\hat{w}_1$ and $\hat{w}_2$ are defined by equations~\eqref{eq:omega1} and \eqref{eq:omega2}, respectively.
If we introduce ``hats'', that is, $\tau_p\rightarrow\hat{\tau}_p $, $\mathcal{D}_p\rightarrow\hat{\mathcal{D}}_p$,
and $\mathcal{S}\rightarrow\hat{\mathcal{S}}$, then the entire discussion subsequent to
equation~\eqref{eq:asymptotic-poles} until Theorem~\ref{th:through-poles} applies without change to the
hat-variables. We are now in a position to formulate an analogue of Theorem~\ref{th:through-poles}.
\begin{theorem}\label{th:through-poles-s0infty0}
Let $(u(\tau), \varphi(\tau))$ be a solution of the system~\eqref{eq:dp3}, \eqref{eq:varphi} corresponding to the
monodromy data $(a,s_{0}^{0},s_{0}^{\infty},s_{1}^{\infty},g_{11},g_{12},g_{21},g_{22})$.
Suppose that
\begin{equation}\label{eqs:conditions-poles-th5.1}
a=2\varkappa+\mi(2n+1),\;\;\mathrm{where}\;
\varkappa\in\mathbb{R}\setminus\{0\},\;
n\in\mathbb{Z}_{\geqslant0},\quad
s_0^{\infty}=0,\quad
\mathrm{and}\quad
s_1^{\infty}\neq0.
\end{equation}
Then, $g_{21}\neq0$, and the remaining monodromy data are given by the following equations:
\begin{equation}\label{eqs:spec-monodromy-poles-th5.1-1}
s_0^0=-2\mi\cosh(2\pi\varkappa),\quad
g_{11}=-\mi\me^{-2\pi\varkappa}g_{21},\quad
g_{12}=-\frac{\me^{2\pi\varkappa}-\mi s_1^{\infty}g_{21}^2}{2\sinh(2\pi\varkappa)g_{21}},\quad
g_{22}=-\frac{\mi+\me^{2\pi\varkappa}s_1^{\infty}g_{21}^2}{2\sinh(2\pi\varkappa)g_{21}}.
\end{equation}
Let
\begin{align}
\hat{\omega}_{1}=&\left(\tfrac{\varepsilon b}2\right)^{-\mi\varkappa}\me^{\frac{\pi\mi}4+\pi\mi(n+1)}\,
\frac{2\pi}{n!}\,
\frac{\Gamma(1+2\mi\varkappa)}{\Gamma(1-2\mi\varkappa)}\,\frac{\me^{-5\pi\varkappa/2}}{s_1^{\infty}g_{21}},
\label{eq:omega-poles1}\\
\hat{\omega}_{2}=&\left(\tfrac{\varepsilon b}2\right)^{\mi\varkappa}\me^{-\frac{\pi\mi}4-\pi\mi(n+1)}\,
\frac{2\pi}{\Gamma(n+1-2\mi\varkappa)}\frac{\Gamma(1-2\mi\varkappa)}{\Gamma(1+2\mi\varkappa)}\,
\me^{-3\pi\varkappa/2}g_{21}.
\label{eq:omega-poles2}
\end{align}
Assume that $\hat{\mathcal{S}}$ is defined as $\mathcal{S}$ in \eqref{eq:S-D-def} with
$\tau_p\rightarrow\hat{\tau}_p $, $\mathcal{D}_p\rightarrow\hat{\mathcal{D}}_p$, and $\delta_d\in[0,2);$
then,
\begin{align}
u(\tau)\underset{\scriptscriptstyle\hat{\mathcal{S}}\ni\tau\to0}{=}
&-\frac{4\varepsilon\varkappa^2\hat{\omega}_1\hat{\omega}_2\Big(1+\mathcal{O}\big(\tau^{2-\delta_d}\big)\Big)}
{\tau\big(\hat{\omega}_1\tau^{-2\mi\varkappa}+\hat{\omega}_2\tau^{2\mi\varkappa}\big)^2},
\label{eq:Asympt0poles-spec1u-}\\
\me^{\mi\varphi (\tau)} \underset{\scriptscriptstyle\hat{\mathcal{S}}\ni\tau\to0}{=}
&\me^{-\pi\varkappa-\pi\mi(n+1)}\frac{2\pi}{\hat{\omega}_1\hat{\omega}_2}
\big(2\tau^2\big)^{-2n-1+2\mi\varkappa}
\Big(1+\mathcal{O}\big(\tau^{2-\delta_d}\big)\Big).\label{eq:Asympt0poles-spec1varphi-}
\end{align}
\end{theorem}
\begin{proof}
This theorem is a refined formulation of the results presented in item~\pmb{$(1)$} of
Theorem~\ref{th:Asympt0-rho-eq-1pn+ia2} for a specific choice of the formal monodromy $a$ given
in \eqref{eqs:conditions-poles-th5.1}. This formulation is required because in this, and only this, case
the denominator of the corresponding leading term of asymptotics of $u(\tau)$ (cf. equation~\eqref{eq:Asympt0spec1u-})
vanishes at the sequence of points $\hat{\tau}_p\to0$.
To see this, note that the denominator in equation~\eqref{eq:Asympt0spec1u-} vanishes at an infinite number of
points iff $\varrho=1/2+\mi\varkappa$, $\varkappa\in\mathbb{R}\setminus\{0\}$. On the other hand,
in item~\pmb{$(1)$} of Theorem~\ref{th:Asympt0-rho-eq-1pn+ia2}, the parameter $\varrho=1+n+\mi a/2$,
$n\in\mathbb{Z}_{\geqslant0}$. Solving the equation $\varrho=1+n+\mi a/2=1/2+\mi\varkappa$, one arrives at
the formula for the formal monodromy $a$ given in \eqref{eqs:conditions-poles-th5.1}. The formulae for the remaining
monodromy data given in the list~\eqref{eqs:spec-monodromy-poles-th5.1-1} are obtained from the corresponding
formulae~\eqref{eqs:spec-monodromy-th3.1-1} for this choice of $a$.

The formulae for the asymptotics~\eqref{eq:Asympt0poles-spec1u-} and \eqref{eq:Asympt0poles-spec1varphi-},
respectively, coincide with the corresponding asymptotics~\eqref{eq:Asympt0spec1u-} and
\eqref{eq:Asympt0spec1varphi-} for $\varrho=1/2+\mi\varkappa$, where, for brevity, we introduced the notation
$\hat{\omega}_j=\hat{w}_j|_{\varrho=1/2+\mi\varkappa}$, $j=1,2$.

The solution of the problem for the vanishing denominators is similar to that presented in
Theorem~\ref{th:through-poles}, and is solved by restricting the asymptotic formulae to the
multiply-connected domain $\hat{\mathcal{S}}$ defined in the theorem.
\end{proof}
To remove the restriction~\eqref{eq:conditions-a-varkappa} for non-positive values of $k$, we can refer to
Theorem~\ref{th:Asympt0-rho-eq-1pn-ia2}. In order to formulate the corresponding result
(see Theorem~\ref{th:through-poles-s1infty0} below), additional notation is necessary.

Define the poles (cf. equation~\eqref{eq:asymptotic-poles} for $\tau_p$ with $k=3$)
\begin{equation}\label{eq:def-tilde-poles}
\tilde{\tau}_p=\exp\left(-\frac{\pi p}{2|\varkappa|}+\frac{\pi}{4\varkappa}
+\frac{\mi}{4\varkappa}\ln\frac{\tilde{\omega}_3}{\tilde{\omega}_{4}}\right),\qquad
\tilde{\omega}_j=\tilde{w}_j\vert_{\varrho=1/2-\mi\varkappa},\quad
j=3,4,
\end{equation}
where $\tilde{w}_3$ and $\tilde{w}_4$ are defined by equations~\eqref{eq:omega3} and \eqref{eq:omega4}, respectively.
If we introduce ``tilde'' variables, that is, $\tau_p\rightarrow\tilde{\tau}_p $,
$\mathcal{D}_p\rightarrow\tilde{\mathcal{D}}_p$, and $\mathcal{S}\rightarrow\tilde{\mathcal{S}}$, then the entire
discussion subsequent to equation~\eqref{eq:asymptotic-poles}
until Theorem~\ref{th:through-poles} applies without change to the tilde-variables. We are now in a position
to formulate an analogue of Theorem~\ref{th:through-poles}.
\begin{theorem}\label{th:through-poles-s1infty0}
Let $(u(\tau), \varphi(\tau))$ be a solution of the system~\eqref{eq:dp3}, \eqref{eq:varphi} corresponding to the
monodromy data $(a,s_{0}^{0},s_{0}^{\infty},s_{1}^{\infty},g_{11},g_{12},g_{21},g_{22})$.
Suppose that
\begin{equation}\label{eqs:conditions-poles-th5.2}
a=2\varkappa-\mi(2n+1),\;\;\mathrm{where}\;
\varkappa\in\mathbb{R}\setminus\{0\},\;
n\in\mathbb{Z}_{\geqslant0},\quad
s_1^{\infty}=0,\quad
\mathrm{and}\quad
s_0^{\infty}\neq0.
\end{equation}
Then, $g_{12}\neq0$, and the remaining monodromy data are given by the following equations:
\begin{equation}\label{eqs:spec-monodromy-poles-th5.2-1}
s_0^0=-2\mi\cosh(2\pi\varkappa),\quad
g_{11}=\frac{s_0^{\infty}g_{12}^2\me^{-2\pi\varkappa}+\mi}{2\sinh(2\pi\varkappa) g_{12}},\quad
g_{21}=\frac{\mi s_0^{\infty}g_{12}^2\me^{-4\pi\varkappa}-\me^{2\pi\varkappa}}{2\sinh(\pi a)g_{12}},\quad
g_{22}=\mi\me^{-2\pi\varkappa}g_{12}.
\end{equation}
Let
\begin{align}
\tilde{\omega}_{3}=&\left(\tfrac{\varepsilon b}2\right)^{\mi\varkappa}\me^{-\frac{\pi\mi}4-\pi\mi n}\,
\frac{2\pi}{n!}\,\frac{\Gamma(1-2\mi\varkappa)}{\Gamma(1+2\mi\varkappa)}
\frac{\me^{3\pi\varkappa/2}}{s_0^{\infty}g_{12}},
\label{eq:omega-poles3}\\
\tilde{\omega}_{4}=&\left(\tfrac{\varepsilon b}2\right)^{-\mi\varkappa}
\me^{\frac{\pi\mi}4+\pi\mi(n+1)}
\frac{2\pi}{\Gamma(n+1+2\mi\varkappa)}\frac{\Gamma(1+2\mi\varkappa)}{\Gamma(1-2\mi\varkappa)}\,
\me^{-3\pi\varkappa/2}g_{12}.
\label{eq:omega-poles4}
\end{align}
Assume that $\tilde{\mathcal{S}}$ is defined as $\mathcal{S}$ in \eqref{eq:S-D-def} with
$\tau_p\rightarrow\tilde{\tau}_p $, $\mathcal{D}_p\rightarrow\tilde{\mathcal{D}}_p$, and $\delta_d\in[0,2);$
then,
\begin{align}
u(\tau)\underset{\scriptscriptstyle\tilde{\mathcal{S}}\ni\tau\to0}{=}
&-\frac{4\varepsilon\varkappa^2\tilde{\omega}_3\tilde{\omega}_4\Big(1+\mathcal{O}\big(\tau^{2-\delta_d}\big)\Big)}
{\tau\big(\tilde{\omega}_3\tau^{2\mi\varkappa}+\tilde{\omega}_4\tau^{-2\mi\varkappa}\big)^2},
\label{eq:Asympt0poles-spec1u+}\\
\me^{\mi\varphi(\tau)}\underset{\scriptscriptstyle\tilde{\mathcal{S}}\ni\tau \to 0}{=}
&\me^{\pi\varkappa+\pi\mi(n+1)}\frac{\tilde{\omega}_3\tilde{\omega}_4}{2\pi}\big(2\tau^2\big)^{2n+1+2\mi\varkappa}
\Big(1+\mathcal{O}\big(\tau^{2-\delta_d}\big)\Big).\label{eq:Asympt0poles-spec1varphi+}
\end{align}
\end{theorem}
\begin{proof}
The proof of this theorem is similar to the proof of Theorem~\ref{th:through-poles-s0infty0}; more precisely,
it is a refined formulation of the results presented in item~\pmb{$(1)$} of
Theorem~\ref{th:Asympt0-rho-eq-1pn-ia2} for the formal monodromy $a$ given in \eqref{eqs:conditions-poles-th5.2}.
In this case, the parameter $\varkappa$ is defined via the relation
$\varrho=1+n-\mi a/2=1/2-\mi\varkappa$, $\varkappa\in\mathbb{R}\setminus\{0\}$.

%The proof of this theorem is similar to the proof of Theorem~\ref{th:through-poles-s0infty0}.
%This theorem is a refined formulation of the results presented in item~\pmb{$(1)$} of
%Theorem~\ref{th:Asympt0-rho-eq-1pn-ia2} for a specific choice of the formal monodromy $a$ given
%in \eqref{eqs:conditions-poles-th5.2}. This formulation is required because in this, and only this, case
%the denominator of the corresponding leading term of asymptotics of $u(\tau)$ (cf. equation~\eqref{eq:Asympt0spec1u+})
%vanishes at the sequence of points $\tilde{\tau}_p\to0$.
%To see this, note that the denominator in equation~\eqref{eq:Asympt0spec1u-} vanishes at an infinite number of
%points iff $\varrho=1/2-\mi\varkappa$, $\varkappa\in\mathbb{R}\setminus\{0\}$. On the other hand,
%in item~\pmb{$(1)$} of Theorem~\ref{th:Asympt0-rho-eq-1pn-ia2}, the parameter $\varrho=1+n-\mi a/2$,
%$n\in\mathbb{Z}_{\geqslant0}$. Solving the equation $\varrho=1+n-\mi a/2=1/2-\mi\varkappa$, one arrives at
%the formula for the formal monodromy $a$ given in \eqref{eqs:conditions-poles-th5.2}. The formulae for the remaining
%monodromy data given in the list~\eqref{eqs:spec-monodromy-poles-th5.2-1} are obtained from the corresponding
%formulae~\eqref{eqs:spec-monodromy-th3.1p-1} for this choice of $a$.

The formulae for the asymptotics~\eqref{eq:Asympt0poles-spec1u+} and \eqref{eq:Asympt0poles-spec1varphi+}, respectively,
coincide with the corresponding asymptotics~\eqref{eq:Asympt0spec1u+} and
\eqref{eq:Asympt0spec1varphi+} for $\varrho=1/2-\mi\varkappa$, where, for brevity, we introduced the notation
$\tilde{\omega}_j=\tilde{w}_j|_{\varrho=1/2-\mi\varkappa}$, $j=3,4$.

The solution of the problem for the vanishing denominators is similar to
that presented in Theorem~\ref{th:through-poles}, and is solved by restricting the asymptotic formulae to the
multiply-connected domain $\tilde{\mathcal{S}}$ defined in the theorem.
\end{proof}

We conclude this section by formulating two statements regarding the properties of the zeros and poles of the
solutions considered in Theorems~\ref{th:through-poles-s0infty0} and \ref{th:through-poles-s1infty0}.
These properties are precisely the same as those formulated in Corollary~\ref{cor:nozeros-nopoles} and
Theorem~\ref{th:poles-u-one-to-one-poles-asympt} for the solutions presented in Theorem~\ref{th:through-poles}.
The proofs for the statements formulated below do not rely on any particular parametrization(s) for the zeros and poles
in terms of the monodromy data, and, therefore, coincide with the proofs of the corresponding statements given above.
\begin{corollary}\label{cor:nozeros-nopoles-spec}
Let $(u(\tau),\varphi(\tau))$ be a solution of the system~\eqref{eq:dp3}, \eqref{eq:varphi}
corresponding to the monodromy data $(a,s_{0}^{0},s_{0}^{\infty},s_{1}^{\infty},g_{11},g_{12},g_{21},g_{22})$.
Suppose that the conditions~\eqref{eqs:conditions-poles-th5.1} of Theorem~{\rm\ref{th:through-poles-s0infty0}} or
the conditions~\eqref{eqs:conditions-poles-th5.2} of Theorem~{\rm\ref{th:through-poles-s1infty0}} are satisfied,
which imply the equations~\eqref{eqs:spec-monodromy-poles-th5.1-1} or
\eqref{eqs:spec-monodromy-poles-th5.2-1}, respectively.

For $\epsilon>0$, define
\begin{equation*}\label{eq:hat-tilde-Sepsilon-def}
\hat{\mathcal{S}}_{\epsilon}:=\left\{\tau\in\hat{\mathcal{S}}: |\tau|<\epsilon\right\}\quad\mathrm{and}\quad
\tilde{\mathcal{S}}_{\epsilon}:=\left\{\tau\in\tilde{\mathcal{S}}: |\tau|<\epsilon\right\}.
\end{equation*}
Then, there exists $\epsilon>0$ such that $u(\tau)$ and $\me^{\mi\varphi(\tau)}$ have neither  zeros nor poles
in $\hat{\mathcal{S}}_{\epsilon}$ or $\tilde{\mathcal{S}}_{\epsilon}$, respectively.
\end{corollary}
\begin{theorem}\label{th:disc-s0infty0s1infty0}
For large enough $p\in\mathbb{N}$, each disc $\hat{\mathcal{D}}_p$ or $\tilde{\mathcal{D}}_p$ contains one,
and only one, pole and no zeros of the solution $u(\tau)$ corresponding to the monodromy data defined in
Theorem~{\rm\ref{th:through-poles-s0infty0}} or Theorem~{\rm\ref{th:through-poles-s1infty0}}, respectively.
\end{theorem}
%%%%%%%%%%%%%%%%%%%%%%%%%%%%%%%%%%%%%%%%%%%%%%%%%%%%%%%%%%%%%%%%%%%%%%%%%%%%%%%%%%%%%%%%%%%%%%%%%%%%%%%%%%%%%%%%%%%%
%%%%%%%%%%%%%%%%%%%%%%%%%%%%%%%%%%%%%%%%%%%%%%%%%%%%%%%%%%%%%%%%%%%%%%%%%%%%%%%%%%%%%%%%%%%%%%%%%%%%%%%%%%%%%%%%%%%%
\section{Meromorphic Solutions}\label{sec:Meromorphic}
Substituting the pole-like expansion for $u(\tau)$ into equation~\eqref{eq:dp3}, one immediately observes that
solutions of this equation cannot have a pole at the origin; thus, all meromorphic solutions admit a Taylor-series
expansion centred at $\tau=0$.
The first three theorems of this section concern the parametrization via the monodromy data of meromorphic solutions
vanishing at the origin, while the fourth theorem deals with non-vanishing meromorphic solutions.
\begin{theorem}\label{th:vanishing-meromorphic-s0s1EQ0}
Assume that $s_0^{\infty}=s_1^{\infty}=0;$ then, $a\neq\mi k$, $k\in\mathbb{Z}$, and
\begin{equation}\label{eq:mondata:s0s1EQ0}
s_0^0=2\mi\cosh(\pi a),\quad
g_{11}=\mi\me^{-\pi a}g_{21},\quad
g_{22}=-\mi\me^{-\pi a}g_{12},\quad
g_{12}g_{21}=-\frac{\me^{\pi a}}{2\sinh(\pi a)}.
\end{equation}
The corresponding functions $u(\tau)$ and $\varphi(\tau)$ are meromorphic with the following Taylor-series expansions
centred at $\tau=0$,
\begin{align}
\varepsilon u(\tau)=&\sum_{k=1}^{\infty}\tilde{b}_{2k-1,0}\tau^{2k-1},\label{eq:u-Taylor-odd}\\
\me^{-\mi\varphi(\tau)}=&\,
\frac{\mi\me^{\pi a}\Gamma^2(1-\mi a)}{2\pi ag_{12}^2}\left(\frac{\varepsilon b}{4}\right)^{\mi a}
\exp\left(\mi\sum_{N=1}^{\infty}\tilde{p}_N\tau^{2N}\right),\label{eq:phi-Taylor-odd}
\end{align}
where $\tilde{b}_{2k-1,0}=b_{2k-1,0}|_{\sigma=-2\mi a}$ and $\tilde{p}_N=p_N|_{\sigma=-2\mi a}$, with the coefficients
$b_{2k-1,0}$ and $p_N$ defined in Appendix~{\rm\ref{app:sec:full0expansion}} and
equations~\eqref{eq:pN}, \eqref{eqs:MkN}, respectively.
\end{theorem}
\begin{proof}
Substituting the conditions $s_0^{\infty}=s_1^{\infty}=0$ into
equations~\eqref{eq:monodromy:s}--\eqref{eq:monodromy:detG}, one finds that $a\neq\mi k$, $k\in\mathbb{Z}$, and
shows that the monodromy data satisfy the conditions~\eqref{eq:mondata:s0s1EQ0}.
We now address the expansion~\eqref{app:eq:0-u-expansion} and equations~\eqref{app:eqs:betas} for $b_{1,\pm1}$.
Using the explicit expressions for $\varpi_n(\pm\rho)$, $n=1,2$, given in \cite{KitVar2004}
(see, also, \cite{KitVar2023}), we get $b_{1,\pm1}=0$; thus, we arrive at the expansion~\eqref{eq:u-Taylor-odd}.
The value for $\me^{-\mi\varphi(0)}$ in equation~\eqref{eq:phi-Taylor-odd} is obtained from the general asymptotics
as $\tau\to0$ for $\me^{-\mi\varphi(\tau)}$ given in Appendix B, Theorem B.1 of \cite{KitVar2023}, and the
Taylor series in the argument of the exponential function in equation~\eqref{eq:phi-Taylor-odd} is derived with
the aid of equation~\eqref{eq:varphi}.
\end{proof}
\begin{theorem}\label{th:vanishing-meromorphicIMa>0}
Assume that $s_1^{\infty}=0$, $a=\mi(n-1/2)$, $n\in\mathbb{N}$, $s_0^{\infty}\in\mathbb{C}$, and
$g_{12}\in\mathbb{C}\setminus\{0\};$ then,
\begin{equation}\label{eq:mondata:s1EQ0}
s_0^0=0,\quad
g_{11}=\frac{(-1)^n-s_0^{\infty}g_{12}^2}{2g_{12}},\quad
g_{21}=-\frac{1+(-1)^ns_0^{\infty}g_{12}^2}{2g_{12}},\quad
g_{22}=(-1)^ng_{12}\neq0.
\end{equation}
The corresponding functions $u(\tau)$ and $\varphi(\tau)$ are meromorphic with the following Taylor-series
expansions centred at $\tau=0$,
\begin{align}
\varepsilon u(\tau)&\underset{\tau\to0}{=}\sum_{m=1}^{\infty}c_m\tau^m=
\sum_{k=1}^n \hat{b}_{2k-1,0}\tau^{2k-1}+\hat{b}_{1,1}\tau^{2n}+\mathcal{O}\big(\tau^{2n+1}\big),
\label{eq:Taylor-u-0at0-s0infty}\\
\me^{-\mi\varphi(\tau)}&\underset{\tau\to0}{=}\frac{(-1)^{n}\,\mi\,
\big((2n-1)!!\big)^2}{2(\varepsilon b)^{n-1/2}g_{12}^2(2n-1)}
\exp\left(-(2n-1)\sum_{n=1}^{\infty}\xi_n\frac{\tau^n}{n}\right),
\label{eq:Taylor-phi-u-0at0-s0infty}
\end{align}
where
\begin{gather*}
c_m=\sum_{\substack{2k-1+l(2n-1)=m\\
k\geqslant1,\;0\leqslant l\leqslant k}} \hat{b}_{2k-1,l},\qquad\quad\left.
\hat{b}_{2k-1,l}=b_{2k-1,l} \vphantom{M^{M^{M^{M}}}}
\right\vert_{\begin{subarray}{l} \scriptscriptstyle{a=\mi(n-1/2)}\\
\scriptscriptstyle{\sigma=2n-1}\end{subarray}} \, \, , \label{eq:u-cm-meromorphic-Ima-positive}\\
\xi_n=\sum_{k=1}^n\left(\frac{2a}{\varepsilon b}\right)^k\sum_{m_i\in M_{k,n}}
\frac{(m_1 + \dotsb + m_n)!}{m_1! \, \dotsb \, m_n!}\prod_{i=1}^n(c_{i+1})^{m_i},
\label{eq:phi-xi-meromorphic-Ima-positive}
\end{gather*}
where the summation set $M_{k,n}$ coincides with $M_{k,N}$ for $N=n$ {\rm(}cf. equations~\eqref{eqs:MkN}{\rm)},
and the coefficients $b_{2k-1,l}$ are defined in Appendix~{\rm\ref{app:subsec:error-correction-terms}}. For
$1\leqslant l\leqslant k$, these coefficients depend on $b$ and $b_{1,1};$ the parameter $b$ and the coefficient
$b_{1,1}$ must be modified as follows: $b\to\varepsilon b$ and $b_{1,1}\to\hat{b}_{1,1}$, where
\begin{equation}\label{eq:hat-b11}
\hat{b}_{1,1}=\me^{\frac{3\pi\mi}{4}}\me^{-\frac{\pi\mi n}{2}}\big(\varepsilon b\big)^{n+\frac12}
\frac{2^{2n}s_0^{\infty}g_{12}^2}{\sqrt{2\pi}\big((2n-1)!!\big)^3}.
\end{equation}
In fact, $\hat{b}_{1,1}$ is the first coefficient in the Taylor series that depends on the monodromy data.
\end{theorem}
\begin{proof}
This is a special case of Theorem~\ref{th:Asympt0-rho-eq-1pn-ia2}, item~\pmb{$(3)$}~%\ref{Th4.2case3}
for $\sigma=-2\mi a$ and $a=\mi(n-1/2)$.
\end{proof}
\begin{theorem}\label{th:vanishing-meromorphicIMa<0}
Assume that $s_0^{\infty}=0$, $a=-\mi(n-1/2)$, $n\in\mathbb{N}$, $s_1^{\infty}\in\mathbb{C}$, and
$g_{21}\in\mathbb{C}\setminus\{0\};$ then,
\begin{equation}\label{eq:mondata:s0EQ0}
s_0^0=0,\quad
g_{22}=\frac{(-1)^n-s_1^{\infty}g_{21}^2}{2g_{21}},\quad
g_{12}=-\frac{1+(-1)^ns_1^{\infty}g_{21}^2}{2g_{21}},\quad
g_{11}=(-1)^ng_{21}\neq0.
\end{equation}
The corresponding functions $u(\tau)$ and $\varphi(\tau)$ are meromorphic with the following Taylor-series
expansions centred at $\tau=0$,
\begin{align}
\varepsilon u(\tau)&\underset{\tau\to0}{=}\sum_{m=1}^{\infty}d_m\tau^m=
\sum_{k=1}^n \check{b}_{2k-1,0}\tau^{2k-1}+\check{b}_{1,-1}\tau^{2n}+\mathcal{O}\big(\tau^{2n+1}\big),
\label{eq:Taylor-u-0at0-s1infty}\\
\me^{\mi\varphi(\tau)}&\underset{\tau\to0}{=}
\frac{(-1)^{n}\,\mi\,\big((2n-1)!!\big)^2}{2(\varepsilon b)^{n-1/2}g_{21}^2(2n-1)}
\exp\left(-(2n-1)\sum_{n=1}^{\infty}\nu_n\frac{\tau^n}{n}\right),
\label{eq:Taylor-phi-u-0at0-s1infty}
\end{align}
where
\begin{gather*}
d_m=\sum_{\substack{2k-1+l(2n-1)=m \\
k\geqslant1,\;0\leqslant l\leqslant k}}\check{b}_{2k-1,-l},\qquad\quad\left.
\check{b}_{2k-1,l}=b_{2k-1,l} \vphantom{M^{M^{M^{M}}}}
\right\vert_{\begin{subarray}{l}\scriptscriptstyle{a=-\mi(n-1/2)}\\
\scriptscriptstyle{\sigma=-(2n-1)}\end{subarray}} \, \, ,\label{eq:u-dm-meromorphic-Ima-negative}\\
\nu_n=\sum_{k=1}^n\left(\frac{2a}{\varepsilon b}\right)^k\sum_{m_i\in M_{k,n}}
\frac{(m_1+\ldots+m_n)!}{m_1!\cdot\ldots\cdot m_n!}\prod_{i=1}^n(d_{i+1})^{m_i},
\label{eq:phi-nu-meromorphic-Ima-negative}
\end{gather*}
where the summation set $M_{k,n}$ coincides with $M_{k,N}$ for $N=n$ {\rm(}cf. equations~\eqref{eqs:MkN}{\rm)},
and the coefficients $b_{2k-1,-l}$ are defined in Appendix~{\rm\ref{app:subsec:error-correction-terms}}.
For $1\leqslant l\leqslant k$, these coefficients depend on $b$ and $b_{1,-1};$ the parameter $b$ and
the coefficient $b_{1,-1}$ must be modified as follows: $b\to\varepsilon b$ and $b_{1,-1}\to\check{b}_{1,-1}$,
where
\begin{equation}\label{eq:check-b1-1}
\check{b}_{1,-1}=\me^{-\frac{3\pi\mi}{4}}\me^{\frac{\pi\mi n}{2}}\big(\varepsilon b\big)^{n+\frac12}
\frac{2^{2n}s_1^{\infty}g_{21}^2}{\sqrt{2\pi}\big((2n-1)!!\big)^3}.
\end{equation}
In fact, $\check{b}_{1,-1}$ is the first coefficient in the Taylor series that depends on the monodromy data.
\end{theorem}
\begin{proof}
This is a special case of Theorem~\ref{th:Asympt0-rho-eq-1pn+ia2}, item~\pmb{$(3)$}~%\ref{Th4.1case3}
for $\sigma=-2\mi a$ and $a=-\mi(n-1/2)$.
\end{proof}
\begin{remark}\label{rem:meromorphic}
The solutions presented in Theorem~\ref{th:vanishing-meromorphic-s0s1EQ0} were considered in \cite{KanekoOhyama2013}.
Theorems~\ref{th:vanishing-meromorphic-s0s1EQ0}--\ref{th:vanishing-meromorphicIMa<0} describe all
meromorphic solutions of equation~\eqref{eq:dp3} for $a\in\mathbb{C}$ that vanish at the origin.
This fact follows from the local analysis
presented in Lemma 2.1 of \cite{KitSIGMA2019}. For the case $a\neq\mi k$, $k\in\mathbb{Z}$, this can be
deduced independently from the results presented in Sections~\ref{sec:general}--\ref{sec:logarithm}.
The monodromy data for the vanishing solutions presented in
Theorem~\ref{th:vanishing-meromorphic-s0s1EQ0} were calculated by another method---based on the odd symmetry of
the solutions---in \cite{KitSIGMA2019} (see Proposition 7.1 in \cite{KitSIGMA2019}); however, the corresponding
asymptotic formula for the function $\varphi(\tau)$ was not obtained in \cite{KitSIGMA2019}.

For $n=1$, the monodromy data for the one-parameter family of solutions $u(\tau)$ vanishing at the origin presented
in Theorems~\ref{th:vanishing-meromorphicIMa>0} and \ref{th:vanishing-meromorphicIMa<0} were calculated by an alternative
method (via B\"acklund transformations applied to non-vanishing meromorphic solutions) in \cite{KitVar2023}; however,
the corresponding asymptotics for $\varphi(\tau)$ was not obtained in \cite{KitVar2023}.
\hfill$\blacksquare$\end{remark}
\begin{theorem}\label{th:nonvanishing-meromorphic}
If $u(\tau)$ is a meromorphic solution of equation~\eqref{eq:dp3} with $u(0)\neq0$, then
$\sigma=4\rho=1$, and thus $s_0^0=0$.
Depending on the values of the Stokes multipliers at the point at infinity, the Taylor-series expansion centred at
$\tau=0$ of the function $u(\tau)$ and the corresponding asymptotics of $\varphi(\tau)$ can be parametrized in terms
of the monodromy data in one of the following ways:
\begin{enumerate}
\item[\pmb{$(1)$}]
$s_0^{\infty}s_1^{\infty}\neq0\;\;\Rightarrow\;\;g_{11}\neq\pm g_{21}$, $g_{22}\neq\pm g_{12}$, and
$a\neq\mi(k+1/2)$, $k\in\mathbb{Z};$ thus,
\begin{gather}
\varepsilon u(\tau)\underset{\tau\to0}{=}\sum_{m=0}^{\infty}f_m\tau^m
=\tilde{b}_{1,-1}+(\tilde{b}_{1,0}+\tilde{b}_{3,-2})\tau+(\tilde{b}_{1,1}+\tilde{b}_{3,-1}
+\tilde{b}_{5,-3})\tau^2+\mathcal{O}\big(\tau^3\big),\label{eq:non-vanishing-meromorphic-u-general}\\
f_m=\sum_{\substack{2k-1+l=m\\-k\leqslant l\leqslant k}}\tilde{b}_{2k-1,l},\qquad
\tilde{b}_{2k-1,l}=b_{2k-1,l}\lvert_{\sigma=1,b\to\varepsilon b},\quad
k\geqslant1,\;\; |l|\leqslant k, \label{eq:fm-def}
\end{gather}
where the coefficients $b_{2k-1,l}$ are defined in Appendix~{\rm\ref{app:sec:full0expansion};} in particular,
one finds that
\begin{equation}\label{eq:b10-b11-nonvanishing-meromorphic-gen}
\tilde{b}_{1,0}=2a\varepsilon b\;\;\mathrm{and}\;\;\tilde{b}_{1,1}=b^2(a^2+1/4)/\tilde{b}_{1,-1}.
\end{equation}
The monodromy parametrization of the leading coefficient reads:
\begin{equation}\label{eq:b1-1-nonvanishing-meromorphic-gen}
\tilde{b}_{1,-1}=\left(\frac{\varepsilon b}{2}\right)^{1/2}\frac{\me^{\frac{\pi a}{2}}}{2\pi}
\Gamma\left(\frac34-\mi\frac{a}{2}\right)\Gamma\left(\frac34+\mi\frac{a}{2}\right)
(g_{11}+g_{21})(g_{12}+g_{22}),
\end{equation}
\begin{gather}
\me^{\mi\varphi(\tau)}\underset{\tau\to0}{=}\me^{-\frac{3\pi\mi}4}
\frac{\Gamma\left(\frac34-\frac{\mi a}{2}\right)}{\Gamma\left(\frac34+\frac{\mi a}{2}\right)}
\left(\frac{g_{12}+g_{22}}{g_{11}+g_{21}}\right)(2\tau^2)^{\mi a}
\exp\left(\mi\frac{\varepsilon b}{f_0}\left(\tau+\sum_{n=1}^{\infty}\eta_n\frac{\tau^{n+1}}{n+1}\right)\right),
\label{eq:varphi-nonvanishing-meromorphic-gen}\\
\eta_n=\sum_{k=1}^n\frac{(-1)^k}{f_0^k}\sum_{m_i\in M_{k,n}}
\frac{(m_1+\ldots+m_n)!}{m_1!\cdot\ldots\cdot m_n!}\prod_{i=1}^n(f_i)^{m_i},
\label{eq:eta-n-nonvanishing-meromorphic-gen}
\end{gather}
where the summation set $M_{k,n}$ coincides with $M_{k,N}$ for $N=n$ {\rm(}cf. equations~\eqref{eqs:MkN}{\rm)}, and
the numbers $f_i$ are defined by the second equality in equation~\eqref{eq:non-vanishing-meromorphic-u-general}{\rm;}
in particular, $f_0=\tilde{b}_{1,-1}$.
\item[\pmb{$(2)$}]
$s_0^{\infty}=0$,\quad $a=\mi(m-1/2)$, $m\in\mathbb{N}$.
\subitem\pmb{$(2.1)$} $a=\mi(2n+3/2)$, $n\in\mathbb{Z}_{\geqslant0}$,\quad
$s_1^{\infty},g_{21}\in\mathbb{C}\setminus\{0\}$,
\begin{equation*}\label{eqs:mondata-non-vanishing-meromorphic-s0inftyEQ0type1}
g_{11}=-g_{21},\quad
g_{12}=\frac{s_1^{\infty}g_{21}^2-1}{2g_{21}},\quad
g_{22}=-\frac{s_1^{\infty}g_{21}^2+1}{2g_{21}}.
\end{equation*}
The function $\varepsilon u(\tau)$ is given by---the same---equations~\eqref{eq:non-vanishing-meromorphic-u-general},
\eqref{eq:fm-def}, and \eqref{eq:b10-b11-nonvanishing-meromorphic-gen}, with $a=\mi(2n+3/2)$.
Equations~\eqref{eq:b1-1-nonvanishing-meromorphic-gen} and \eqref{eq:varphi-nonvanishing-meromorphic-gen} should be
changed to
\begin{gather*}
\tilde{b}_{1,-1}=-\frac{\sqrt{2\pi}}{4}\frac{\sqrt{\varepsilon b}}{s_1^{\infty}g_{21}^2}\frac{(2n+1)!!}{(2n)!!},
\label{eq:b1-1-nonvanishing-meromorphic-s0inftyEQ0}\\
\me^{\mi\varphi(\tau)}=\me^{-\frac{\pi\mi}{4}}\frac{(-1)^n(2n+1)!s_1^{\infty}}{\sqrt{2\pi}(2\tau)^{4n+3}}
\exp\left(\mi\frac{\varepsilon b}{f_0}\left(\tau+\sum_{n=1}^{\infty}\eta_n\frac{\tau^{n+1}}{n+1}\right)\right),
\label{eq:phi-nonvanishing-meromorphic-s0inftyEQ0}
\end{gather*}
where the coefficients $\eta_n$ are calculated via equation~\eqref{eq:eta-n-nonvanishing-meromorphic-gen}, and
the coefficients $f_m$ are given by equation~\eqref{eq:fm-def} accompanied with the same specialization for
$\tilde{b}_{2k-1,l}$ explained in the previous sentence, i.e., $a=\mi(2n+3/2)$.
\subitem\pmb{$(2.2)$}
$a=\mi(2n+1/2)$, $n\in\mathbb{Z}_{\geqslant0}$,\quad
$s_1^{\infty},g_{21}\in\mathbb{C}\setminus\{0\}$,
\begin{equation*}\label{eqs:mondata-non-vanishing-meromorphic-s0inftyEQ0type2}
g_{11}=g_{21},\quad
g_{12}=-\frac{1+s_1^{\infty}g_{21}^2}{2g_{21}},\quad
g_{22}=\frac{1-s_1^{\infty}g_{21}^2}{2g_{21}}.
\end{equation*}
The function $\varepsilon u(\tau)$ is given by---the same---equations~\eqref{eq:non-vanishing-meromorphic-u-general},
\eqref{eq:fm-def}, and \eqref{eq:b10-b11-nonvanishing-meromorphic-gen}, with $a=\mi(2n+1/2)$.
Equations~\eqref{eq:b1-1-nonvanishing-meromorphic-gen} and \eqref{eq:varphi-nonvanishing-meromorphic-gen} should be
changed to~\footnote{\label{foot:-1!!=0}
In equations~\eqref{eq:b1-1-nonvanishing-meromorphic-s0inftyEQ0type2} and
\eqref{eq:b1-1-nonvanishing-meromorphic-s1inftyEQ0type} for the case $n=0$, it is assumed that $(-1)!!=1$.}
\begin{gather}
\tilde{b}_{1,-1}=-4\me^{\frac{\pi\mi}{4}}\frac{\sqrt{\varepsilon b}}{\sqrt{2\pi}}\frac{(2n)!!}{(2n-1)!!}s_1^{\infty}g_{21}^2,
\label{eq:b1-1-nonvanishing-meromorphic-s0inftyEQ0type2}\\
\me^{\mi\varphi(\tau)}=\me^{\frac{\pi\mi}{4}}\frac{(-1)^n(2n)!s_1^{\infty}}{\sqrt{2\pi}(2\tau)^{4n+1}}
\exp\left(\mi\frac{\varepsilon b}{f_0}\left(\tau+\sum_{n=1}^{\infty}\eta_n\frac{\tau^{n+1}}{n+1}\right)\right),
\label{eq:phi-nonvanishing-meromorphic-s0inftyEQ0type2}\nonumber
\end{gather}
where the coefficients $\eta_n$ are calculated via equation~\eqref{eq:eta-n-nonvanishing-meromorphic-gen}, and
the coefficients $f_m$ are given by equation~\eqref{eq:fm-def} accompanied
with the same specialization for $\tilde{b}_{2k-1,l}$ explained in the previous sentence, i.e., $a=\mi(2n+1/2)$.
\item[\pmb{$(3)$}]
$s_1^{\infty}=0$,\quad $a=-\mi(m-1/2)$, $m\in\mathbb{N}$.
\subitem\pmb{$(3.1)$}
$a=-\mi(2n+3/2)$, $n\in\mathbb{Z}_{\geqslant0}$,\quad
$s_0^{\infty},g_{12}\in\mathbb{C}\setminus\{0\}$,
\begin{equation*}\label{eqs:mondata-non-vanishing-meromorphic-s1inftyEQ0type1}
g_{22}=-g_{12},\quad
g_{21}=\frac{s_0^{\infty}g_{12}^2-1}{2g_{12}},\quad
g_{11}=-\frac{s_0^{\infty}g_{12}^2+1}{2g_{12}}.
\end{equation*}
The function $\varepsilon u(\tau)$ is given by---the same---equations~\eqref{eq:non-vanishing-meromorphic-u-general},
\eqref{eq:fm-def}, and \eqref{eq:b10-b11-nonvanishing-meromorphic-gen}, with $a=-\mi(2n+3/2)$.
Equations~\eqref{eq:b1-1-nonvanishing-meromorphic-gen} and \eqref{eq:varphi-nonvanishing-meromorphic-gen} should be
changed to
\begin{gather*}
\tilde{b}_{1,-1}=\me^{-\frac{\pi\mi}{4}}\frac{\sqrt{2\pi}}{4}\frac{\sqrt{\varepsilon b}}{s_0^{\infty}g_{12}^2}
\frac{(2n+1)!!}{(2n)!!}, \label{eq:b1-1-nonvanishing-meromorphic-s1inftyEQ0}\\
\me^{\mi\varphi(\tau)}=\me^{-\frac{\pi\mi}{4}}\frac{\sqrt{2\pi}(-1)^n(2\tau)^{4n+3}}{(2n+1)!s_0^{\infty}}%{\sqrt{2\pi}(2\tau)^{4n+3}}
\exp\left(\mi\frac{\varepsilon b}{f_0}\left(\tau+\sum_{n=1}^{\infty}\eta_n\frac{\tau^{n+1}}{n+1}\right)\right),
\label{eq:phi-nonvanishing-meromorphic-s1inftyEQ0}
\end{gather*}
where the coefficients $\eta_n$ are calculated via equation~\eqref{eq:eta-n-nonvanishing-meromorphic-gen}, and
the coefficients $f_m$ are given by equation~\eqref{eq:fm-def} accompanied
with the same specialization for $\tilde{b}_{2k-1,l}$ explained in the previous sentence, i.e., $a=-\mi(2n+3/2)$.
\subitem\pmb{$(3.2)$}
$a=-\mi(2n+1/2)$, $n\in\mathbb{Z}_{\geqslant0}$,\quad
$s_0^{\infty},g_{12}\in\mathbb{C}\setminus\{0\}$,
\begin{equation*}\label{eqs:mondata-non-vanishing-meromorphic-s1inftyEQ0type2}
g_{22}=g_{12},\quad
g_{21}=-\frac{1+s_0^{\infty}g_{12}^2}{2g_{12}},\quad
g_{11}=\frac{1-s_0^{\infty}g_{12}^2}{2g_{12}}.
\end{equation*}
The function $\varepsilon u(\tau)$ is given by---the same---equations~\eqref{eq:non-vanishing-meromorphic-u-general},
\eqref{eq:fm-def}, and \eqref{eq:b10-b11-nonvanishing-meromorphic-gen}, with $a=-\mi(2n+1/2)$.
Equations~\eqref{eq:b1-1-nonvanishing-meromorphic-gen} and \eqref{eq:varphi-nonvanishing-meromorphic-gen} should be
changed to~\footref{foot:-1!!=0}
\begin{gather}
\tilde{b}_{1,-1}=\me^{\frac{3\pi\mi}{4}}\frac{\sqrt{\varepsilon b}}{\sqrt{2\pi}}
\frac{(2n)!!}{(2n-1)!!}s_0^{\infty}g_{12}^2,
\label{eq:b1-1-nonvanishing-meromorphic-s1inftyEQ0type}\\
\me^{\mi\varphi(\tau)}=\me^{\frac{\pi\mi}{4}}\frac{\sqrt{2\pi}(2\tau)^{4n+1}}{(-1)^n(2n)!s_0^{\infty}}
\exp\left(\mi\frac{\varepsilon b}{f_0}\left(\tau+\sum_{n=1}^{\infty}\eta_n\frac{\tau^{n+1}}{n+1}\right)\right),
\nonumber%\label{eq:phi-nonvanishing-meromorphic-s1inftyEQ0type2}
\end{gather}
where the coefficients $\eta_n$ are calculated via equation~\eqref{eq:eta-n-nonvanishing-meromorphic-gen}, and
the coefficients $f_m$ are given by equation~\eqref{eq:fm-def} accompanied
with the same specialization for $\tilde{b}_{2k-1,l}$ explained in the previous sentence, i.e., $a=-\mi(2n+1/2)$.
\end{enumerate}
\end{theorem}
\begin{proof}
Begin with the proof of the asymptotics for the meromorphic solutions presented in item~\pmb{$(1)$} of the theorem.
Assume that $u(\tau)$ has a Taylor-series expansion centred at $\tau=0$ with $u(0)\neq0$; this expansion should
coincide with the general asymptotic expansion~\eqref{app:eq:0-u-expansion} for $\sigma=\pm1$.
Due to the symmetry $\sigma\to-\sigma$ discussed in Subsection~\ref{app:subsec:error-correction-terms}, the
expansions with $\sigma=\pm1$ correspond to the same monodromy data; thus, we can put $\sigma=4\rho=1$, so that
$s_0^0=0$ (cf. equation~\eqref{eq:rho-general}). Equation~\eqref{eq:rho-general}
also implies that, if we require $s_0^{\infty}s_1^{\infty}\ne0$, then $a\neq\mi(k+1/2)$, $k\in\mathbb{Z}$. The other
conditions on the monodromy data are a consequence of equations~\eqref{eq:monodromy:s0} and \eqref{eq:monodromy:s1}.

The series~\eqref{eq:non-vanishing-meromorphic-u-general}, with the coefficients~\eqref{eq:fm-def} and
\eqref{eq:b10-b11-nonvanishing-meromorphic-gen}, is a rearrangement of the series~\eqref{app:eq:0-u-expansion} for
$\sigma=1$. The formula \eqref{eq:b1-1-nonvanishing-meromorphic-gen} for
$\tilde{b}_{1,-1}$ is calculated via equation~\eqref{app:eqs:betas}, where $\beta_{1,-1}=\tilde{b}_{1,-1}$, and
explicit expressions for $\varpi_n(\pm\rho)$, $n=1,2$, are given in Appendix B, Theorem B.1 of \cite{KitVar2023};
the latter theorem also allows one to calculate the leading term of asymptotics as $\tau\to0$ of the function
$\me^{\mi\varphi(\tau)}$ presented in equation~\eqref{eq:varphi-nonvanishing-meromorphic-gen}. The Taylor series
in the argument of the exponential function
in equation~\eqref{eq:varphi-nonvanishing-meromorphic-gen} is obtained with the help of equation~\eqref{eq:varphi}.

The proofs of the results presented in items~\pmb{$(2)$} and ~\pmb{$(3)$} are similar; item~\pmb{$(2)$} is a
special case of Theorem~\ref{th:Asympt0-rho-eq-1pn+ia2} (cf. subitem~\pmb{$(2.1)$} for $\varrho=1/4$ and
subitem~\pmb{$(2.2)$} for $\varrho=3/4$), and item~\pmb{$(3)$} is a special case of
Theorem~\ref{th:Asympt0-rho-eq-1pn-ia2} (cf. subitem~\pmb{$(3.1)$} for $\varrho=1/4$ and
subitem~\pmb{$(3.2)$} for $\varrho=3/4$).
\end{proof}
%%%%%%%%%%%%%%%%%%%%%%%%%%%%%%%%%%%%%%%%%%%%%%%%%%%%%%%%%%%%%%%%%%%%%%%%%%%%%%%%%%%%%%%%%%%%%%%%%%%%%%%%%%%%%
%%%%%%%%%%%%%%%%%%%%%%%%%%%%%%%%%%%%%%%%%%%%%%%%%%%%%%%%%%%%%%%%%%%%%%%%%%%%%%%%%%%%%%%%%%%%%%%%%%%%%%%%%%%%%
%\section{Numerical Examples}\label{sec:numerics}
%%%%%%%%%%%%%%%%%%%%%%%%%%%%%%%%%%%%%%%%%%%%%%%%%%%%%%%%%%%%%%%%%%%%%%%%%%%%%%%%%%%%%%%%%%%%%%%%%%%%%%%%%%%%%
%%%%%%%%%%%%%%%%%%%%%%%%%%%%%%%%%%%%%%%%%%%%%%%%%%%%%%%%%%%%%%%%%%%%%%%%%%%%%%%%%%%%%%%%%%%%%%%%%%%%%%%%%%%%%
\appendix
\section{Appendix. The Complete Small-$\tau$ Asymptotic Expansion of the General Solution $u(\tau)$}
\label{app:sec:full0expansion}
\subsection{Error Correction Term of the Power-Like Isomonodromy Asymptotics as $\tau\to0$}
\label{app:subsec:error-correction-terms}
The local expansion of the general solution $u(\tau)$ of equation~\eqref{eq:dp3} with
$\varepsilon=1$
in a neighbourhood of $\tau=0$ can be presented in the form of the following convergent (for small enough $\tau$) series:
\begin{equation}\label{app:eq:0-u-expansion}
u(\tau)=\sum_{k=1}^{\infty}\tau^{2k-1}\sum_{m=-k}^{k}b_{2k-1,m}\tau^{m\sigma},
\end{equation}
where $\sigma\neq0$ and the coefficients $b_{2k-1,m}$ are $\tau$-independent complex
numbers.\footnote{\label{foot:epsilon}
To get the $\varepsilon$-dependent variant of this expansion, one has to introduce $\varepsilon$ on both sides of all the
equations in this appendix according to the rule $u\to\varepsilon u$ and $b\to\varepsilon b$.}
The parameters $\sigma$ and $b_{1,\pm1}$ satisfy the
following conditions:
\begin{equation}\label{app:eq:sigma-b0-b1pm1}
|\mathrm{Re}\,\sigma|<2,\qquad
b_{1,0}=\frac{2ab}{\sigma^2},\qquad
b_{1,1}b_{1,-1}=\frac{b^2(4a^2+\sigma^2)}{4\sigma^4};
\end{equation}
otherwise, they can be taken arbitrarily.
The remaining coefficients $b_{2k-1,m}$, $k\geqslant2$, $|m|\leqslant k$, can be determined uniquely in terms of
$\sigma$ and one of the parameters $b_{1,\pm1}$ upon substitution of the series~\eqref{app:eq:0-u-expansion} into
equation \eqref{eq:dp3}.

There are several methodologies for proving the existence of such local expansions; in the context of the Painlev\'e
equations, we refer to the papers \cite{Shimomura1982-2,Shimomura1982-2+,KimuraH1983,TakanoK1986,GILy2013}.
In this appendix, we do not consider the formal proof of the expansion~\eqref{app:eq:0-u-expansion}, but, rather,
focus our attention on its computational aspects.

To determine the coefficients $b_{2k-1,m}$, substitute the expansion~\eqref{app:eq:0-u-expansion} into
equation~\eqref{eq:dp3}, and find that, for $k=1,2,3,\ldots$,
\begin{gather}\label{app:eq:b2k+1+k+1+k}
b_{2k+1,k+1}=(-1)^{k}\,\frac{2^{2k}(k+1)b_{1,1}^{k+1}}{(\sigma+2)^{2k}},\quad
%b_{2k+1,-k}=(-1)^{k-1}\,\frac{2^{2k+2}(2k+2-k\sigma)^2abb_{1,-1}^{k}}{\sigma^2(\sigma-4)^2(\sigma-2)^{2k-1}},
b_{2k+1,k}=(-1)^{k}\,\frac{2^{2k+2}(2k+2+k\sigma)^2ab\,b_{1,1}^{k}}{\sigma^2(\sigma+4)^2(\sigma+2)^{2k-1}},
%b_{2k+1,-k-1}=(-1)^{k}\,\frac{2^{2k}(k+1)b_{1,-1}^{k+1}}{(\sigma-2)^{2k}},\qquad
%k=1,2,\ldots.
\end{gather}
\begin{gather}\label{app:eq:b2k+1-k-1-k}
b_{2k+1,-k-1}=(-1)^{k}\,\frac{2^{2k}(k+1)b_{1,-1}^{k+1}}{(\sigma-2)^{2k}},\quad
b_{2k+1,-k}=(-1)^{k-1}\,\frac{2^{2k+2}(2k+2-k\sigma)^2ab\,b_{1,-1}^{k}}{\sigma^2(\sigma-4)^2(\sigma-2)^{2k-1}}.
%b_{2k+1,k}=(-1)^{k}\,\frac{2^{2k+2}(2k+2+k\sigma)^2abb_{1,1}^{k}}{\sigma^2(\sigma+4)^2(\sigma+2)^{2k-1}},\quad
%b_{2k+1,-k}=(-1)^{k-1}\,\frac{2^{2k+2}(2k+2-k\sigma)^2abb_{1,-1}^{k}}{\sigma^2(\sigma-4)^2(\sigma-2)^{2k-1}},
\end{gather}
We define the \emph{level} of the coefficient $b_{2k-1,m}$ to be the number $k$ in its first subscript; thus,
the total number of coefficients at level $k$ is $2k+1$.
As an example, we present the remaining coefficients for the levels 2 and 3:
\begin{equation}\label{app:eq:b30}
b_{3,0}=4b^2\frac{(20a^2\sigma^2+3\sigma^4-48a^2-4\sigma^2)}{\sigma^4(\sigma+2)^2(\sigma-2)^2}
=-\frac{12a^2b^2}{\sigma^4}+(a^2+1)b^2\left(\frac{\sigma^2+4}{(\sigma^2-4)^2}-\frac{1}{\sigma^2}\right),
\end{equation}
\begin{equation}\label{app:eq:b51}
\begin{gathered}
b_{5,1}=4b^2b_{1,1}((32\sigma^5+8\sigma^4-748\sigma^3-1120\sigma^2+1680\sigma+2880)a^2-
12\sigma^6-71\sigma^5\\-80\sigma^4+84\sigma^3+144\sigma^2)/
((\sigma-2)^2(\sigma+4)(\sigma+2)^4\sigma^4),
\end{gathered}
\end{equation}
\begin{equation}\label{app:eq:b5-1}
\begin{gathered}
b_{5,-1}=4b^2b_{1,-1}((32\sigma^5-8\sigma^4-748\sigma^3+1120\sigma^2+1680\sigma-2880)a^2+
12\sigma^6-71\sigma^5\\+80\sigma^4+84\sigma^3-144\sigma^2)/
((\sigma+2)^2(\sigma-4)(\sigma-2)^4\sigma^4),
\end{gathered}
\end{equation}
\begin{equation}\label{app:eq:b50}
b_{5,0}=\frac{192ab^3(7\sigma^6+36a^2\sigma^4-100\sigma^4-560a^2\sigma^2+192\sigma^2+1280a^2)}
{\sigma^6(\sigma-4)^2(\sigma+4)^2(\sigma-2)^2(\sigma+2)^2}.
\end{equation}

The expansion~\eqref{app:eq:0-u-expansion} is symmetric with respect to the change $\sigma\to-\sigma$ and
$b_{2k-1,m}\to b_{2k-1,-m}$. Note that equation~\eqref{eq:dp3} depends neither on the parameter $\sigma$ nor on any
of the coefficients $b_{2k-1,m}$, that is, it is also symmetric with respect to the change of variables
indicated above. This means that the coefficients $b_{2k-1,\pm m}$ are related to each other
by the change $\sigma\to-\sigma$ and $b_{1,1}\to b_{1,-1}$. This property can be observed upon comparing the
left- and right-most equations, respectively, in \eqref{app:eq:b2k+1+k+1+k} and \eqref{app:eq:b2k+1-k-1-k},
and \eqref{app:eq:b51} with \eqref{app:eq:b5-1}.

We computed the coefficients $b_{2k-1,m}$ up to the level $9$; therefore,
equations~\eqref{app:eq:b2k+1+k+1+k} and \eqref{app:eq:b2k+1-k-1-k} are verified only
for $k=1,2,\ldots,9$.\footnote{\label{app:footnote:computation-bkm}
On an eleven-year-old notebook (4Gb RAM, i7 processor, 4th generation), \textsc{Maple} 15 computed the coefficients
of four levels, from 2 to 5, in approximately 17 seconds; in contrast, on a notebook with 16Gb RAM, i7 processor,
12th generation, \textsc{Maple} 17 executes the same computation in roughly 2.5 seconds. The latter notebook computed
the coefficients of the subsequent 4 levels, from 6 to 9, in 333 seconds.
These calculations were carried out without taking into account the symmetry between the coefficients discussed above:
by employing the stated symmetry, the computational time could be reduced by $30$ to $35$ percent.
In order to illustrate the increase in complexity of the calculation, we refer, say, to equation~\eqref{app:eq:b51},
which takes less than two lines of text to display; in contrast, the coefficient $b_{17,1}$, being presented in the
same, explicit way, would require at least 132 lines of text to display!}
In Subsection~\ref{app:subsec:super-generating-function} below, these formulae are proved for all $k\in\mathbb{N}$.
The complexity of the coefficients, together with their number, increases quickly with the growth of the level;
therefore, the reader should not be overly optimistic about the prospect of obtaining as many of the coefficients as
necessary for the achievement of the required degree of accuracy.

The simplest application of the expansion~\eqref{app:eq:0-u-expansion} is the calculation of the error-correction term
for the small-$\tau$ asymptotics obtained in \cite{KitVar2004}. Recall that the small-$\tau$ asymptotic formula obtained
in \cite{KitVar2004} reads
\begin{equation}\label{app:eq:Asympt0:u2004}
u(\tau)\underset{\tau\to0^+}{=} \, \frac{\tau b \me^{\pi a/2}}{16 \pi}\!
\left(\varpi_{1}(\rho)\tau^{2\rho}+\varpi_{1}(-\rho)\tau^{-2\rho}\right)\!
\left(\varpi_{2}(\rho)\tau^{2\rho}+\varpi_{2}(-\rho)\tau^{-2\rho}\right)\!
\left(1+o\big(\tau^{\delta})\right),
\end{equation}
where $\delta>0$, and the $\tau$-independent coefficients $\varpi_{n}(\rho)$, $n=1,2$, depend on the branching parameter,
$\rho$, and the monodromy data: their explicit formulae are given in \cite{KitVar2004,KitVar2023}.
Since equation~\eqref{app:eq:Asympt0:u2004} is symmetric with respect to the reflection $\rho\to -\rho$, we assume that
$\mathrm{Re}\,\rho\geqslant0$; then, expanding the brackets on the right-hand side of \eqref{app:eq:Asympt0:u2004}, one arrives at
 \begin{equation}\label{app:eq:Asympt0:u2004expanded}
u(\tau)\underset{\tau\to0^+}{=}\beta_{1,-1}\tau^{1-4\rho}+\beta_{1,0}\tau+\beta_{1,1}\tau^{1+4\rho}+
o\big(\tau^{1-4\rho+\delta}\big),
\end{equation}
where
\begin{equation}\label{app:eqs:betas}
\beta_{1,\pm1}=\frac{b\me^{\pi a/2}}{16 \pi}\varpi_{1}(\pm\rho)\varpi_{2}(\pm\rho),\qquad
\beta_{1,0}=\frac{b\me^{\pi a/2}}{16 \pi}
\big(\varpi_{1}(\rho)\varpi_{2}(-\rho)+\varpi_{1}(-\rho)\varpi_{2}(\rho)\big).
\end{equation}
Now, using the explicit expressions for $\varpi_{n}(\pm\rho)$, which can be taken from either one of the works
\cite{KitVar2004} or \cite{KitVar2023}, one proves that the coefficients $\beta_{1,m}$, $m=-1,0,1$, satisfy the same
equations~\eqref{app:eq:sigma-b0-b1pm1} (with $\sigma^2=(4\rho)^2$) as the coefficients $b_{1,m}$ with the
corresponding subscripts. Comparing the expansions~\eqref{app:eq:Asympt0:u2004expanded} and
\eqref{app:eq:0-u-expansion} and taking into account the symmetry $\sigma\to-\sigma$, we can set $\sigma=4\rho$ and
$\beta_{k,m}=b_{k,m}$, $m=-1,0,1$. This comparison allows us to derive a more precise evaluation for the parameter
$\delta$ in the correction term of equation~\eqref{app:eq:Asympt0:u2004}.

\textcolor[rgb]{0,0,1}{
The error estimate to the level 1 terms in the expansion \eqref{app:eq:0-u-expansion} can be written as $\mathcal{O} \big(
\tau^{\delta_1} \big)$, where $\delta_1=3-2\mathrm{Re}\,\rho$. If we assume that $1-4\mathrm{Re}\,\rho+\delta \geqslant \delta_1$,
then the first three explicit terms in the asymptotics \eqref{app:eq:Asympt0:u2004expanded} are larger than the corresponding
error estimate. The solution of equation \eqref{eq:dp3} with such asymptotics does not exist because substituting this
expansion into \eqref{eq:dp3} we get a term which cannot be cancelled by the term generated {}from the correction
$o(\tau^{1-4\rho+\delta})$.\footnote{\label{foot:app:small-o-terms-asymptotics} \textcolor[rgb]{0,0,1}{
The substitution is simpler if the multiplicative
form of the asymptotics \eqref{app:eq:Asympt0:u2004} is exploited, and equation \eqref{eq:dp3} should be integrated {}from
some finite value $\tau_0$ to $\tau \to 0$. The small-$\tau$ isomonodromy asymptotics is differentiable with respect to $\tau$,
that is, the asymptotics of the derivative $u^{\prime}(\tau)$ is the formal derivative of the asymptotics of $u(\tau)$.}} There is an
alternative argument which does not require direct substituton; in case $1-4\mathrm{Re}\,\rho+\delta\geqslant\delta_1$, then, we
have two different solutions: one defined by the isomonodromy asymptotics, and the other by the full asymptotic expansion; however,
all points of the monodromy manifold are already assigned to the ``isomonodromy'' solutions via their asymptotics, so the monodromy
manifold does not have any ``space'' for the solutions defined via the asymptotic expansion \eqref{app:eq:0-u-expansion}.}

\textcolor[rgb]{0,0,1}{
We now assume that $1-4\mathrm{Re}\,\rho+\delta<\delta_1$; then, the asymptotics defined by the expansion \eqref{app:eq:0-u-expansion}
satisfies the condition for the isomonodromy asymptotics \eqref{app:eq:Asympt0:u2004expanded}. If we equate the largest terms of these
asymptotics, then, for both solutions defined by these asymptotics, we get, via a direct solution of the monodromy problem \cite{KitVar2004},
the same monodromy data, which means that they coincide, and we see that the correction $o\big(\tau^{1-4\rho+\delta}\big)$ in the expansion
\eqref{app:eq:Asympt0:u2004expanded} can, in fact, be strengthened to $\mathcal{O}\big(\tau^{1-4\rho+\delta}\big)$ since $\delta>0$ is not
fixed; subsequently, comparing the latter estimate with the largest term of the second level in the expansion \eqref{app:eq:0-u-expansion}, we
obtain $1-4\mathrm{Re}\,\rho+\delta=3-2\mathrm{Re}\,\sigma$, where $\sigma=4\rho$, which implies $\delta=2-4\mathrm{Re}\,\rho$. If one does
not assume that $\mathrm{Re}\,\rho>0$ and reverts back to the symmetric form of the asymptotics \eqref{app:eq:Asympt0:u2004}, then one gets
$\delta=2-4|\mathrm{Re}\,\rho|$.}

In certain special cases, the evaluation of $\delta$ can be improved; assume, say, that $\mathrm{Re}\,\sigma<0$, and consider the degenerate
case $b_{1,1}=0$. (Note that this does not necessarily imply that $b_{1,-1}=0$.) In this case, however, the coefficients $b_{2k-1,m}=0$ for all
$m \in \mathbb{N}$; it is enough, in fact, to know that $b_{3,2}=b_{3,1}=0$ (cf. equations~\eqref{app:eq:b2k+1+k+1+k} and \eqref{app:eq:b2k+1-k-1-k}),
so that the largest non-vanishing correction term is $b_{3,0} \tau^3$, hence $\delta=2$.
%%%%%%%%%%%%%%%%%%%%%%%%%%%%%%%%%%%%%%%%%%%%%%%%%%%%%%%%%%%%%%%%%%%%%%%%%%%%%%%%%%%%%%%%%%%%%%%%%%%%%%%%%%%%%%%%%%%%%%%%%%%%%%
%%%%%%%%%%%%%%%%%%%%%%%%%%%%%%%%%%%%%%%%%%%%%%%%%%%%%%%%%%%%%%%%%%%%%%%%%%%%%%%%%%%%%%%%%%%%%%%%%%%%%%%%%%%%%%%%%%%%%%%%%%%%%%
\subsection{Super-Generating Function}\label{app:subsec:super-generating-function}
The formal construction of the
super-generating function for the coefficients of the expansion~\eqref{app:eq:0-u-expansion} is defined as
\begin{equation}\label{app:eq:supergenfuncton}
A(x,y)=\sum_{n=0}^{\infty}y^nA_{n}(x),
\end{equation}
where the coefficients $A_n(x)$, $n\in\mathbb{Z}_{\geqslant0}$, are generating functions for the coefficients
$b_{2k-1,k-n}$,
\begin{gather}
A_0(x)=\sum_{k=1}^{\infty}b_{2k-1,k}x^k,\label{app:eq:A0-series}\\
A_n(x)=\sum_{k=\lfloor(n-1)/2\rfloor+1}^{\infty}b_{2k-1,k-n}x^k,\qquad
n\geqslant1,\label{app:eq:An-series}
\end{gather}
where $\lfloor\cdot\rfloor$ denotes the floor of the real number.
It is easy to see that each coefficient $b_{2k-1,m}$, $k\in\mathbb{N}$, $m=-k,\ldots,k$, belongs to one, and only
one, function $A_n(x)$ for some $n\in\mathbb{Z}_{\geqslant0}$.

Define the linear differential operator $D$ acting in the space of formal power series of two variables $x$ and $y$ as
\begin{equation}\label{app:eq:D-definition}
D:=(2+\sigma)x\frac{\partial}{\partial x}-\sigma y\frac{\partial}{\partial y};
\end{equation}
then, the function $A\equiv A(x,y)$ solves the PDE
\begin{equation}\label{app:eq:A-PDE}
D^2(\ln A)=-8A+2a\frac{bxy}{A}+\left(\frac{bxy}{A}\right)^2.
\end{equation}
Note that equation~\eqref{app:eq:A-PDE} has the symmetry $x\leftrightarrow y$,
$-\sigma\leftrightarrow 2+\sigma$. This symmetry, however, cannot be interpreted in terms of the coefficients
$b_{2k-1,m}$ because the symmetry has ``renotational sense'', that is, $x$ plays the role of $y$ and
\emph{vice versa} (see equation~\eqref{app:eq:u-A-relation} below).
To justify equation~\eqref{app:eq:A-PDE}, one has to use the following relation between the super-generating function
$A(x,y)$ (cf. equation~\eqref{app:eq:supergenfuncton}) and the solution $u(\tau)$ (cf.
equation~\eqref{app:eq:0-u-expansion}):
\begin{equation}\label{app:eq:u-A-relation}
u(\tau)=\frac1{\tau} A(\tau^{2+\sigma},\tau^{-\sigma}).
\end{equation}

Now, we apply equation~\eqref{app:eq:A-PDE} for the calculation of the generating functions $A_n(x)$: substitute
the series~\eqref{app:eq:supergenfuncton} into equation~\eqref{app:eq:A-PDE}, take the numerator of the resulting
equation, and equate to zero the coefficients of powers of $y^n$. For $n=0$, we get the following second-order ODE,
\begin{equation}\label{app:eq:A0}
D_x^2\ln A_0(x)=-8A_0(x),
\end{equation}
where
\begin{equation}\label{app:eq:Dx}
D_x:=(2+\sigma)\,x\frac{\md}{\md x}
\end{equation}
is the $x$-part of the operator $D$ (cf. equation~\eqref{app:eq:D-definition}).
Equation~\eqref{app:eq:A0} has the following solutions:
\begin{equation}\label{app:eq:A0solution}
A_{0,gen}(x)=\frac{(2+\sigma)^2\,C_1^2\,C_2x^{C_1}}{4(x^{C_1}+C_2)^2},\qquad
A_{0, spec}(x)=-\frac{(\sigma+2)^2}{4\ln^2(C_2x)},
\end{equation}
where $C_1$ and $C_2$ are constants of integration. To get the solution that is consistent with the definition of
$A_0(x)$, one must set $C_1=1$ and $C_2=(\sigma+2)^2/(4b_{1,1})$, so that
\begin{equation}\label{app:eq:A0-explicit}
A_0(x)=\frac{b_{1,1}x}{(1+4b_{1,1}x/(\sigma+2)^2)^2}=\frac{(\sigma+2)^2z}{4(1+z)^2},\qquad
z=\frac{4b_{1,1}x}{(\sigma+2)^2}.
\end{equation}
Expanding the function $A_0(x)$ in equation~\eqref{app:eq:A0-explicit} into a power series in $x$ and comparing this
expansion with equation~\eqref{app:eq:A0-series}, one proves the left-most equation in \eqref{app:eq:b2k+1+k+1+k}
and, due to the symmetry $\sigma\to-\sigma$ and $b_{1,1}\to b_{1,-1}$ (cf.
Subsection~\ref{app:subsec:error-correction-terms}), also the left-most equation in \eqref{app:eq:b2k+1-k-1-k}.

By continuing this process of determining the generating functions described above, one obtains the following
ODE for the function $A_1(x)$:
\begin{equation}\label{app:eq:A1}
\left(\left(D_x-\sigma\right)^2+8A_0(x)\right)A_{1,0}(x)=\frac{2abx}{A_0(x)},\qquad
A_{1,0}(x):=\frac{A_1(x)}{A_0(x)}.
\end{equation}
The homogeneous part of equation~\eqref{app:eq:A1} is a degenerate hypergeometric equation, all of whose solutions
have, for $\sigma\neq-2+2/n_1$, $n_1\in\mathbb{N}$, a branching point at $x=0$. Since the right-hand side of
equation~\eqref{app:eq:A1} is a rational function of $x$, it follows that, for generic values of $\sigma$ and
$b_{11}\neq0$, there exists a unique rational solution $A_1(x)$ of this equation which, in terms of the variable $z$
(cf. equation~\eqref{app:eq:A0-explicit}), reads
\begin{equation}\label{app:eq:A1z-explicit}
A_1(x)=\frac{ab(2+\sigma)^2z(z\sigma-\sigma-4)(z^2\sigma+2z(\sigma^2+4\sigma+2)-\sigma-4)}
{2\sigma^2(4+\sigma)^2b_{1,1}(z+1)^3 }.
\end{equation}
For the special values of $\sigma$ mentioned above, we also have the same solution~\eqref{app:eq:A1z-explicit},
because adding to $A_1(x)$ the rational part of the solution of the homogenous equation invalidates the
expansion~\eqref{app:eq:An-series}, that is, it leads to the appearence of non-positive powers of $x$ in the
expansion.
Reverting back to the original variable $x$ and developing $A_1(x)$ into a power series in $x$, one finds that
\begin{equation}\label{app:eq:A1x-series}
A_1(x)=\frac{2abx}{\sigma^2}-\frac{16abb_{1,1}x^2}{\sigma^2(2+\sigma)}+
\frac{256ab(3+\sigma)^2b_{1,1}^2x^3}{\sigma^2(4+\sigma)^2(2+\sigma)^3}-
\frac{256ab(8+3\sigma)^2b_{1,1}^3x^4}{\sigma^2(4+\sigma)^2(2+\sigma)^5}+\mathcal{O}(x^5).
\end{equation}
Comparing the expansions~\eqref{app:eq:A1x-series} and \eqref{app:eq:An-series} term-by-term, one verifies the second
relation given in \eqref{app:eq:sigma-b0-b1pm1} and the right-most formula in \eqref{app:eq:b2k+1+k+1+k} for
$k=1,2,3$; the formula for arbitrary $k\in\mathbb{N}$ given there can be proven by decomposing $A_1(x)$ into a
sum of partial fractions.

Even though the procedure for the construction of the generating functions $A_n(x)$ is straightforward, it requires
rather cumbersome calculations, if done by hand. We checked that \textsc{Maple} was able to handle
these calculations for $n=2,3,4$ in a few seconds; but thus far we haven't found a compact presentation for
the corresponding results.

Consider, for example, the calculation of the generating function $A_2(x)$. This calculation shows, on the one hand,
the increased complexity of the coefficients, and, on the other hand, that it is general enough to estimate the
complexity of successive calculations for the generating functions $A_n(x)$.
The function $A_2(x)$ solves the following ODE:
\begin{equation}\label{app:eq:A2}
\left(\left(D_x-2\sigma\right)^2+8A_0\right)A_{2,0}=
\frac12\left(D_x-2\sigma\right)^2A_{1,0}^2-2abx\frac{A_{1,0}}{A_0}+\left(\frac{bx}{A_0}\right)^2,\quad
A_{2,0}:=\frac{A_2(x)}{A_0(x)}.
\end{equation}
This equation is similar to equation~\eqref{app:eq:A0}; however, its right-hand side is more complicated, thus
resulting in a substantially more involved explicit formula for the solution. Every solution that is single-valued
at $x=0$ is a rational function of $x$. If one takes into account that the solution should have a first-order zero at
$x=0$ (cf. equation~\eqref{app:eq:An-series} for $n=2$), then one arrives at the following partial-fraction
decomposition of $A_2(x)$, presented, again, in terms of the variable $z$ (cf. the right-most equation
in \eqref{app:eq:A0-explicit}):
\begin{equation}\label{app:eq:A2-partfractions}
A_2(x)=\sum_{k=0}^3\xi_kz^k+\sum_{k=1}^4\frac{\xi_{-k}}{(z+1)^k},
\end{equation}
where the coefficients $\xi_k$, $k=-4,-3,-2,-1,0,1,2,3$, satisfy the condition
\begin{equation}\label{app:eq:ksi-relation}
\xi_0+\sum_{k=1}^4\xi_{-k}=0,
\end{equation}
and depend only on $a$ and $\sigma$. Explicit formulae for the coefficients $\xi_k$ read:
\begin{align*}
\xi_3=&\frac{b^2(2+\sigma)^2((4+\sigma)^2+4a^2)}{16b_{1,1}^2(4+\sigma)^4},\quad\\
\xi_2=&\frac{b^2\left(2+\sigma)^2(4(5\sigma^2+40\sigma+68)a^2+(3\sigma^2+24\sigma+44\right)(4+\sigma)^2)}
{4b_{1,1}^2(4+\sigma)^4(6+\sigma)^2},\quad
\end{align*}
\begin{align*}
\xi_1=&\frac{b^2(2+\sigma)^2}{16b_{1,1}^2\sigma(4+\sigma)^4(6+\sigma)^2}
\left(4(8\sigma^5+158\sigma^4+1061\sigma^3+2964\sigma^2+3412\sigma+1152)a^2\right.\\
+&\left.\sigma(12\sigma^3+121\sigma^2+380\sigma+388)(4+\sigma)^2\right),\\
\xi_0=&-\frac{b^2(2+\sigma)^5}{b_{1,1}^2\sigma^2(4+\sigma)^4(6+\sigma)^2(2-\sigma)^2}
\left(2(8\sigma^5+95\sigma^4+184\sigma^3-584\sigma^2-96\sigma+576)a^2\right.\\
+&\left.3\sigma^2(\sigma^2+4\sigma-6)(4+\sigma)^2\right),
\end{align*}
\begin{align*}
\xi_{-1}=&\frac{3b^2(2+\sigma)^5(2+3\sigma)^2}{4b_{1,1}^2\sigma^4(4+\sigma)^4(6+\sigma)^2(2-\sigma)^2}
\left(2(4\sigma^5+45\sigma^4+72\sigma^3-344\sigma^2-96\sigma+576)a^2\right.\\
+&\left.\sigma^2(\sigma^2+4\sigma-6)(4+\sigma)^2\right),\\
\xi_{-2}=&-\frac{b^2(2+\sigma)^6}{4b_{1,1}^2\sigma^4(4+\sigma)^4(6+\sigma)^2(2-\sigma)^2}
\left(2(148\sigma^6+1657\sigma^5+2898\sigma^4-12584\sigma^3\right.\\
-&\left.11792\sigma^2+22656\sigma+19584)a^2+3\sigma^2(6+7\sigma)(\sigma^2+4\sigma-6)(4+\sigma)^2\right),
\end{align*}
\begin{align*}
\xi_{-3}=&\frac{b^2(2+\sigma)^7}{2b_{1,1}^2\sigma^4(4+\sigma)^4(6+\sigma)^2(2-\sigma)^2}
\left(2(48\sigma^5+463\sigma^4+248\sigma^3\right.\\
-&\left.4808\sigma^2+1056\sigma+7488)a^2+3\sigma^2(\sigma^2+4\sigma-6)(4+\sigma)^2\right),\\
\xi_{-4}=&-\frac{12b^2(2+\sigma)^8a^2}{b_{1,1}^2\sigma^4(4+\sigma)^4}.
\end{align*}
One now verifies that equation~\eqref{app:eq:ksi-relation} is satisfied. Expanding the function $A_2(x)$
(cf. equation \eqref{app:eq:A2-partfractions})
into a power series in $z$, taking into account the relation between $z$ and $x$,
and comparing this with the expansion~\eqref{app:eq:An-series},
one obtains an explicit formula for the coefficients $b_{2k-1,k-2}$,
\begin{equation}\label{app:eq:A2power-few}
b_{2k-1,k-2}=\frac{(-4)^kb_{1,1}^k}{(2+\sigma)^{2k}}\sum_{p=1}^4\binom{k+p-1}{p-1}\xi_{-p},\qquad
k=4,5,\ldots,
\end{equation}
where $\tbinom{m}{k}=\tfrac{m!}{k!(m-k)!}$ is the binomial coefficient, and, the first three off-set coefficients,
\begin{equation}\label{app:eq:b2k-1=k-2}
b_{2k-1,k-2}=\frac{(-4)^kb_{1,1}^k}{(2+\sigma)^{2k}}\left((-1)^k\xi_k+\sum_{p=1}^4\binom{k+p-1}{p-1}\xi_{-p}\right),
\qquad
k=1,2,3.
\end{equation}
Substituting $k=1,2,3$ into equation \eqref{app:eq:b2k-1=k-2}, one reproduces $b_{1,-1}$, $b_{3,0}$, and $b_{5,1}$
given, respectively, by the third equation in \eqref{app:eq:sigma-b0-b1pm1}, and equations \eqref{app:eq:b30} and
\eqref{app:eq:b51}.

The calculation of the generating functions $A_0(x)$, $A_1(x)$, and $A_2(x)$ with the help of \textsc{Maple} on
a generic laptop takes but a few seconds; therefore, from the practical point of view, one can continue such
calculations to obtain successive generating functions.
These functions for $n\geqslant3$ satisfy the following inhomogeneous degenerate hypergeometric equation,
\begin{equation}\label{app:eq:An}
\begin{aligned}
\left(\left(D_x-n\sigma\right)^2+8A_0\right)A_{n,0}=&
\sum_{k=2}^n\frac{(-1)^k}{k}
\sum_{\substack{i_1+\ldots+i_k=n\\i_j\in\mathbb{N}}}\left(D_x-n\sigma\right)^2A_{i_1,0}A_{i_2,0}\ldots A_{i_k,0}\\
+&\frac{2abx}{A_0}\sum_{k=1}^{n-1}(-1)^k\sum_{\substack{i_1+\ldots+i_k=n-1\\i_j\in\mathbb{N}}}A_{i_1,0}A_{i_2,0}
\ldots A_{i_k,0}\\
+&\left(\frac{8bx}{A_0}\right)^2\,\sum_{k=1}^{n-2}(-1)^k(k+1)\sum_{\substack{i_1+\ldots+i_k=n-2\\i_j\in\mathbb{N}}}
A_{i_1,0}A_{i_2,0}\ldots A_{i_k,0},
\end{aligned}
\end{equation}
where, for any natural $k$, $A_{k,0}:=A_k(x)/A_0(x)$. Equation~\eqref{app:eq:An} is, in fact, valid starting from
$n=1$, provided one starts the summations in the last two sums from $k=0$ and agrees to abide by the standard rules
for interpreting the $\sum$-operator in such degenerate situations, e.g., $\sum^{-1}_{k=0}:=0$. To complete the
definition of $A_n(x)$, we have to state that, for all $n\in\mathbb{N}$, $A_n(x)$ is the unique rational solution of
equation~\eqref{app:eq:An} with the first term of its Taylor expansion at $x=0$ of the order $x^N$, where
$N=\lfloor\tfrac{n-1}2\rfloor+1$ (cf. equation~\eqref{app:eq:An-series}). The last condition is necessary in order to
remove an ambiguity of the rational solution for some special values of $\sigma$: for generic values of $\sigma$,
this condition is satisfied automatically.

The right-hand side of equation~\eqref{app:eq:An} becomes progressively more complicated for increasing values of $n$.
This fact, however, does not encumber the procedure for finding explicitly the generating functions $A_n(x)$:
we have performed calculations for $n=3$ and $4$ without encountering a visible increase of the time of calculations.
The main problem is the presentation of the results in an observable form: the partial-fraction decomposition
of $A_n(x)$ helps, to some extent, because, structurally, $A_n(x)$ in terms of $z$ has poles only at $z=-1$
and $z=\infty$; however, the number and complexity of the corresponding coefficients is incresing rapidly.
Computationally, the procedure is quite simple for \textsc{Maple} to execute, so that, ultimately, it is the
lack of available memory that halts the calculations.

\subsection{A Symmetric Power-Like Asymptotic Expansion as $\tau\to0$}\label{app:subsec:asympt-universal}
As mentioned in Subsection~\ref{app:subsec:error-correction-terms}, the original
expansion~\eqref{app:eq:0-u-expansion} possesses the symmetry $\sigma\to-\sigma$, $b_{1,1}\leftrightarrow b_{1,-1}$.
In the construction of the super-generating function $A(x,y)$ in
Subsection~\ref{app:subsec:super-generating-function}, this symmetry is lost because of the ``non-symmetric''
definition of $A(x,y)$. We can, however, using the symmetry, define a symmetric reflection of the function $A(x,y)$,
namely, the function $\bar{A}(\bar{x},\bar{y})$, by making in the definitions~\eqref{app:eq:supergenfuncton},
\eqref{app:eq:A0-series}, and \eqref{app:eq:An-series} the changes $A_n(x)\to\bar{A}_n(\bar{x})$ for
$n\in\mathbb{Z}_{\geqslant0}$, $b_{2k-1,m}\to b_{2k-1,-m}$ for all $m\in\mathbb{Z}$,
$x\to\bar{x}$, and $y\to\bar{y}$. In the subsequent constructions of
Subsection~\ref{app:subsec:super-generating-function}, we have to change, additionally, $\sigma\to-\sigma$,
so that, in particular, the operator $D\to\bar{D}$, where
\begin{equation*}\label{app:eq:barD}
\bar{D}=(2-\sigma)\bar{x}\frac{\partial}{\partial\bar{x}}+\sigma\bar{y}\frac{\partial}{\partial\bar{y}}.
\end{equation*}
In the formula expressing $u(\tau)$ in terms of $A(x,y)$, we must substitute $x=\tau^{2+\sigma}$ and
$y=\tau^{-\sigma}$, and, in the corresponding formula in terms of $\bar{A}(\bar{x},\bar{y})$, the conjugated
variables $\bar{x}$ and $\bar{y}$ should be replaced by
$\bar{x}=\tau^{2-\sigma}$ and $\bar{y}=\tau^{\sigma}$, so that this formula reads
\begin{equation}\label{app:eq:u-bar-A-relation}
u(\tau)=\frac1{\tau}\bar{A}(\tau^{2-\sigma},\tau^{\sigma}).
\end{equation}
Adding equations~\eqref{app:eq:u-A-relation} and \eqref{app:eq:u-bar-A-relation}, we arrive at the symmetized
form for the function $u(\tau)$:
\begin{equation}\label{app:eq:u-A-bar-A-relation}
u(\tau)=\frac1{2\tau}\big(A(\tau^{2+\sigma},\tau^{-\sigma})+\bar{A}(\tau^{2-\sigma},\tau^{\sigma})\big).
\end{equation}

As an application of the ideas considered herein, we obtain the uniform
(with respect to $\sigma$) leading term of the power-like small-$\tau$ asymptotics of the function $u(\tau)$.
Using the definition of $A_0(x)$ (cf. equation~\eqref{app:eq:A0series}) and its symmetry conjugate
$\bar{A}_0(\bar{x})$, we can rewrite the expansion~\eqref{app:eq:0-u-expansion} as follows:
\begin{equation}\label{app:eq:u:Asympt-uniformA}
u(\tau)\underset{\tau\to0^+}=\frac{1}{\tau}\big(A_0(x)+\bar{A}_0(\bar{x})+b_{1,0}\tau^2\big)
+\mathcal{O}\big(\tau^{3-|\mathrm{Re}\,\sigma|}\big),\quad
x=\tau^{2+\sigma},\;\;
\bar{x}=\tau^{2-\sigma}.
\end{equation}
Taking into account the explicit expression for $A_0(x)$ (cf. equation~\eqref{app:eq:A0-explicit}), and the
corresponding expression for $\bar{A}_0(\bar{x})$ obtained via the symmetry described above, one finds that
\begin{equation}\label{app:eq:u:Asympt-uniform}
u(\tau)\underset{\tau\to0^+}=
\tau\left(\frac{b_{1,1}\tau^{\sigma}}{\left(1+\frac{4b_{1,1}\tau^{2+\sigma}}{(\sigma+2)^2}\right)^2}+
\frac{b_{1,-1}\tau^{-\sigma}}{\left(1+\frac{4b_{1,-1}\tau^{2-\sigma}}{(\sigma-2)^2}\right)^2}
+\frac{2ab}{\sigma^2}\right)+\mathcal{O}\!\left(\tau^{3-|\mathrm{Re}\,\sigma|}\right).
\end{equation}
The asymptotics~\eqref{app:eq:u:Asympt-uniform} is valid for all values of the parameter $\sigma$ such that
$\mathrm{Re}\,\sigma\in[-2,2]$, excluding the values $\sigma=0,\pm2$.\footnote{\label{app:foot:sigma0pm2}
As a matter of fact, this formula also works for $\sigma=0,\pm2$ in the sense of the proper limiting procedures
$\sigma\to0$ or $\sigma\to\pm2$. To evaluate the limits, one must use the monodromy parametrizations of $b_{1,\pm1}$.}
For the latter values of $\sigma$, the asymptotics is constructed in terms of logarithmic functions: these cases
are considered in Appendices~\ref{app:sec:full0expansionLog} and \ref{app:sec:full0expansionLog2} below.
Of course, it is assumed that the asymptotics~\eqref{app:eq:u:Asympt-uniform} is valid for those solutions $u(\tau)$
corresponding to monodromy data for which the parametrizations of $b_{1,\pm1}$ via these data make
sense (see the definition of this parametrization in the paragraph following equation~\eqref{app:eqs:betas}).
Clearly, in some domains of $\sigma$, the asymptotics~\eqref{app:eq:u:Asympt-uniform} can be simplified
(some terms can be omitted); in Section~\ref{sec:general}, say, we derived asymptotics for $u(\tau)$
in terms of the parameter $\varrho$, with $\sigma=4\varrho$: this formula is, in fact, valid in the strip
$0<\mathrm{Re}\,\sigma<4$ punctured at $\sigma=2$. For $\mathrm{Re}\,\sigma>1$, the first term of the
asymptotics~\eqref{app:eq:u:Asympt-uniform} is smaller than the correction term, and can, therefore, be omitted,
and, if $\mathrm{Re}\,\sigma>2$, the third term of the asymptotics~\eqref{app:eq:u:Asympt-uniform} is smaller
than the correction term, and can also be neglected, thus the leading terms in the asymptotic
formulae~\eqref{eq:Asympt0u-} and \eqref{app:eq:u:Asympt-uniform} coincide, even though they are obtained
by different methods! Note, however, that for $0<\mathrm{Re}\,\sigma<2$, the asymptotic
formula~\eqref{app:eq:u:Asympt-uniform} provides a more accurate approximation for $u(\tau)$ than the
asymptotic results presented in Section~\ref{sec:general}; furthermore, for $\mathrm{Re}\,\sigma\in(-2,2)$, the
asymptotics~\eqref{app:eq:u:Asympt-uniform} better approximates the function $u(\tau)$ than does the
asymptotics~\eqref{app:eq:Asympt0:u2004expanded} obtained in \cite{KitVar2004} (see the discussion of the
error estimation in the paragraph subsequent to equation~\eqref{app:eqs:betas}).

In light of the discussion above, it is interesting to see how one can go beyond the
leading term of the uniform asymptotics and obtain explicitly at least the first correction term. In this respect,
our strategy is the same as for the derivation of the uniform asymptotic expansion~\eqref{app:eq:u:Asympt-uniform}:
rewrite the original asymptotic expansion~\eqref{app:eq:0-u-expansion} with the help of the functions
$A_0(x)$ and $A_1(x)$ (cf. equations~\eqref{app:eq:A0-series} and \eqref{app:eq:An-series} for $n=1$) and their
symmetry conjugates $\bar{A}_0(\bar{x})$ and $\bar{A}_1(\bar{x})$ keeping all terms up to the level $k=3$:
\begin{equation}\label{app:eq:derivation-level3}
u(\tau)\underset{\tau\to0^+}=\frac{1}{\tau}\big(A_0(x)+\bar{A}_0(\bar{x})+yA_1(x)+\bar{y}\bar{A}_1(\bar{x})
-b_{1,0}\tau^2+b_{3,0}\tau^4\big)+\mathcal{O}\big(\tau^{5-|\mathrm{Re}\,\sigma|}\big),
\end{equation}
where we subtracted $b_{1,0}\tau^2$, because it is counted twice, once in each of the terms $yA_1(x)$
and $\bar{y}\bar{A}_1(\bar{x})$. Taking note of the leading term of the expansion as presented in
equation~\eqref{app:eq:u:Asympt-uniformA}, we next rewrite the expansion~\eqref{app:eq:derivation-level3}
using the relations $yx=\bar{y}\bar{x}=\tau^2$:
\begin{equation}\label{app:eq:derivation-level3n}
u(\tau)\underset{\tau\to0^+}=\frac{1}{\tau}\big(A_0(x)+\bar{A}_0(\bar{x})+b_{1,0}\tau^2\big)
+\tau\big(A_1(x)/x+\bar{A}_1(\bar{x})/\bar{x}-2b_{1,0}\big)+b_{3,0}\tau^3
+\mathcal{O}\big(\tau^{5-|\mathrm{Re}\,\sigma|}\big).
\end{equation}
Using equation~\eqref{app:eq:A1z-explicit}, where $z$ is defined in equation~\eqref{app:eq:A0-explicit}, and
recalling that $b_{1,0}=2ab/\sigma^2$ (cf. equation~\eqref{app:eq:sigma-b0-b1pm1}), one finds that
\begin{equation}\label{app:eq:A1:x-b10}
\frac{A_1(x)}{x}-b_{1,0}=-\frac{4abz(\sigma + 2)}{\sigma^2(\sigma+4)^2(z+1)}
\left(\frac{2(\sigma+2)^2}{(z+1)^2}-\frac{\sigma^2 - 4}{z + 1}+4\right),
\end{equation}
and its conjugate reads
\begin{equation}\label{app:eq:barA1:x-b10}
\frac{\bar{A}_1(\bar{x})}{\bar{x}}-b_{1,0}=\frac{4ab\bar{z}(\sigma-2)}{\sigma^2(\sigma-4)^2(\bar{z}+1)}
\left(\frac{2(\sigma-2)^2}{(\bar{z}+1)^2}-\frac{\sigma^2-4}{\bar{z}+1}+4\right),\qquad
\bar{z}=\frac{4b_{1,-1}\bar{x}}{(\sigma-2)^2}.
\end{equation}
\begin{remark}\label{app:eq:Theorem3.1error-corr-term}
The asymptotic formula~\eqref{app:eq:derivation-level3n} allows us to justify the error estimate for the asymptotics
obtained in Section~\ref{sec:general} (cf. equations~\eqref{eq:Asympt0u-} and \eqref{eq:Asympt0u+}).
If $\mathrm{Re}\,\varrho\in(0,1/2)$, these asymptotic formulae correspond to the parameter
$\mathrm{Re}\,\sigma\in(0,2)$; then, referring
to the asymptotic expansion~\eqref{app:eq:0-u-expansion} and comparing these expansions, one finds that the term
$\mathcal{O}(\tau)$ as $\tau\to0$ is absent in the expansions~\eqref{eq:Asympt0u-} and \eqref{eq:Asympt0u+}.
The leading term of asymptotics in Section~\ref{sec:general} is of the order $\tau^{1-4\varrho}$; denoting
the correction term as $\mathcal{O}(\tau^{\delta})$, we get the condition
$1-4\mathrm{Re}\,\varrho+\delta=1$, so that $\delta=4\mathrm{Re}\,\varrho$.

If $\mathrm{Re}\,\varrho\in[1/2,1)$, then $\mathrm{Re}\,\sigma\in[2,4)$, so that, as $\tau\to0$,
$\bar{z}=\mathcal{O}(\tau^{2-\sigma})\to\infty$ and $z=\mathcal{O}(\tau^{2+\sigma})\to0$.
Then, from the asymptotic expansion~\eqref{app:eq:derivation-level3n}, we see again that, in the
expansions~\eqref{eq:Asympt0u-} and \eqref{eq:Asympt0u+}, the $\mathcal{O}(\tau)$ term is omitted. The order of
the leading term in the asymptotics~\eqref{eq:Asympt0u-} and \eqref{eq:Asympt0u+} is $\tau^{-3+4\varrho}$;
denoting the correction term as $\mathcal{O}(\tau^{\delta})$, we arrive at the condition
$-3+4\mathrm{Re}\,\varrho+\delta=1$, so that $\delta=4(1-\mathrm{Re}\,\varrho)$.
\hfill $\blacksquare$\end{remark}
%%%%%%%%%%%%%%%%%%%%%%%%%%%%%%%%%%%%%%%%%%%%%%%%%%%%%%%%%%%%%%%%%%%%%%%%%%%%%%%%%%%%%%%%%%%%%%%%%%%%%%%%%%%%%%%%%%%%%%%%%%%%
%%%%%%%%%%%%%%%%%%%%%%%%%%%%%%%%%%%%%%%%%%%%%%%%%%%%%%%%%%%%%%%%%%%%%%%%%%%%%%%%%%%%%%%%%%%%%%%%%%%%%%%%%%%%%%%%%%%%%%%%%%%%
\section{Appendix. The Complete Small-$\tau$ Regular Logarithmic\\ Asymptotic Expansion of $u(\tau)$:
Theorem~\ref{th:ln-regular-a-non0}}\label{app:sec:full0expansionLog}
\subsection{Error Correction Term of the Isomonodromy Logarithmic Asymptotics As $\tau\to0$}
\label{app:subsec:error-correction-termsLog}
The generic logarithmic asymptotic expansion as $\tau\to0$ can be written as~\footnote{\label{foot:reglogexpansion}
This expansion can be obtained by considering a formal $\sigma\to0$ limit of the general power-like asymptotic
expansion~\eqref{app:eq:0-u-expansion}.}
\begin{equation}\label{app:eq:0-u-expansionLog}
u(\tau)=\sum_{k=1}^{\infty}\tau^{2k-1}\sum_{m=0}^{2k}c_{2k-1,m}(\ln\tau)^m.
\end{equation}
This expansion is convergent in a neighborhood of $\tau=0$.
We assume that $|\arg\,\tau|<\pi$ and the principal branch of $\ln$-function is taken. In this subsection, we study
the expansion~\eqref{app:eq:0-u-expansionLog} under the assumption that $a\neq0$. This expansion is also applicable
for $a=\mi k$, $k\in\mathbb{Z}\setminus\{0\}$; however, for $a=\mi k\in\mi\mathbb{Z}$, there are special variants of
the expansion~\eqref{app:eq:0-u-expansionLog} which we will discuss in a follow-up paper.

Substituting the expansion~\eqref{app:eq:0-u-expansionLog} into equation~\eqref{eq:dp3}, one finds:
\begin{equation}\label{eqs:c12-c11-c10}
c_{1,2}=-ab,\qquad
c_{1,1}=-abc,\qquad
c_{1,0}=-\frac{b(a^2c^2+1)}{4a},
\end{equation}
where $c\in\mathbb{C}$ is a parameter. We call the coefficients $c_{2k-1,m}$, $m=0,1\ldots,2k$, the
\emph{coefficients of level $k$}. The coefficients of level $k=1$ are given by equations~\eqref{eqs:c12-c11-c10}.
Below, we present the coefficients of levels $k=2$ and $3$:
\begin{gather*}
c_{3,4}=-2a^2b^2,\quad
c_{3,3}=-4a^2b^2(c-1),\quad
c_{3,2}=-b^2(3a^2c^2-6a^2c+4a^2+1),\label{app:eqs:c34-c33-c32}\\
c_{3,1}=-b^2(c-1)(a^2c^2-2a^2c+2a^2+1),\label{app:eq:c31}\\
c_{3,0}=-\frac{b^2}{8a^2}(a^4c^4-4a^4c^3+2a^2(4a^2+1)c^2-4a^2(2a^2+1)c+1),\label{app:eq:c30}
\end{gather*}
\begin{gather*}
c_{5,6}=-3a^3b^3,\quad
c_{5,5}=-3a^3b^3(3c-4),\quad
c_{5,4}=-\frac{ab^3}{8}(90a^2c^2-240a^2c+185a^2+18),\label{app:eqs:c56-c55-c54}
\end{gather*}
\begin{gather*}
c_{5,3}=-\frac{ab^3}{8}(60a^2c^3-240a^2c^2+(370a^2+36)c-209a^2-48),\label{app:eq:c53}\\
c_{5,2}=-\frac{3b^3}{16a}(15a^4c^4-80a^4c^3+(185a^4+18a^2)c^2-(209a^4+48a^2)c+91a^4+35a^2+3),\label{app:eq:c52}
\end{gather*}
\begin{gather*}
c_{5,1}=-\frac{b^3}{64a}(36a^4c^5-240a^4c^4+(740a^4+72a^2)c^3-(1254a^4+288a^2)c^2\\
+(1092a^4+420a^2+36)c-401a^4-258a^2-48),\label{app:eq:c51}\\
c_{5,0}=-\frac{b^3}{128a^3}(6a^6c^6-48a^6c^5+(185a^6+18a^4)c^4-(418a^6+96a^4)c^3\\+(546a^6+210a^4+18a^2)c^2
-(401a^6+258a^4+48a^2)c+128a^6+128a^4+25a^2+6)\label{app:eq:c50}
\end{gather*}
Although the number of parameters on which the coefficients $c_{2k-1,m}$
depend is one less than the number of parameters in the coefficients $b_{2k-1,m}$ studied in
Appendix~\ref{app:sec:full0expansion} ($\sigma$ is absent), they look even more
complicated.\footnote{\label{app:foot:computation-ckm}
We measured the amount of time that was required in order to compute the coefficients $c_{2k-1,m}$ for the
first 9 levels, as done for the coefficients $b_{2k-1,m}$ (see footnote~\ref{app:footnote:computation-bkm}), and
the results were quite surprising! On the old notebook, \textsc{Maple} 15  showed the quickest computation time was
about 105 seconds, while on the new notebook, \textsc{Maple} 2017 executed the same calculation in approximately
470 seconds! We also ran the same calculation on the new notebook using \textsc{Maple} 2022, and its fastest
computation time was roughly 137 seconds, whilst for the calculation discussed in
footnote~\ref{app:footnote:computation-bkm}, \textsc{Maple} 2022 executed it nearly 15 seconds slower than
\textsc{Maple} 2017! The only thing we can confirm  is the fact that the length of the \textsc{Maple} output for
the coefficients $b_{17,0}$ and $c_{17,0}$, when compared side-by-side, look very similar; both coefficients require
about 36 lines to display.}
As a result of the computation of the coefficients $c_{2k-1,m}$, one can formulate the following conjecture.
\begin{conjecture}\label{app:con:c2k-1m}
\begin{equation*}\label{app:eq:c2k-1m}
c_{2k-1,m}=b^kP_k(a,c)a^{2\lfloor\frac{m+1}{2}\rfloor-k},
\end{equation*}
where $P_k(a,c)$ is a polynomial in the two variables $a$ and $c$ such that
$P_k(0,c)$ is a non-vanishing polynomial in $c$, $\deg_cP_k(a,c)=2k-m$, and
$\deg_aP_k(a,c)=2k-2\lfloor\frac{m+1}2\rfloor$.
\end{conjecture}

Of course, in order to justify the existence of the expansion~\eqref{app:eq:0-u-expansionLog}, one has to derive
and study the recursion relation for the coefficients; this relation establishes the existence of the
expansion~\eqref{app:eq:0-u-expansionLog} for all $a\in\mathbb{C}\setminus\{0\}$ and
$c\in\mathbb{C}$.\footnote{\label{app:foot:a=0} The logarithmic asymptotics as $\tau\to0$ for $a=0$ are obtained
in \cite{KitVar2004} and simplified in \cite{KitVarZapiski2024}. The full asymptotic expansion, together with the
corresponding super-generating function, will appear in a follow-up paper.}
In Appendix~\ref{app:subsec:coefficients-expansionLog2} below, we show how one can corroborate
a similar statement for the other logarithmic expansion~\eqref{app:eq:0-u-expansionLog2} without having to write
an explicit formula for the recursion relation; however, the proof of Conjecture~\ref{app:con:c2k-1m} may turn out to
be more complicated despite the fact that the recursion relation would be presented explicitly.
\begin{remark}\label{rem:errorTh61}
We now verify the error-correction term for the asymptotics of the functions
$u(\tau)$ and $\me^{\mi\varphi(\tau)}$ given in Theorem~\ref{th:ln-regular-a-non0}. Note that the error-correction
term originally obtained for these asymptotics as $\tau\to0$ was $\mathcal{O}(\tau^{\delta})$
\cite{KitVar2004,KitVarZapiski2024}. The expansion~\eqref{app:eq:0-u-expansionLog} allows one to obtain a more precise
estimate for the error presented in Theorem~\ref{th:ln-regular-a-non0}.
Consider the solution $u(\tau)$ defined by the expansion~\eqref{app:eq:0-u-expansionLog} with the same parameter $c$
(cf. equation~\eqref{eq:def-c-log}) as in the asymptotics for $u(\tau)$ given in Theorem~\ref{th:ln-regular-a-non0};
both asymptotics, therefore, have the same leading-order behaviour.
The function $u(\tau)$ defined by the asymptotic expansion~\eqref{app:eq:0-u-expansionLog} is a solution of
equation~\eqref{eq:dp3} and thus corresponds to some point on the monodromy manifold: the solution $u(\tau)$ in
Theorem~\ref{th:ln-regular-a-non0} corresponds to the same point since the remaining points
on the monodromy manifold correspond to solutions with different asymptotic behaviours, as follows from the
results of this paper. The coincidence of the solutions implies that the correction term in the
asymptotics~\eqref{eq:ln-regular-u-factor} can be determined by referring to expansion~\eqref{app:eq:0-u-expansionLog}.
The error estimate in the asymptotics~\eqref{eq:ln-regular-varphi} for $\me^{\mi\varphi(\tau)}$ is obtained by
integrating equation~\eqref{eq:varphi}.
\hfill $\blacksquare$\end{remark}
%%%%%%%%%%%%%%%%%%%%%%%%%%%%%%%%%%%%%%%%%%%%%%%%%%%%%%%%%%%%%%%%%%%%%%%%%%%%%%%%%%%%%%%%%%%%%%%%%%%%%%%%%%%%%%%%%%%%%%%%%%%%%
%%%%%%%%%%%%%%%%%%%%%%%%%%%%%%%%%%%%%%%%%%%%%%%%%%%%%%%%%%%%%%%%%%%%%%%%%%%%%%%%%%%%%%%%%%%%%%%%%%%%%%%%%%%%%%%%%%%%%%%%%%%%%
\subsection{Super-Generating Function for the Regular Logarithmic Asymptotics}
\label{app:subsec:super-generating-functionLog}
The formal construction of the super-generating function for the coefficients of the
expansion~\eqref{app:eq:0-u-expansionLog} is similar to the one presented in
Subsection~\ref{app:subsec:super-generating-function}, namely,
\begin{equation}\label{app:eq:supergenfunctonLog}
\hat{A}(x,y)=\sum_{n=0}^{\infty}y^n\hat{A}_{n}(x),
\end{equation}
where the coefficient functions $\hat{A}_n(x)$, $n\in\mathbb{Z}_{\geqslant0}$, are generating functions for
the coefficients $c_{2k-1,2k-n}$:
\begin{gather}
\hat{A}_0(x)=\sum_{k=1}^{\infty}c_{2k-1,2k}x^k,\label{app:eq:A0-seriesLog}\\
\hat{A}_n(x)=\sum_{k=\lfloor(n-1)/2\rfloor+1}^{\infty}c_{2k-1,2k-n}x^k,\qquad
n\geqslant1.\label{app:eq:hatAn-series}
\end{gather}
Define the linear differential operator $\hat{D}$ acting in the space of formal power
series of two variables $x$ and $y$ as follows:
\begin{equation}\label{app:eq:hatD-definition}
\hat{D}:=2x(1+y)\frac{\partial}{\partial x}-y^2\frac{\partial}{\partial y};
\end{equation}
then, the function $\hat{A}\equiv\hat{A}(x,y)$ solves the PDE
\begin{equation}\label{app:eq:hatA-PDE}
\hat{D}^2(\ln\hat{A})=-8\hat{A}+2a\frac{bxy^2}{\hat{A}}+\left(\frac{bxy^2}{\hat{A}}\right)^2.
\end{equation}
Assuming that a proper solution of the PDE~\eqref{app:eq:hatA-PDE} is constructed, one can obtain the solution of
equation~\eqref{eq:dp3} via the relation
\begin{equation}\label{app:eq:u-hatA-relation}
u(\tau)=\frac1{\tau}\hat{A}\big(\tau^2\ln^2\tau,(\ln\tau)^{-1}\big).
\end{equation}

Now, in a manner similar to the one taken for the functions $A_n(x)$ in
Subsection~\ref{app:subsec:super-generating-function}, we show that this construction leads to explicit formulae
for the functions $\hat{A}_n(x)$.

For $n=0$, we get exactly the same equation (cf. equation~\eqref{app:eq:A0}) for $\hat{A}_0(x)$ that we got for
the function $A_0(x)$, but with $\sigma=0$,
\begin{equation}\label{app:eq:hatA0}
\hat{D}_x^2\ln\hat{A}_0(x)=-8\hat{A}_0(x),
\end{equation}
where
\begin{equation}\label{app:eq:hatDx}
\hat{D}_x:=2x\frac{\md}{\md x}
\end{equation}
is the $x$-part of the operator $\hat{D}$. The analysis of this equation is even simpler
than the one for equation~\eqref{app:eq:A0}, where, for equation~\eqref{app:eq:hatA0}, we have to choose the rational
solution
\begin{equation}\label{app:eq:hatA0explicit}
\hat{A}_0(x)=-\frac{Cx}{(1-Cx)^2},
\end{equation}
where $C$ is the constant of integration. To determine $C$, one expands the right-hand side of
equation~\eqref{app:eq:hatA0explicit} into a Taylor series centred at $x=0$,
\begin{equation*}\label{app:eq:A0series}
\hat{A}_0(x)=-\sum_{k=1}^{\infty}k(Cx)^k,
\end{equation*}
and compares it with the series~\eqref{app:eq:A0-seriesLog}; then,
\begin{equation}\label{app:eq:C}
C=-c_{1,2}=ab,
\end{equation}
where, for the latter equation, we used the first relation in \eqref{eqs:c12-c11-c10}; thus, we've calculated all
the coefficients
\begin{equation*}\label{app:eq:c2k-1,2k-explicit}
c_{2k-1,2k}=-k(ab)^k,\qquad
k\in\mathbb{N}.
\end{equation*}

To calculate the next generating function $\hat{A}_1(x)$, substitute the expansion~\eqref{app:eq:supergenfunctonLog}
into equation~\eqref{app:eq:hatA-PDE} and equate the terms that are linear in $y$ on both sides of the resulting
equation,
\begin{equation}\label{app:eq:hatA1}
\hat{D}_x^2\left(\frac{\hat{A}_1(x)}{\hat{A}_0(x)}\right)+8\hat{A}_1(x)+2\hat{D}_x^2\ln\hat{A}_0(x)=0,
\end{equation}
where $\hat{A}_0(x)$ is given in equation~\eqref{app:eq:hatA0explicit}; thus, equation~\eqref{app:eq:hatA1} is the
inhomogeneous degenerate hypergeometric equation with general solution
\begin{equation*}\label{app:eq:hatA1-generalsolution}
\hat{A}_1(x)=-\frac{x(1+Cx)C_1}{(1-Cx)^3}-\frac{x(4+\ln(x)+Cx\ln(x))C_0}{(1-Cx)^3}-\frac{4Cx}{(1-Cx)^3},
\end{equation*}
where $C_0$ and $C_1$ are constants of integration. One sets $C_0=0$ because the expansion~\eqref{app:eq:hatAn-series}
for $n=1$ does not contain any logarithmic terms; therefore,
\begin{equation*}\label{app:eq:hatA1-specialsolution}
\hat{A}_1(x)=-\frac{x(1+Cx)C_1}{(1-Cx)^3}-\frac{4Cx}{(1-Cx)^3}\underset{x\to0}{=}-x(C_1+4C)+\mathcal{O}\big(x^2\big).
\end{equation*}
Comparing the latter expansion with the definition of $\hat{A}_1(x)$ (cf. equation~\eqref{app:eq:hatAn-series} for
$n=1$ and the second equation in \eqref{eqs:c12-c11-c10}), we find $-C_1-4C=c_{1,1}=-abc=-Cc$, so that $C_1=C(c-4)$.
It is convenient to decompose $\hat{A}_1(x)$ into partial fractions:
\begin{equation}\label{app:eq:hatA1-partialfractions}
\hat{A}_1(x)=-\frac{c-4}{1-C x}+\frac{3c-8}{(1-C x)^2}-\frac{2(c-2)}{(1-C x)^3}.
\end{equation}
Expanding each fraction in equation~\eqref{app:eq:hatA1-partialfractions} into Taylor series centred at $x=0$,
one finds that
\begin{equation}\label{app:eq:hatA1-Taylor}
%\hat{A}_1(x)=\sum_{k=0}^{\infty}(-(c-4)+(3c-8)(k+1)-(c-2)(k+2)(k+1))(Cx)^k
\hat{A}_1(x)=-\sum_{k=1}^{\infty}k(ck-2(k-1))(Cx)^k.
\end{equation}
Comparing, now, the expansions~\eqref{app:eq:hatAn-series} and \eqref{app:eq:hatA1-Taylor}, and taking into account
\eqref{app:eq:C}, we get
\begin{equation*}\label{app:eq:c2k-1,2k-1}
c_{2k-1,2k-1}=-k(k(c-2)+2)(ab)^k,\qquad
k\in\mathbb{N}.
\end{equation*}

We present below, without detailed explanations, the construction for the functions $\hat{A}_2(x)$ and $\hat{A}_3(x)$.
The inhomogeneous degenerate hypergeometric equation for the function $\hat{A}_2(x)$  is
\begin{equation}\label{app:eq:hatA2}
\begin{gathered}
\hat{D}_x^2\left(\frac{\hat{A}_2(x)}{\hat{A}_0(x)}\right)+8\hat{A}_2(x)-\frac{2abx}{\hat{A}_0(x)}
-\frac12\hat{D}_x^2\left(\frac{\hat{A}_1(x)}{\hat{A}_0(x)}\right)^2
+2\hat{D}_x^2\left(\frac{\hat{A}_1(x)}{\hat{A}_0(x)}\right)-2\hat{D}_x\left(\frac{\hat{A}_1(x)}{\hat{A}_0(x)}\right)\\
+\hat{D}_x^2\ln\hat{A}_0(x)-\hat{D}_x\ln\hat{A}_0(x)=0.
\end{gathered}
\end{equation}
\begin{remark}\label{app:rem:choice-of-C}
A rational solution of equation~\eqref{app:eq:hatA2} exists if and only if the parameter $C$ in $\hat{A}_0(x)$ is
chosen as per equation~\eqref{app:eq:C}; therefore, it was not necessary to calculate $c_{1,2}$ in order to determine
the value of the parameter $C$, and the possibility of the continuation of our construction dictates the correct
value of $C$. Henceforth, we assume that $C=ab$.
\hfill $\blacksquare$\end{remark}
With reference to Remark~\ref{app:rem:choice-of-C}, one finds the one-parameter ($C_2$) rational solution of
equation~\eqref{app:eq:hatA2} which, when decomposed into partial fractions, reads
\begin{equation}\label{app:eq:hatA2-partialfractions}
\begin{gathered}
\hat{A}_2(x)=-\frac12+\frac{abx}{8}-\frac{4C_2-11}{4(1-abx)}+\frac{12C_2-8c^2+24c-35}{4(1-abx)^2}\\
-\frac{4C_2-10c^2+36c-37}{2(1-abx)^3}-\frac{3(c-2)^2}{(1-abx)^4}.
\end{gathered}
\end{equation}
Expand $\hat{A}_2(x)$ into a Taylor series centred at $x=0$:
\begin{equation}\label{app:eq:A2Taylor1}
\hat{A}_2(x)=-\frac{ab}{8}(8C_2+8c^2-48c+57)x+\mathcal{O}\big(x^2\big).
\end{equation}
Comparing the leading coefficient of the expansion~\eqref{app:eq:A2Taylor1} with the analogous one in
the expansion~\eqref{app:eq:hatAn-series}, one finds
\begin{equation}\label{app:eq:A2C2}
-\frac{ab}{8}(8C_2+8c^2-48c+57)=c_{1,0}=-\frac{b(a^2c^2+1)}{4a};
\end{equation}
thus, solving equation~\eqref{app:eq:A2C2} for $C_2$, we obtain
\begin{equation}\label{app:eq:C2explicit}
C_2= -\frac{1}{8a^2}(6a^2c^2-48a^2c+57a^2-2).
\end{equation}
Now, using the partial fraction expansion~\eqref{app:eq:hatA2-partialfractions} for the function $\hat{A}_2(x)$,
we find, after a straightforward calculation, that, for $k\geqslant2$,
\begin{equation}\label{app:eq:c[2k-1,2k-2]}
c_{2k-1,2k-2}=-\frac{(ab)^k}{8}\left(2k^2(2k-1)(c-2)^2+8k(2k-1)(c-2)+(k+2)(5k-2)+\frac{2k^2}{a^2}\right).
\end{equation}
\begin{remark}\label{app:rem:C2free}
The initial coefficient $c_{1,0}$ has the off-set value (cf. equation~\eqref{eqs:c12-c11-c10}); however, its value
allows us to determine the remaining ``regular'' coefficients of the series \eqref{app:eq:c[2k-1,2k-2]}. It seems
that there should be some other idea that would allow one to fix the coefficients without having to appeal to
\emph{a priori} calculated coefficients for small values of $k$.
This idea has already been demonstrated in Remark~\ref{app:rem:choice-of-C}, where it was explained that the
determination of the value of $C$ is the crucial issue for the existence of the
expansion~\eqref{app:eq:supergenfunctonLog} with rational coefficient functions. To check whether or not this idea
works, we continue with the calculation of the function $\hat{A}_3(x)$, but without reference to the particular value
of $C_2$ obtained above (cf. equation~\eqref{app:eq:C2explicit}).
\hfill $\blacksquare$\end{remark}
To simplify the notation in some of the formulae below, we write $\hat{A}_k(x)=\hat{A}_k$ for $k=0,1,2,3,4$.
With this notation, the equation for the determination of the function $\hat{A}_3(x)$ reads:
\begin{equation}\label{app:eq:hatA3}
\begin{gathered}
\hat{D}_x^2\left(\frac{\hat{A}_3}{\hat{A}_0}\right)+8\hat{A}_3+2\frac{\hat{A}_1}{\hat{A}_0}+2abx\frac{\hat{A}_1}{\hat{A}_0^2}
+\hat{D}_x^2\left(\frac{\hat{A}_1}{\hat{A}_0}\right)-3\hat{D}_x\left(\frac{\hat{A}_1}{\hat{A}_0}\right)\\
+\hat{D}_x^2\left(\frac13\left(\frac{\hat{A}_1}{\hat{A}_0}\right)^3-\frac{\hat{A}_1\hat{A}_2}{\hat{A}_0^2}\right)
+2\hat{D}_x^2\left(\frac{\hat{A}_2}{\hat{A}_0}-\frac12\left(\frac{\hat{A}_1}{\hat{A}_0}\right)^2\right)
-4\hat{D}_x\left(\frac{\hat{A}_2}{\hat{A}_0}-\frac12\left(\frac{\hat{A}_1}{\hat{A}_0}\right)^2\right)=0.
\end{gathered}
\end{equation}
Equation~\eqref{app:eq:hatA3} has, for any value of the parameter $C_2\in\mathbb{C}$, a rational solution that
depends on
%BELOW IS ALSO THE CORRECT FORMULA FOR $A_3$, but I decided to give a decomposition into partial fractions
%a parameter $C_3$, namely,
%\begin{equation}\label{app:eq:A3rational}
%\begin{gathered}
%\hat{A}_3(x)=-\frac{abx(1+abx)C_3}{(1-abx)^3}+\frac{abx}{8(1-abx)^5}\left((abx)^5-4(c-3/4)(abx)^4-48(c^2\right.\\
%+(-3/2-(2/3)C_2)c-1/6+(2/3)C_2)(abx)^2-24(c^3-6c^2+((2/3)C_2+23/6)c\\
%+35/8+(4/3)C_2)abx
%\left.
%-8c^3+96c^2+(-360-16C_2)c+349+64C_2\right).
%\end{gathered}
%\end{equation}
a parameter $C_3$. The partial fraction decomposition of $\hat{A}_3(x)$ is
\begin{equation}\label{app:eq:A3partialfractions}
\begin{gathered}
\hat{A}_3(x)=-\frac{abx}{8}+\frac{c}{2}-1-\frac{4c^3-24c^2+48c-32}{(1-abx)^5}\\
+\frac{14c^3-60c^2+87c-46-12cC_2+24 C_2}{2(1-abx)^4}-\frac{6 c^3-21c+9+4C_3-20cC_2+32C_2}{2(1-abx)^3}\\
+\frac{24c^2-16c-29+12C_3-16cC_2+16C_2}{4(1-abx)^2}-\frac{10c-15+4C_3}{4(1-abx)}.
\end{gathered}
\end{equation}
Consider the first terms of the Taylor expansion at $x=0$ of the function $\hat{A}_3(x)$,
\begin{equation}\label{app:eq:hatA3Taylor2}
\begin{aligned}
\hat{A}_3(x)=&-\frac{ab}{8}(8C_3+8c^3-96c^2+360c-349+16cC_2-64C_2)x\\
&-\frac{(ab)^2}{2}(8C_3+16c^3-156c^2+473c-410+24cC_2-72C_2)x^2+\mathcal{O}\big(x^3\big),
\end{aligned}
\end{equation}
and compare it with its definition~\eqref{app:eq:hatAn-series} for $n=3$; the term proportional to $x$
vanishes, which allows us to get the constant $C_3$ in terms of $C_2$:
\begin{equation}\label{app:eq:C3inC2}
C_3=(8-2c)C_2+349/8-45c+12c^2-c^3.
\end{equation}
At this stage of the calculation, we continue to assume that $C_2$ has not been determined. We carry forward with our
construction and try to get a rational function $\hat{A}_4(x)$ which solves equation~\eqref{app:eq:An-diffeq}
for $n=4$ (see below). In case one uses either \textsc{Maple} or \textsc{Mathematica}, finding a general solution
to this equation is not problematic, and the solution for $n=4$ is obtained almost immediately; however, we see
that the rational solution of this equation exists iff the constant of integration $C_2$ is fixed according to
equation~\eqref{app:eq:C2explicit}, in which case, the final form of the function $\hat{A}_3(x)$ reads
\begin{equation}\label{app:eq:A3parfracFINAL}
\begin{gathered}
\hat{A}_3(x)=-\frac{abx}{8}+\frac{c}{2}-1-\frac{4(c-2)^3}{(1-abx)^5}
+\frac{(c-2)\big((46c^2-208c+217)a^2-6\big)}{4a^2(1-abx)^4}\\
-\frac{(46c^3-336c^2+765c-545)a^2-14c+32}{4a^2(1-abx)^3}
+\frac{(36c^3-312c^2+802c-607)a^2-20c+56}{8a^2(1-abx)^2}\\
-\frac{(4c^3-48c^2+158c-137)a^2-4c+16}{8a^2(1-abx)}.
\end{gathered}
\end{equation}
Using the representation~\eqref{app:eq:A3parfracFINAL} for the function $\hat{A}_3(x)$, we obtain the general formula
for the coefficients $c_{2k-1,2k-3}$:
\begin{equation}\label{app:eq:c2k-1,2k-3}
\begin{gathered}
c_{2k-1,2k-3}=-\frac{(ab)^k}{2}\left(\frac{k^2(2k-1)(k-1)}{6}(c-2)^3+k(2k-1)(k-1)(c-2)^2\right.\\
\left.+(k-1)\left(\frac{5k^2}{4}+2k-1+\frac{k^2}{2a^2}\right)(c-2)+k\left(\frac{k-1}{a^2}+\frac{9k-10}{4}\right)\right),\qquad
k\geqslant2.
\end{gathered}
\end{equation}
As stated above, only if the constant of integration $C_2$ in the rational function $\hat{A}_3(x)$ is chosen as per
equation \eqref{app:eq:C2explicit} can one find a rational solution of equation~\eqref{app:eq:An-diffeq} for $n=4$:
\begin{equation}\label{app:eq:hatA4partialfractions}
\begin{gathered}
\hat{A}_4(x)=-\left(1+\frac{4}{a^2}\right)\frac{(abx)^3}{256}+\left(17+\frac{44}{a^2}\right)\frac{(abx)^2}{576}
+\left(144c-469-\frac{388}{a^2}\right)\frac{abx}{2304}\\
-\frac{3}{16}\left(2c^2-8c+3-\frac{2}{a^2}\right)-\frac{5(c-2)^4}{(1-abx)^6}
+\frac{(c-2)^2\big((36c^2-160c+161)a^2-6\big)}{2a^2(1-abx)^5}\\
-\frac{(1564c^4-14176c^3+46212c^2-64352c+32315)a^4-12(50c^2-216c+227)a^2+12}{64a^4(1-abx)^4}\\
+\frac{\big(72(243c^4-2432c^3+8489c^2-12238c)+441091\big)a^4-4(2988c^2-14400c+16265)a^2+504}{1152a^4(1-abx)^3}\\
-\frac{\big(48(65c^4-760c^3+2965c^2-4523c)+107041\big)a^4-4(888c^2-5088c+6491)a^2+240}{768a^4(1-abx)^2}\\
+\frac{(720c^4-11520c^3+56880c^2-98640c+46009)a^4-4(360c^2-2880c+4763)a^2+144}{2304a^4(1-abx)}.
\end{gathered}
\end{equation}
Developing, with the help of equation~\eqref{app:eq:hatA4partialfractions}, the function $\hat{A}_4(x)$ into a Taylor
series about $x=0$ and comparing the resulting expansion with the definition of the function $\hat{A}_n(x)$
for $n=4$ (cf. equation~\eqref{app:eq:hatAn-series}), one finds, for $k=2$ and $k=3$, two off-set coefficients,
\begin{equation}\label{app:eq:c3,0:c5,2}
\begin{gathered}
c_{3,0}=-\frac{(ab)^2}{8}\left(\frac{1}{a^4}+\frac{2c(c-2)}{a^2}+c(c-2)(c^2-2c+4)\right),\\
c_{5,2}=-\frac{3(ab)^3}{16}\left(\frac{3}{a^4}+\frac{18c^2-48c+35}{a^2}+15c^4-80c^3+185c^2-209c+91\right),
\end{gathered}
\end{equation}
and the general formula
\begin{equation}\label{app:eq:c2k-1,2k-4}
\begin{gathered}
c_{2k-1,2k-4}=-\frac{(abx)^k}{96}\bigg(k^2(k-1)(2k-1)(2k-3)(c-2)^4+8k(k-1)(2k-1)(2k-3)(c-2)^3\\
+3(k-1)(2k-3)\left(5k^2+8k-4+\frac{2k^2}{a^2}\right)(c-2)^2+6k(2k-3)\left(\frac{4(k-1)}{a^2}+9k-10\right)(c-2)\\
+3\left(\frac52+\frac1{a^2}\right)^2k^3-\left(\frac{3}{a^4}+\frac{43}{6a^2}+\frac{55}{24}\right)k^2
-6\left(19+\frac{10}{a^2}\right)k+9\left(10+\frac{4}{a^2}\right)\!\!\bigg),
\quad
k=4,5,6,\dotsc.
\end{gathered}
\end{equation}

In general, for $n=0,1,2,\ldots$, the functions $\hat{A}_n(x)\equiv\hat{A}_n$ are defined as rational solutions of
the following degenerate hypergeometric equation,
\begin{equation}\label{app:eq:An-diffeq}
\begin{gathered}
\hat{D}_x^2(f_n+2f_{n-1}+f_{n-2})-\hat{D}_x\big(2(n-1)f_{n-1}+(2n-3)f_{n-2}\big)+(n-1)(n-2)f_{n-2}+8\hat{A}_n\\
=\frac{2abx}{\hat{A}_0}\sum_{l=0}^{n-2}(-1)^l\!\!\sum_{\substack{p_1+\ldots+p_l=n-2\\p_i\geqslant1,\;i=1,\ldots,l}}
\hat{A}_{p_10}\cdot\ldots\cdot\hat{A}_{p_l0}
+\frac{b^2x^2}{\hat{A}_0^2}\sum_{l=1}^{n-4}(-1)^l(l+1)
\sum_{\substack{p_1+\ldots+p_l=n-4\\p_i\geqslant1,i=1,\ldots,l}}\tilde{A}_{p_10}\cdot\ldots\cdot\tilde{A}_{p_l0},
\end{gathered}
\end{equation}
whose Taylor-series expansions start with the term $x^{\lfloor(n-1)/2\rfloor+1}$, and where,
in equation~\eqref{app:eq:An-diffeq}, we adopted the notations
\begin{equation}\label{app:eq:hatAp0-fk-def}
\hat{A}_{p0}:=\frac{\hat{A}_{p}}{\hat{A}_{0}},\quad
p\in\mathbb{N},\qquad
\mathrm{and}\qquad
f_n=\sum_{l=1}^{n}\frac{(-1)^{l-1}}{l}\sum_{\substack{p_1+\ldots+p_l=n\\p_i\geqslant1,\;i=1,\ldots,l}}
\hat{A}_{p_10}\cdot\ldots\cdot\hat{A}_{p_l0}.
\end{equation}
Since, for $n>4$, the value of the inner sums in both of the double sums vanish for $l=0$, it follows that the
summation in both of the double sums  actually begins at $l=1$. For $n=0,1,2,3,4$, equation~\eqref{app:eq:An-diffeq}
remains valid. In this case, though, one has to assume that  $f_{-2}=f_{-1}=0$ and $f_0=\ln\hat{A}_0$; moreover,
the following---natural---conventions regarding the summation in the double sums are made: (i) if the upper limit
of the external sum of a double sum in equation~\eqref{app:eq:An-diffeq} is negative, then the double sum equals
zero; (ii) if the upper limit of the external sum is zero, then the corresponding double sum equals $1$; and (iii)
if the upper limit of the external sum is positive, then the value of the inner sum corresponding to $l=0$ vanishes,
so that the summation starts at $l=1$.
\begin{remark}\label{rem:app:calculationsLog1}
The explicit formula for the function $\hat{A}_4(x)$ (cf. equation~\eqref{app:eq:hatA4partialfractions}) is rather
cumbersome, so that the question arises as to whether or not the procedure can effectively be continued to construct
the functions $\hat{A}_n(x)$ for larger values of $n$, especially when the growth in size with respect to $n$
of the inhomogeneous part of the ODE~\eqref{app:eq:An-diffeq} is taken into account. Inherent in such calculations is
the substantial problem of having to store into the memory of the program all of the previously obtained results,
that is, the functions $\hat{A}_p(x)$ for $p<n$. We calculated $\hat{A}_n(x)$ for $n=5$ and $6$: the calculation of
each function, modulo the previously obtained functions, was completed within 2 seconds. We assume that such
straightforward calculations on a standard, modern laptop can be performed in a reasonable time frame (a few hours),
at least for values of $n$ up to $15$.
\hfill $\blacksquare$\end{remark}
\begin{remark}\label{rem:app:specialSolution:c=2}
The partial fraction decomposition of the generating function $\hat{A}_n(x)$ consists of fractions with
denominators of the form $(1-abx)^k$ with $k\leqslant n+2$; this leads to the fact that, for generic values of
the parameters, the coefficients $c_{2k-1,2k-n}$ are polynomials in $k$ of degree $n+1$.
It is easy to see that the residues of the partial fractions beginning with the fraction with highest order
$k=n+2$ in the denominator are successively proportional to $(c-2)^n$, $(c-2)^{n-2}$, etc., up to powers
of $(c-2)$ that remain positive.
For $c=2$, therefore, the partial fractions corresponding to the first $\lfloor\tfrac{n+1}{2}\rfloor$ members of the
pole expansion of $\hat{A}_n(x)$ vanish. Thus, for $c=2$, the order of the pole of $\hat{A}_n(x)$
is $n+2-\lfloor\tfrac{n+1}{2}\rfloor$, which means that the coefficients $c_{2k-1,2k-n}$ are polynomials in $k$
of $\deg_kc_{2k-1,2k-n}=n+1-\lfloor\tfrac{n+1}{2}\rfloor$.

An interesting observation is the fact that $c$ depends on the scaling parameter $b$
(cf. equation~\eqref{eq:def-c-log}), which means that, for any solution having regular logarithmic behaviour as
$\tau\to0$, one can choose a proper scaling for which the growth of the coefficients of the
corresponding asymptotic expansion achieves its minimal possible value.
\hfill $\blacksquare$\end{remark}
%%%%%%%%%%%%%%%%%%%%%%%%%%%%%%%%%%%%%%%%%%%%%%%%%%%%%%%%%%%%%%%%%%%%%%%%%%%%%%%%%%%%%%%%%%%%%%%%%%%%%%%%%%%%%%%%%%%%%%%%%%%%%
%%%%%%%%%%%%%%%%%%%%%%%%%%%%%%%%%%%%%%%%%%%%%%%%%%%%%%%%%%%%%%%%%%%%%%%%%%%%%%%%%%%%%%%%%%%%%%%%%%%%%%%%%%%%%%%%%%%%%%%%%%%%%
\section{Appendix. The Complete Small-$\tau$ Irregular Logarithmic\\ Asymptotic Expansion of $u(\tau)$:
Theorems~\ref{th:Asympt0-ln-rho12-} and \ref{th:Asympt0-ln-rho12+}${}^{\mathbf\prime}$}
\label{app:sec:full0expansionLog2}
\subsection{Asymptotic Expansion and Properties of its Coefficients}\label{app:subsec:coefficients-expansionLog2}
Bearing in mind the application of B\"acklund transformations to the expansion~\eqref{app:eq:0-u-expansionLog},
the complete logarithmic asymptotic expansion related with the leading term of asymptotics given in
Theorems~\ref{th:Asympt0-ln-rho12-} and \ref{th:Asympt0-ln-rho12+}${}^{\mathbf\prime}$ is
\begin{equation}\label{app:eq:0-u-expansionLog2}
u(\tau)=\sum_{k=0}^{+\infty}\tau^{2k-1}\sum_{m=-2\lfloor k/2\rfloor}^{+\infty}\tilde{c}_{2k-1,m}(\ln\tau)^{-m},\qquad
\tilde{c}_{-1,0}=\tilde{c}_{-1,1}=0,\quad
\tilde{c}_{-1,2}=-\frac14.
\end{equation}
The expansion~\eqref{app:eq:0-u-expansionLog2} depends on the single parameter $\tilde{c}_{-1,3}\in\mathbb{C}$ and
is convergent in a neighbourhood of $\tau=0$; furthermore, as in the expansion~\eqref{app:eq:0-u-expansion},
we assume that $|\arg\,\tau|<\pi$ and the principle branch of $\ln$-function is chosen.
The expansion~\eqref{app:eq:0-u-expansionLog2} is valid for all values of the parameter of formal monodromy
$a\in\mathbb{C}$, although for $a=\mi k$, $k\in\mathbb{Z}\setminus\{0\}$, there are some simplifications that we
address in an upcoming paper.
We say that the coefficients $\tilde{c}_{2k-1,m}$ are the \emph{coefficients of level $k$}. In contrast to the cases
considered in Appendices~\ref{app:sec:full0expansion} and \ref{app:sec:full0expansionLog}, for the present case, the
number of coefficients that belong to each level in the generic situation is infinite; there is, however, one
interesting special case, $\tilde{c}_{-1,3}=0$, for which all the levels are finite.

Below, we present explicit expressions for the coefficients of the first few levels that were calculated by
directly substituting the expansion~\eqref{app:eq:0-u-expansionLog2} into equation~\eqref{eq:dp3} with
$\varepsilon=1$:\footref{foot:epsilon}
\begin{equation}\label{app:eq:tilde-c-1m}
\pmb{\text{Level}\;k=0:}\quad
\tilde{c}_{-1,m}=(-1)^{m-1}2^{m-4}(m-1)\,\tilde{c}_{-1,3}^{m-2},\qquad
m\in\mathbb{N}.
\end{equation}
Numerically, we obtained only the first few members of this sequence: a generic formula can be verified  with the
help of the generating function studied in the next subsection.
\begin{equation}\label{app:eq:tilde-c1m}
\begin{gathered}
\pmb{\text{Level}\;k=1:}\quad
\tilde{c}_{1,0}=\frac{ab}{2},\quad
\tilde{c}_{1,1}=-ab,\quad
\tilde{c}_{1,2}=ab\,(2\,\tilde{c}_{-1,3}+1),\\
\tilde{c}_{1,3}=-\frac{ab}{2}(8\,\tilde{c}_{-1,3}^2+8\,\tilde{c}_{-1,3}+1),\quad
\tilde{c}_{1,4}= ab\,\tilde{c}_{-1,3}(8\,\tilde{c}_{-1,3}^2+12\,\tilde{c}_{-1,3}+3),\\
\tilde{c}_{1,5}=-4ab\,\tilde{c}_{-1,3}^2(2\tilde{c}_{-1,3}+3)(2\tilde{c}_{-1,3}+1),\quad
\tilde{c}_{1,6}=8ab\,\tilde{c}_{-1,3}^3(4\,\tilde{c}_{-1,3}^2+10\,\tilde{c}_{-1,3}+5),\\
\tilde{c}_{1,7}=-8ab\,\tilde{c}_{-1,3}^4(8\,\tilde{c}_{-1,3}^2+24\,\tilde{c}_{-1,3}+15),\quad
\tilde{c}_{1,8}=16ab\,\tilde{c}_{-1,3}^5(8\,\tilde{c}_{-1,3}^2+28\,\tilde{c}_{-1,3}+21),\\
\tilde{c}_{1,9}=-128ab\,\tilde{c}_{-1,3}^6(2\,\tilde{c}_{-1,3}^2+8\,\tilde{c}_{-1,3}+7),\quad
\tilde{c}_{1,10}=256ab\,\tilde{c}_{-1,3}^7(\tilde{c}_{-1,3}+3)(2\tilde{c}_{-1,3}+3);
\end{gathered}
\end{equation}
\begin{equation}\label{app:eq:tilde-c3m}
\begin{gathered}
\pmb{\text{Level}\;k=2:}\quad
\tilde{c}_{3,-2}=-\frac{b^2(a^2+1)}{4},\quad
\tilde{c}_{3,-1}=-b^2\big((a^2+1)\tilde{c}_{-1,3}-a^2-1/2\big),\\
\tilde{c}_{3,0}=-b^2\Big((a^2+1)\big(\tilde{c}_{-1,3}^2-2\,\tilde{c}_{-1,3}+71/32\big)+\tilde{c}_{-1,3}-13/8\Big),\quad
\tilde{c}_{3,1}=\frac{3b^2(37a^2+5)}{32},\\
\tilde{c}_{3,2}=-\frac{b^2}{64}(444a^2\tilde{c}_{-1,3}+239a^2+60\,\tilde{c}_{-1,3}+15),\\
\tilde{c}_{3,3}=\frac{b^2}{256}(3552a^2\tilde{c}_{-1,3}^2+3824a^2\tilde{c}_{-1,3}+623a^2+480\,\tilde{c}_{-1,3}^2
+240\,\tilde{c}_{-1,3}+15),\\
\tilde{c}_{3,4}=-\frac{3b^2}{128}\big((1184\,\tilde{c}_{-1,3}^3+1912\,\tilde{c}_{-1,3}^2+623\,\tilde{c}_{-1,3}+32)a^2
+160\,\tilde{c}_{-1,3}^3+120\,\tilde{c}_{-1,3}^2+15\,\tilde{c}_{-1,3}\big),\\
\tilde{c}_{3,5}=\frac{b^2\tilde{c}_{-1,3}}{32}\big((1776\,\tilde{c}_{-1,3}^3\!+3824\,\tilde{c}_{-1,3}^2\!+1869\,\tilde{c}_{-1,3}\!
+192)a^2+240\,\tilde{c}_{-1,3}^3\!+240\,\tilde{c}_{-1,3}^2\!+45\,\tilde{c}_{-1,3}\big),\\
\tilde{c}_{3,6}=-\frac{b^2\tilde{c}_{-1,3}^2}{16}\big((1776\,\tilde{c}_{-1,3}^3\!+4780\,\tilde{c}_{-1,3}^2\!
+3115\,\tilde{c}_{-1,3}\!+480)a^2\!+240\,\tilde{c}_{-1,3}^3\!+300\,\tilde{c}_{-1,3}^2\!+75\,\tilde{c}_{-1,3}\big),\\
\tilde{c}_{3,7}=\frac{3b^2\tilde{c}_{-1,3}^3}{16}\big((1184\,\tilde{c}_{-1,3}^3\!+3824\,\tilde{c}_{-1,3}^2\!
+3115\,\tilde{c}_{-1,3}\!+640)a^2\!+160\,\tilde{c}_{-1,3}^3\!+240\,\tilde{c}_{-1,3}^2\!+75\,\tilde{c}_{-1,3}\big),\\
\tilde{c}_{3,8}=-\frac{b^2\tilde{c}_{-1,3}^4}{8}\big((3552\,\tilde{c}_{-1,3}^3\!+13384\,\tilde{c}_{-1,3}^2\!
+13083\,\tilde{c}_{-1,3}\!+3360)a^2\\
+480\,\tilde{c}_{-1,3}^3\!+840\,\tilde{c}_{-1,3}^2\!+315\,\tilde{c}_{-1,3}\big);
\end{gathered}
\end{equation}
\begin{equation}\label{app:eq:tilde-c5:-2:-1:0}
\begin{gathered}
\pmb{\text{Level}\;k=3:}\quad
\tilde{c}_{5,-2}=\frac{ab^3(a^2+1)}{4},\quad
\tilde{c}_{5,-1}=ab^3\big((a^2+1)\tilde{c}_{-1,3} - 13a^2/8 - 9/8\big),\\
\tilde{c}_{5,0}=\frac{ab^3}{36}\big((36\,\tilde{c}_{-1,3}^2\!-117\,\tilde{c}_{-1,3}\!+176)a^2\!
+36\,\tilde{c}_{-1,3}^2-81\,\tilde{c}_{-1,3}+83\big),\\
\end{gathered}
\end{equation}
\begin{gather*}\label{app:eq:tilde-c5m}
\tilde{c}_{5,1}=-\frac{ab^3}{864}(7685a^2+2309),\quad
\tilde{c}_{5,2}=\frac{ab^3}{10368}\big((184440\,\tilde{c}_{-1,3}\!+111659)a^2\!+55416\,\tilde{c}_{-1,3}\!+20171\big),\\
\tilde{c}_{5,3}=-\frac{ab^3}{7776}\big((276660\,\tilde{c}_{-1,3}^2\!+334977\,\tilde{c}_{-1,3}\!+67630)a^2\!
+83124\,\tilde{c}_{-1,3}^2\!+60513\,\tilde{c}_{-1,3}\!+6622\big),\\
\tilde{c}_{5,4}=\frac{ab^3}{20736}\big((1475520\,\tilde{c}_{-1,3}^3\!+2679816\,\tilde{c}_{-1,3}^2\!
+1082080\,\tilde{c}_{-1,3}\!+89181)a^2\!\\
+443328\,\tilde{c}_{-1,3}^3\!+484104\,\tilde{c}_{-1,3}^2\!+105952\,\tilde{c}_{-1,3}\!+3645\big),\\
\tilde{c}_{5,5}=-\frac{ab^3}{2592}\big((368880\,\tilde{c}_{-1,3}^4\!+893272\,\tilde{c}_{-1,3}^3\!
+541040\,\tilde{c}_{-1,3}^2\!+89181\,\tilde{c}_{-1,3}\!+2592)a^2\\
+110832\,\tilde{c}_{-1,3}^4\!+161368\,\tilde{c}_{-1,3}^3\!+52976\,\tilde{c}_{-1,3}^2\!+3645\,\tilde{c}_{-1,3}\big),\\
\tilde{c}_{5,6}=\frac{ab^3\tilde{c}_{-1,3}}{7776}\big((2213280\,\tilde{c}_{-1,3}^4\!+6699540\,\tilde{c}_{-1,3}^3\!
+5410400\,\tilde{c}_{-1,3}^2\!+1337715\,\tilde{c}_{-1,3}\!+77760)a^2\!\\
+664992\,\tilde{c}_{-1,3}^4\!+1210260\,\tilde{c}_{-1,3}^3\!+529760\,\tilde{c}_{-1,3}^2\!
+54675\,\tilde{c}_{-1,3}\big),\\
\tilde{c}_{5,7}=-\frac{ab^3\tilde{c}_{-1,3}^2}{648}\big((368880\,\tilde{c}_{-1,3}^4\!+1339908\,\tilde{c}_{-1,3}^3\!
+1352600\,\tilde{c}_{-1,3}^2\!+445905\,\tilde{c}_{-1,3}\!+38880)a^2\\
+110832\,\tilde{c}_{-1,3}^4\!+242052\,\tilde{c}_{-1,3}^3\!+132440\,\tilde{c}_{-1,3}^2\!+18225\,\tilde{c}_{-1,3}\big),\\
\tilde{c}_{5,8}=\frac{ab^3\tilde{c}_{-1,3}^3}{1296}\big((1475520\,\tilde{c}_{-1,3}^4\!+6252904\,\tilde{c}_{-1,3}^3\!
+7574560\,\tilde{c}_{-1,3}^2\!+3121335\,\tilde{c}_{-1,3}\!+362880)a^2\\
+443328\,\tilde{c}_{-1,3}^4\!+ 1129576\,\tilde{c}_{-1,3}^3\!+741664\,\tilde{c}_{-1,3}^2\!+127575\,
\tilde{c}_{-1,3}\big),\dotsc.
\end{gather*}
Perusing these formulae, one can deduce several interesting properties of the coefficients $\tilde{c}_{2k-1,m}$;
for example,
\begin{equation}\label{app:eq:c-2k-1-mPolynomial}
\tilde{c}_{2k-1,m}=b^kP_{k,m}(a,\tilde{c}_{-1,3}),
\end{equation}
where $P_{k,m}(a,\tilde{c}_{-1,3})$ is a polynomial with rational coefficients; moreover, the following properties
can be conjectured.
\begin{conjecture}\label{con:singular-log-coeff}
For $k,m,n\in\mathbb{N}:$
\begin{enumerate}
\item[{\rm(1)}]
$a=0$ is a first-order zero of the polynomials $P_{2n-1,m}(a,\tilde{c}_{-1,3})$ and
$P_{2n,m}(0,\tilde{c}_{-1,3})\neq0;$
\item[{\rm(2)}]
for $m\geqslant k+3$, $\tilde{c}_{-1,3}=0$ is a zero of order $m-k-2$ for the polynomials
$P_{k,m}(a,\tilde{c}_{-1,3});$
\item[{\rm(3)}]
$\deg P_{k,m}=k+m-1$, $\deg_{a} P_{k,m}=k$, and $\deg_{\tilde{c}_{-1,3}} P_{k,m}=m-1$.
\end{enumerate}
\end{conjecture}
We, in fact, calculated the coefficients $\tilde{c}_{k,m}$ for levels $k=4$ and $5$ up to $m=9$ and $7$,
respectively; these calculations support the conjectures (1)--(3) made above.\footnote{\label{foot:calc-time-AppC}
The time required for the calculation of all the coefficients mentioned in this sentence on the new notebook
computer is approximately 2 minutes and 10 seconds. Our attempt to calculate the coefficients of
level $k=8$ up to $m=10$, however, failed, so we halted the calculation after roughly 1 hour and 30 minutes,
when almost 4Gb of RAM was used. Note that, to
calculate $\tilde{c}_{15,10}$, one has to calculate the coefficients of the previous level $7$ up to $m=12$,
of the level $6$ up to $m=14$, etc.}
The formulae presented above allow the reader to posit further conjectures regarding the properties of the
polynomials $P_{k,m}$.

We would like to draw the reader's attention to conjecture (2) above for the polynomials
$P_{k,m}(a,\tilde{c}_{-1,3})$, because it is the only---proposed---case for which all the levels $k$ in the
expansion~\eqref{app:eq:0-u-expansionLog2} become finite! In Section~\ref{sec:logarithm}
(cf. Corollary~\ref{cor:c+=c-=0}), we found the monodromy data of the solution corresponding to $\tilde{c}_{-1,3}=0$.
At the end of Subsection~\ref{app:subsec:super-generating-functionLog2} below, we explain how the hypothesis
of the finiteness of the levels follows from the construction of the super-generating function; see the discussion
below Conjecture~\ref{app:con:tilde-c-2k-1m}.

At the very least, we expect that some of the properties mentioned above could be proved with the help of the
recurrence relation for the coefficients $\tilde{c}_{2k-1,m}$ via mathematical induction; this recurrence relation,
however, is quite cumbersome, and such considerations would lead us far too astray from our current goals.
In particular, it is important to prove that the coefficients  $\tilde{c}_{k,m}$ are defined for all values of the
parameters $a,\tilde{c}_{-1,3}\in\mathbb{C}$, and $b>0$. For this purpose, it is sufficient to prove the
property~\eqref{app:eq:c-2k-1-mPolynomial}. First, we establish the $b$-dependence
of the coefficients $\tilde{c}_{2k-1,m}$. Recall that $b>0$ is a parameter that can be removed from
equation~\eqref{eq:dp3} via scaling, which, for $\varepsilon=1$, reads
\begin{equation}\label{app:eq:rescaling}
u(\tau)=\sqrt{b}\,u_1(\tau_1),\qquad
\tau_1=\sqrt{b}\,\tau,
\end{equation}
where $u_1(\tau_1)$ is a solution of equation~\eqref{eq:dp3} for $b=1$ and $u(\tau)$ solves equation~\eqref{eq:dp3}
for arbitrary $b>0$.
To make the distinction between the coefficients that correspond to $b=1$ and those that correspond to generic values
of $b$, we denote them, respectively, as $\tilde{c}_{-1,3}^1$ and $\tilde{c}_{-1,3}^b$. Now, we write the
asymptotics~\eqref{app:eq:0-u-expansionLog2} for the function $u_1(\tau_1)$, and then perform a re-scaling of the
function $u(\tau)$ (cf. equation~\eqref{app:eq:rescaling}); then,
expanding as $\tau\to0$,
\begin{equation*}\label{app:eq:lnb-tau-expansion}
(\ln\sqrt{b}+\ln\tau)^{-k}=(\ln\tau)^{-k}-k\ln\sqrt{b}\,(\ln\tau)^{-k-1}
+\frac{k(k+1)}{2}\ln^2\sqrt{b}\,(\ln\tau)^{-k-2}-\ldots, \quad k\in\mathbb{Z},
\end{equation*}
and comparing the coefficients, we get
\begin{equation}\label{app:eq:c-2k-1,m-b-expansion}
\tilde{c}_{2k-1,m}^b=b^k\big(\tilde{c}_{2k-1,m}^1-(m-1)\ln\sqrt{b}\,\tilde{c}_{2k-1,m-1}^1
+\frac{(m-2)(m-1)}{2}\ln^2\sqrt{b}\,\tilde{c}_{2k-1,m-2}^1+\ldots\big).
\end{equation}
At first glance, it appears that there is a contradiction with equation~\eqref{app:eq:c-2k-1-mPolynomial}, because
we get a polynomial depending not only on $a$ and $\tilde{c}_{-1,3}$, but also on ($\ln\sqrt{b}$)! This, surely,
seems wrong, because equation~\eqref{eq:dp3} depends quadratically on $b$, and thus, by substituting into it the
ansatz~\eqref{app:eq:0-u-expansionLog2}, we get that, in the worst case, the coefficients $\tilde{c}_{2k-1,m}$ are
rational functions of $b$. The resolution of this apparent visual contradiction is to express all of our coefficients
in terms of $a$ and $\tilde{c}_{-1,3}^b$ (note that, in equation~\eqref{app:eq:c-2k-1-mPolynomial}, we have, in fact,
$\tilde{c}_{-1,3}^b$, and not $\tilde{c}_{-1,3}^1$), because $\tilde{c}_{-1,3}^1$ is independent of $b$ while
$\tilde{c}_{-1,3}^b$ may be $b$-dependent.
This observation proves the $b$-dependence of equation~\eqref{app:eq:c-2k-1-mPolynomial}, because it shows that
the polynomial of $\ln\sqrt{b}$ inside the parentheses in equation~\eqref{app:eq:c-2k-1,m-b-expansion} is,
in fact, of order zero, and  we can rewrite it in the form $\tilde{c}_{2k-1,m}=b^kP_{k,m}(a,\tilde{c}_{-1,3}^b)$,
where, at this juncture, we cannot conclude that the function $P_{k,m}$ is necessarily a polynomial, but can claim
that it is some rational function of its variables.

For the reader who is perplexed by this proof, we present a couple of examples showing how the mechanism of the
ln-cancellation works. A surprise appears at the stage that, by definition,  both  $\tilde{c}_{-1,3}^1$  and
$\tilde{c}_{-1,3}^b$ are defined in the same way, namely, as $\mathbb{C}$-valued parameters, that is,
they coincide at first glance. Recall that $\tilde{c}_{-1,2}^b=\tilde{c}_{-1,2}^1=-1/4$: they coincide because
$\tilde{c}_{-1,0}=\tilde{c}_{-1,1}=0$, independent of the value of $b$.
To find  $\tilde{c}_{-1,3}^b$,  we refer to equation~\eqref{app:eq:c-2k-1,m-b-expansion} for $k=0$ and $m=3$:
\begin{equation}\label{app:eq:tilde-c-13b}
\tilde{c}_{-1,3}^b=\tilde{c}_{-1,3}^1-2\ln(\sqrt{b})\,\tilde{c}_{-1,2}^1=\tilde{c}_{-1,3}^1+\frac12\ln\sqrt{b}.
\end{equation}
Now, let's check equation~\eqref{app:eq:c-2k-1,m-b-expansion} for the next value, $m=4$:
\begin{equation*}\label{app:eq:c-14b}
\tilde{c}_{-1,4}^b=\tilde{c}_{-1,4}^1-3\ln\sqrt{b}\,\tilde{c}_{-1,3}^1+3\ln^2\sqrt{b}\,\tilde{c}_{-1,2}
=-3(\tilde{c}_{-1,3}^1)^2-3\ln\sqrt{b}\,\tilde{c}_{-1,3}^1-\frac34\ln^2\sqrt{b}=-3(\tilde{c}_{-1,3}^b)^2,
\end{equation*}
where, in the last calculation, we used equation~\eqref{app:eq:tilde-c-1m} for $b=1$ and $m=4$. Consider one more
example for $k=1$: setting $m=0$, $m=1$, and $m=2$ successively in equation~\eqref{app:eq:c-2k-1,m-b-expansion},
we get
\begin{equation*}\label{app:eq:tilde-c10-c11-c12}
\tilde{c}_{1,0}^b=b\,\tilde{c}_{1,0}^1,\quad
\tilde{c}_{1,1}^b=b\,\tilde{c}_{1,1}^1,\;\;\mathrm{and}\;\;
\tilde{c}_{1,2}^b=b(\tilde{c}_{1,2}^1-\ln\sqrt{b}\,c_{1,1}^1);
\end{equation*}
since the first two equations demonstrate the correct dependence of the coefficients on $b$, let us consider the third
one, namely,
\begin{equation*}\label{app:eq:tilde-c12}
\tilde{c}_{1,2}^b=b(a(2\tilde{c}_{-1,3}^1+1)-\ln\sqrt{b}\,\tilde{c}_{1,1}^1)
=b\big(2a(\tilde{c}_{-1,3}^b-\frac12\ln\sqrt{b})+a-\ln\sqrt{b}\,\tilde{c}_{1,1}^1\big)
=ab(2\tilde{c}_{-1,3}^b+1),
\end{equation*}
where, to obtain the first equation, we used the third equation in the first line of the list~\eqref{app:eq:tilde-c1m}
for $b=1$, to get the second equation, we employed the relation~\eqref{app:eq:tilde-c-13b}, and, finally, the second
equation in the first line of the list~\eqref{app:eq:tilde-c1m}, that is, $\tilde{c}_{1,1}^1=-a$, provided the
cancellation of the $\ln$-terms.

We revert back to our original notation and prove that the function $P_{k,m}(a,\tilde{c}_{-1,3})$ is a polynomial
of two variables. In principle, this fact can be established by appealing to the recurrence relation for the
coefficients $\tilde{c}_{2k-1,m}$; but, because this relation is complicated, we are not going to use its explicit
form in this work, and will, therefore, exploit only those properties of this relation that are pertinent to
the current proof.

To derive the recurrence relation, multiply both sides of equation~\eqref{eq:dp3} by $\tau^3(u(\tau))^2$ and
substitute for $u(\tau)$ its asymptotic expansion~\eqref{app:eq:0-u-expansionLog2}; then, for $k=0, 1, 2, \ldots$,
collect, successively, the ``coefficients'' of like powers of $\tau^{2k}$. These ``coefficients'' are, in fact,
series of powers of $\ln\tau$, the coefficients of which are $\tau$-independent polynomials of the
$\tilde{c}_{2k-1,m}$'s.

Consider what happens for $k=0$. The first non-trivial coefficient corresponds to the term $(\ln\tau)^{-6}$, and
equals $8\tilde{c}_{-1,2}^3+2\tilde{c}_{-1,2}^2$, which consists of two contributions: the first one stems from the
term $8\tau^3(u(\tau))^3$ on the right-hand side of this equation, and the second one originates from its
differential part. Setting $8\tilde{c}_{-1,2}^3+2\tilde{c}_{-1,2}^2=0$, it follows that the only solution allowing
one to develop a non-trivial expansion~\eqref{app:eq:0-u-expansionLog2} is $\tilde{c}_{-1,2}=-1/4$. Then, proceeding
to the subsequent power of $\ln\tau$, we get $(24\tilde{c}_{-1,2}^2+6\tilde{c}_{-1,2})\tilde{c}_{-1,3}=0$, which
implies that $\tilde{c}_{-1,3}$ is a complex parameter. The following terms, for $m>3$, read:
\begin{equation}\label{app:eq:f-1m}
(24\tilde{c}_{-1,2}^2+(m(m-3)+6)\tilde{c}_{-1,2})\tilde{c}_{-1,m}\equiv
m(m-3)\tilde{c}_{-1,2}\tilde{c}_{-1,m}=f_{-1,m},
\end{equation}
where
$f_{-1,m}$ is a polynomial with integer coefficients of the variables $\tilde{c}_{-1,m'}$ for $m'<m$ and the
parameters $a$ and $b$. The last statement is apparent because the transformed equation~\eqref{eq:dp3} is a
polynomial with integer coefficients in terms of $u(\tau)$, its derivatives, and the parameters $a$ and $b$, and,
at the same time, the ansatz~\eqref{app:eq:0-u-expansionLog2} does not have any coefficients $\tilde{c}_{-1,m}$
in the denominator.
Solving equation~\eqref{app:eq:f-1m} successively for $m=4,5,\ldots$, one arrives at the
formulae~\eqref{app:eq:tilde-c-1m}. Then, we continue this procedure for the higher levels $k=1,2,\ldots$, and,
in this way, obtain equations of the form
\begin{equation}\label{app:eq:ckm}
(2k)^2\tilde{c}_{-1,2}\tilde{c}_{2k-1,m}=f_{k,m},
\end{equation}
where $f_{k,m}$ is a polynomial of the ``lower-order'' coefficients $\tilde{c}_{2k'-1,m'}$, $k'<k$ and
$m'\leqslant m+2(k-k')$, and the parameters $a$ and $b$. Note that,
in equation~\eqref{app:eq:ckm}, since the coefficient $(2k)^2c_{-1,2}=-k^2\neq0$ for all $m$,
it is clear that all the $\tilde{c}_{2k-1,m}$'s are polynomials with rational coefficients of the
variables $a$, $b$, and $\tilde{c}_{-1,3}$.
%%%%%%%%%%%%%%%%%%%%%%%%%%%%%%%%%%%%%%%%%%%%%%%%%%%%%%%%%%%%%%%%%%%%%%%%%%%%%%%%%%%%%%%%%%%%%%%%%%%%%%%%%%%%%%%%%%%%%%%%%%%%%%%
%%%%%%%%%%%%%%%%%%%%%%%%%%%%%%%%%%%%%%%%%%%%%%%%%%%%%%%%%%%%%%%%%%%%%%%%%%%%%%%%%%%%%%%%%%%%%%%%%%%%%%%%%%%%%%%%%%%%%%%%%%%%%%%
\subsection{Super-Generating Function for the Irregular Logarithmic Asymptotics}
\label{app:subsec:super-generating-functionLog2}
In contrast to the expansions~\eqref{app:eq:0-u-expansion} and \eqref{app:eq:0-u-expansionLog} studied in the
previous appendices, the asymptotic expansion~\eqref{app:eq:0-u-expansionLog2} has infinite levels, as a result of
which, we present in this subsection the super-generating function that computes the coefficients of the levels.
The construction of the super-generating function for the coefficients of the
expansion~\eqref{app:eq:0-u-expansionLog2}, though similar to those presented in
Subsections~\ref{app:subsec:super-generating-function} and \ref{app:subsec:super-generating-functionLog}, is simpler:
\begin{equation}\label{app:eq:supergenfunctonLog2}
\tilde{A}(x,y)=\sum_{k=0}^{\infty}y^k\tilde{A}_{k}(x),
\end{equation}
where the coefficient functions $\tilde{A}_k(x)$, $k\in\mathbb{Z}_{\geqslant0}$, are generating functions for
the coefficients $\tilde{c}_{2k-1,m}$:
\begin{equation}\label{app:eq:tildeAk-series}
\tilde{A}_k(x)=\sum_{m=-2\lfloor k/2\rfloor}^{\infty}\tilde{c}_{2k-1,m}x^k,\qquad
k\geqslant0.
\end{equation}
Define the linear differential operator $\tilde{D}$ acting in the space of formal power series of two variables
$x$ and $y$ as follows:
\begin{equation}\label{app:eq:tildeD-definition}
\tilde{D}=-x^2\frac{\partial}{\partial x}+2y\frac{\partial}{\partial y};
\end{equation}
then, the function $\tilde{A}\equiv\tilde{A}(x,y)$ solves the PDE
\begin{equation}\label{app:eq:tildeA-PDE}
\tilde{D}^2(\ln\tilde{A})=-8\tilde{A}+2a\frac{by}{\tilde{A}}+\left(\frac{by}{\tilde{A}}\right)^2.
\end{equation}
Assuming that a proper solution of the PDE~\eqref{app:eq:tildeA-PDE} is constructed, one can obtain the solution of
the ODE~\eqref{eq:dp3} with the help of the relation
\begin{equation}\label{app:eq:u-tildeA-relation}
u(\tau)=\frac1{\tau}\tilde{A}\big(1/\ln\tau,\tau^2\big).
\end{equation}
What, then, is the proper solution? The function $\tilde{A}(x,y)$ (cf. expansion~\eqref{app:eq:supergenfunctonLog2})
is a formal solution of equation~\eqref{app:eq:tildeA-PDE} with generating functions $\tilde{A}_k(x)$ that are
rational functions of $x$, and the function $\tilde{A}_0(x)$ is normalized by the small-$x$
expansion~\eqref{app:eq:tildeAk-series} with $k=0$, $\tilde{c}_{-1,0}=\tilde{c}_{-1,1}=0$, and $\tilde{c}_{-1,2}=-1/4$.

Define the $x$-part of the operator $\tilde{D}$ as
\begin{equation}\label{app:eq:tildeDx}
\tilde{D}_x:=-x^2\frac{\md}{\md x},
\end{equation}
and construct the first few generating functions $\tilde{A}_k(x)$, $k=0,1,2,3$.

Substituting the expansion~\eqref{app:eq:supergenfunctonLog2} into equation~\eqref{app:eq:tildeA-PDE} and equating to
zero the coefficients of the monomials $y^k$ for $k=0,1,2,\ldots$, one obtains ODEs defining the generating functions
$\tilde{A}_k(x)$. The function $\tilde{A}_0(x)$ satisfies the ODE
\begin{equation}\label{app:eq:tildeA0}
\tilde{D}_x^2\ln\tilde{A}_0(x)+8\tilde{A}_0(x)=0.
\end{equation}
The general solution of this ODE reads:
\begin{equation}\label{app:eq:generictildeA0}
\tilde{A}_0(x)=-\frac{C_1^2}{16\cos^2(C_1(C_2x-1)/(2x))},
\end{equation}
where $C_1$ and $C_2$ are constants of integration. To achieve our goal, we need a solution $\tilde{A}_0(x)$ that is
a rational function of $x$. This is a special solution of the ODE~\eqref{app:eq:tildeA0} that can be obtained
from the general one by making the scaling limit $C_1\to0$, $C_2=C+\pi/C_1$ in equation~\eqref{app:eq:generictildeA0},
where $C$ is a complex parameter; then, we find that
\begin{equation}\label{app:eq:tildeA0rational}
\tilde{A}_0(x)=-\frac{1}{4(1/x - C)^2}.
\end{equation}
Comparing the expansion~\eqref{app:eq:supergenfunctonLog2} with the definition~\eqref{app:eq:u-tildeA-relation} for
$u(\tau)$ in terms of $\tilde{A}(x,y)$, we see that $\tilde{A}_0(x)/\tau$
(cf. equation~\eqref{app:eq:tildeA0rational}) coincides with the leading term of asymptotics of the function $u(\tau)$
obtained in Theorems~\ref{th:Asympt0-ln-rho12-} and \ref{th:Asympt0-ln-rho12+}${}^{\mathbf\prime}$; moreover, the
following conditions hold:
\begin{equation}\label{app:eq:Cc+c-c-1,3}
-\frac{C}2=c_+=c_-=c_{-1,3},
\end{equation}
where $c_+$ and $c_-$ are defined in terms of the monodromy data in
Theorems~\ref{th:Asympt0-ln-rho12-} and \ref{th:Asympt0-ln-rho12+}${}^{\mathbf\prime}$. Expanding, now, the
function $\tilde{A}_0(x)$ (cf. equation~\eqref{app:eq:tildeA0rational}) into a Taylor series about $x=0$, one proves
the general formula for the coefficients $\tilde{c}_{-1,m}$ given in equation~\eqref{app:eq:tilde-c-1m}.

The equation for the generating function $\tilde{A}_1(x)$ reads:
\begin{equation}\label{app:eq:tildeA1}
\big((\tilde{D}_x+2)^2+8\tilde{A}_0(x)\big)\tilde{A}_{10}=\frac{2ab}{\tilde{A}_0(x)},\qquad
\tilde{A}_{10}:=\frac{\tilde{A}_1(x)}{\tilde{A}_0(x)}.
\end{equation}
The general solution of equation~\eqref{app:eq:tildeA1} is
\begin{equation}\label{app:eq:tildeA1gensol}
\begin{aligned}
\tilde{A}_{10}\,\tilde{A}_0(x)=&\frac{\big(C_1x^3+C_2(3C^2x^2-3Cx+1)\big)\me^{-\frac{2}{x}}}{(Cx - 1)^3}\\
 &+\frac{ab}{2}\,\frac{((C^2+C+1)x^2-(2C+1)x+1)((C+1)x-1)}{(Cx - 1)^3},
\end{aligned}
\end{equation}
where $C_1$ and $C_2$ are constants of integration, and $C$ is given in equation~\eqref{app:eq:Cc+c-c-1,3}.
Since we need a rational solution, we set $C_1=C_2=0$, and obtain, finally,
\begin{equation}\label{app:eq:tildeA1rational}
\tilde{A}_1(x)=\frac{ab}{2}\,\frac{((C^2+C+1)x^2-(2C+1)x+1)((C+1)x-1)}{(Cx - 1)^3}.
\end{equation}
Expanding the function $\tilde{A}_1(x)$ given in equation~\eqref{app:eq:tildeA1rational} into a Taylor series about
$x=0$, one shows that
\begin{equation}\label{app:eq:tildeA1-series}
\begin{aligned}
\tilde{A}_1(x)=&\frac{ab}{2C^3}\left(C^3+2C^2+2C+1-\sum_{m=0}^{\infty}\left(\frac{(m+2)(m+1)}{2}-(2C+3)(m+1)
\right.\right.\\
&+2C^2+4C+3\Bigg)C^mx^m\Bigg).
\end{aligned}
\end{equation}
Comparing expansion~\eqref{app:eq:tildeAk-series} for $k=1$ with the expansion~\eqref{app:eq:tildeA1-series},
we get that $\tilde{c}_{1,0}=ab/2$, and
\begin{equation}\label{app:eq:tilde-c1m-formula}
\begin{aligned}
\tilde{c}_{1,m}&=-abC^{m-3}\left(C^2-(m-1)C+\frac{(m-1)(m-2)}{4}\right)\\
&=(-1)^m2^{m-5}ab\,\tilde{c}_{-1,3}^{m-3}\big(16\tilde{c}_{-1,3}^2+8(m-1)\tilde{c}_{-1,3}+(m-1)(m-2)\big),
\quad
m\in\mathbb{N}.
\end{aligned}
\end{equation}
The formula~\eqref{app:eq:tilde-c1m-formula} should be compared with the coefficients
(cf. equations~\eqref{app:eq:tilde-c1m}) $\tilde{c}_{1,m}$, $m=1,2\ldots10$, that were calculated by directly
substituting the expansion~\eqref{app:eq:0-u-expansionLog2} into equation~\eqref{eq:dp3}.

The equation for the determination of the generating function for level $2$ can be written as follows:
\begin{equation}\label{app:eq:tildeA2ode}
\big((\tilde{D}_x+4)^2+8\tilde{A}_0(x)\big)\tilde{A}_{20}=\frac12(\tilde{D}_x+4)^2\tilde{A}_{10}^2-
2ab\frac{\tilde{A}_{10}}{\tilde{A}_0(x)}+\frac{b^2}{\tilde{A}_0(x)^2},\qquad
\tilde{A}_{20}:=\frac{\tilde{A}_2(x)}{\tilde{A}_{0}(x)},
\end{equation}
where the function $\tilde{A}_{10}$ is defined by equation~\eqref{app:eq:tildeA1}. The general solution of
equation~\eqref{app:eq:tildeA2ode} can be presented as
\begin{equation}\label{app:eq:tildeA2gensol}
\tilde{A}_{2}^{gen}(x)=\frac{\big(C_1x^3+C_2(3C^2x^2-3Cx+1)\big)\me^{-\frac{4}{x}}}{(Cx - 1)^3}+\tilde{A}_{2}(x),
\end{equation}
where, with slight abuse of notation, we denote by $\tilde{A}_{2}(x)$ a special rational solution of
equation~\eqref{app:eq:tildeA2ode} corresponding to vanishing values of the constants of integration, $C_1=C_2=0$.
This special solution coincides with the generating function for level 2. It is convenient to present
it via a partial-fraction decomposition:
\begin{equation}\label{app:eq:tildeA2partfrac}
\begin{gathered}
\tilde{A}_{2}(x)=-\frac{b^2(a^2+1)}{4x^2}+\frac{b^2\big((a^2+1)C+2a^2+1\big)}{2x}
-\frac{b^2}{256C^4}\Bigg(\big(64(a^2+1)C^6+128(2a^2+1)C^5\\
+8(71a^2+19)C^4+24(37a^2+5)C^3+4(239a^2+15)C^2+(623a^2+15)C+192a^2\big)\\
-\frac{192a^2}{(Cx-1)^4}-\frac{(623a^2\!+15)C+768a^2}{(Cx-1)^3}
-\frac{4(239a^2\!+15)C^2+3(623a^2\!+15)C+1152a^2}{(Cx-1)^2}\\
-\frac{24(37a^2+5)C^3+8(239a^2+15)C^2+3(623a^2+15)C+768a^2}{Cx-1}\Bigg).
\end{gathered}
\end{equation}
Equation~\eqref{app:eq:tildeA2partfrac} allows one to derive a general formula for the coefficients $\tilde{c}_{3,m}$
(cf. equations~\eqref{app:eq:tilde-c3m}). The first term in equation~\eqref{app:eq:tildeA2partfrac} immediately
provides us with the formula for $\tilde{c}_{3,-2}$, while the second term, after substituting $C=-2\tilde{c}_{-1,3}$,
coincides with $\tilde{c}_{3,-1}$, and, finally, setting $x=0$ in the denominators of the terms in the
third and fourth lines, we find that
\begin{equation}\label{app:eq:tilde-c30}
\tilde{c}_{3,0}=b^2\left(\frac{C^2}{4}(a^2+1)+C(a^2+1/2)+\frac{71}{32}a^2+\frac{19}{32}\right).
\end{equation}
Substituting  $C=-2\tilde{c}_{-1,3}$ into equation~\eqref{app:eq:tilde-c30}, we arrive at the formula for
$\tilde{c}_{3,0}$ that is equivalent to the one written in the list of equations~\eqref{app:eq:tilde-c3m}.
Expanding $\tilde{A}_{2}(x)$ into a Taylor series about $x=0$, we find that
\begin{equation}\label{app:eq:tilde-c3mformula}
\begin{aligned}
%\tilde{c}_{3,m}=&-\frac{b^2C^{m-4}}{512}\Big(64a^2m^3\!-\big((623a^2\!+15)C+384a^2\big)m^2\!+
%\big(8(239a^2\!+15)C^2\!+3(623a^2\!+15)C\\
%&+704a^2)m-48(37a^2\!+5)C^3\!-8(239a^2\!+15)C^2\!-2(623a^2\!+15)C-384a^2\Big)\\
\tilde{c}_{3,m}=&\frac{b^2C^{m-4}}{512}\Big(48(37a^2+5)C^3-8(239a^2+15)(m - 1)C^2\\
&+(623a^2+15)(m - 1)(m - 2)C-64a^2(m - 1)(m - 2)(m - 3)\Big)\\
=&(-1)^{m-1}2^{m-12}b^2\tilde{c}_{-1,3}^{m-4}\Big(192(37a^2+5)\tilde{c}_{-1,3}^3
+16(239a^2+15)(m - 1)\tilde{c}_{-1,3}^2\\
&+(623a^2+15)(m - 1)(m - 2)\tilde{c}_{-1,3}+32a^2(m - 1)(m - 2)(m - 3)\Big),
\qquad
m\in\mathbb{N}.
\end{aligned}
\end{equation}
The generating function $\tilde{A}_3(x)$ is the rational solution of the ODE
\begin{equation}\label{app:eq:tilde-A3ode}
\Big(\big(\tilde{D}_x+6\big)^2+8\tilde{A}_0\Big)\tilde{A}_{30}
=\big(\tilde{D}_x+6\big)^2\big(\tilde{A}_{10}\tilde{A}_{20}-\frac13\tilde{A}_{10}^3\big)
+\frac{2ab}{\tilde{A}_0}\big(\tilde{A}_{10}^2-\tilde{A}_{20}\big)
-\frac{2b^2}{\tilde{A}_0^2}\tilde{A}_{10},
\end{equation}
where, for $k=1,2,3$, $\tilde{A}_{k0}:=\frac{\tilde{A}_{k}(x)}{\tilde{A}_{0}(x)}$, with the rational functions
$\tilde{A}_{k-1}(x)$ obtained in the previous steps (cf. equations~\eqref{app:eq:tildeA0rational},
\eqref{app:eq:tildeA1rational}, and \eqref{app:eq:tildeA2partfrac}). The form of the general solution of
equation~\eqref{app:eq:tilde-A3ode} is similar to the one in equation~\eqref{app:eq:tildeA2gensol}, namely,
\begin{equation}\label{app:eq:tildeA3gensol}
\tilde{A}_{3}^{gen}(x)=\frac{\big(C_1x^3+C_2(3C^2x^2-3Cx+1)\big)\me^{-\frac{6}{x}}}{(Cx - 1)^3}+\tilde{A}_{3}(x),
\end{equation}
but, in this case, the unique particular rational solution $\tilde{A}_{3}(x)$ defining the generating function for
level 3 is more complicated:
\begin{equation}\label{app:eq:tildeA3partfrac}
\begin{gathered}
\tilde{A}_{3}(x)=\frac{b^3a(a^2+1)}{4x^2}-\frac{b^3a\big(4(a^2+1)C+13a^2+9\big)}{8x}
+\frac{b^3a}{4C^5}\Bigg(\Big((a^2+1)C^7+(13a^2+9)\frac{C^6}{2}\\
+(176a^2\!+83)\frac{C^5}{9}+(7685a^2\!+2309)\frac{C^4}{216}+(111659a^2\!+20171)\frac{C^3}{2592}
+(33815a^2\!+3311)\frac{C^2}{972}\\
+3(367a^2+15)\frac{C}{64}+4a^2\Big)+\frac{4a^2}{(Cx-1)^5}+\frac{\tilde{\kappa}_4}{(Cx-1)^4}
+\frac{\tilde{\kappa}_3}{(Cx-1)^3}
+\frac{\tilde{\kappa}_2}{(Cx-1)^2}
+\frac{\tilde{\kappa}_1}{Cx-1}\Bigg),
\end{gathered}
\end{equation}
where
\begin{align*}\label{app:eqs:kappa1kappa2kappa3}
\tilde{\kappa}_1=&\,(7685a^2+2309)\frac{C^4}{216}+(111659a^2+20171)\frac{C^3}{1296}+(33815a^2+3311)\frac{C^2}{324}\\
&+3(367a^2+15)\frac{C}{16}+20a^2,\\
\tilde{\kappa}_2=&\,(111659a^2+20171)\frac{C^3}{2592}+(33815a^2+3311)\frac{C^2}{324}+9(367a^2+15)\frac{C}{32}+40a^2,\\
\tilde{\kappa}_3=&\,(33815a^2+3311)\frac{C^2}{972}+3((367a^2+15)\frac{C}{16}+40a^2,\\
\tilde{\kappa}_4=&\,3(367a^2+15)\frac{C}{64}+20a^2.
\end{align*}
The first two terms in the first line of equation~\eqref{app:eq:tildeA3partfrac} give rise to the coefficients
$\tilde{c}_{5,-2}$ and $\tilde{c}_{5,-1}$ presented in the list of equations~\eqref{app:eq:tilde-c5:-2:-1:0}.
To obtain the next coefficient $\tilde{c}_{5,0}$ in this list, one has to set $x=0$ in the ``large parentheses''
in equation~\eqref{app:eq:tildeA3partfrac}, yielding
\begin{equation*}\label{app:eq:tilde-c50}
\tilde{c}_{5,0}=\frac{b^3a}{72}\big((18C^2+117C+352)a^2+18C^2+81C+166\big),
\end{equation*}
which, after the substitution $C=-2\tilde{c}_{-1,3}$, coincides with the corresponding formula in the
list of equations~\eqref{app:eq:tilde-c5:-2:-1:0}. Expanding $\tilde{A}_{3}(x)$ given in
equation~\eqref{app:eq:tildeA3partfrac} into a Taylor series about $x=0$, we find that
\begin{equation*}\label{app:eq:tilde-c5m-formula}
\tilde{c}_{5,m}=\frac{ab^3C^{m-5}}{4}\left(-4a^2\frac{(m+4)!}{m!\,4!}
+\tilde{\kappa}_{4}\frac{(m+3)!}{m!\,3!}-\tilde{\kappa}_{3}\frac{(m+2)!}{m!\,2!}
+\tilde{\kappa}_{2}(m+1)-\tilde{\kappa}_{1}\right),\quad
m\in\mathbb{N},
\end{equation*}
or, more explicitly,
%\begin{equation*}\label{app:eq:tilde-c5m-formula-explicit}
%\begin{aligned}
%\tilde{c}_{5,m}=&(-1)^m2^{m-7}ab^3\tilde{c}_{-1,3}^{m-5}\Bigg(\frac{a^2m^4}{6}
%+\Big((367a^2+15)\frac{\tilde{c}_{-1,3}}{64}-\frac{5a^2}{3}\Big)m^3
%+\Big((33815a^2+3311)\frac{\tilde{c}_{-1,3}^2}{486}\\
%&-3(367a^2+15)\frac{\tilde{c}_{-1,3}}{32}+\frac{35a^2}{6}\Big)m^2
%+\Big((111659a^2+20171)\frac{\tilde{c}_{-1,3}^3}{324}-(33815a^2+3311)\frac{\tilde{c}_{-1,3}^2}{162}\\
%&+(4037a^2+165)\frac{\tilde{c}_{-1,3}}{64}-\frac{25a^2}{3}\Big)m
%+(15370a^2+4618)\frac{\tilde{c}_{-1,3}^4}{27}-(111659a^2+20171)\frac{\tilde{c}_{-1,3}^3}{324}\\
%&+(33815a^2+3311)\frac{\tilde{c}_{-1,3}^2}{243}-3(367a^2+15)\frac{\tilde{c}_{-1,3}}{32}+4a^2\Bigg),
%\qquad
%m\in\mathbb{N}.
%\end{aligned}
%\end{equation*}
\begin{equation}\label{app:eq:tilde-c5m-formula-explicit}
\begin{aligned}
\tilde{c}_{5,m}=&(-1)^m2^{m-7}ab^3\tilde{c}_{-1,3}^{m-5}\Bigg(\frac{2(7685a^2+2309)}{27}\tilde{c}_{-1,3}^4
+\frac{111659a^2+20171}{324}(m-1)\tilde{c}_{-1,3}^3\\
&+\frac{33815a^2+3311}{486}(m-1)(m-2)\tilde{c}_{-1,3}^2
+\frac{367a^2+15}{64}(m-1)(m-2)(m-3)\tilde{c}_{-1,3}\\
&+\frac{a^2}{6}(m-1)(m-2)(m-3)(m-4)\Bigg),
\qquad
m\in\mathbb{N}.
\end{aligned}
\end{equation}
These coefficients, for $m=1,2,\ldots,8$, coincide with the ones computed directly by substituting the
expansion~\eqref{app:eq:0-u-expansionLog2} into equation~\eqref{eq:dp3} (cf. the list of equations following
\eqref{app:eq:tilde-c5:-2:-1:0}).

We calculated, furthermore, the generating functions $\tilde{A}_4(x)$, $\tilde{A}_5(x)$, and $\tilde{A}_6(x)$.
We did not observe an increase in the computation time, although, of course, there was some: the answers appear
virtually the moment one's finger is lifted {}from the ``enter'' button! The complexity of the answers, however,
increases; for example, on our \textsc{Maple} output sheet, the functions $\tilde{A}_3(x)$, $\tilde{A}_4(x)$,
$\tilde{A}_5(x)$, and $\tilde{A}_6(x)$ require $5$, $15$, $21$, and $46$ lines, respectively, to display. It seems
that the principal limitation with such computations is the number of digits required for printing the answers.
Another limitation for the continuation of these computations is that the functions very quickly become unobservable.
Although, as follows, say, {}from the construction of the function $\tilde{A}_3(x)$, there is some obvious pole
structure of these functions, the main problem is the calculation of the corresponding residues. Based on this pole
structure, which can be deduced {}from the recurrence relation for the generating functions (see
equations~\eqref{app:eq:recurrence-tildeA-k} and \eqref{app:eq:tilde-f-k} below), we arrive at the following
conjecture:
\begin{conjecture}\label{app:con:tilde-c-2k-1m}
\begin{equation}\label{app:eq:tilde-c-2k-1-m-formula}
\tilde{c}_{2k-1,m}=(-1)^{m-k-1}C^{m-k-2}\sum_{l=0}^{k+1}P_{k,m}^l(a^2)(m-1)_lC^{k+1-l},
\qquad
m\in\mathbb{N},
\end{equation}
where $P_{k,m}^l(t)$ are polynomials in $t$ of degree $\deg P_{k,m}^l(t)=\lfloor\tfrac{k}{2}\rfloor$,
and $(m-1)_l=(m-1)\cdot\ldots\cdot(m-l)$ is the falling factorial of
length $l$.\footnote{\label{app:foot:fallingfactorial}Note that, by definition, $(m-1)_0=1$.}
The coefficients of the polynomials $P_{k,m}^l(t)$ are positive rational numbers.
\end{conjecture}
Conjecture~\ref{app:con:tilde-c-2k-1m} shows, in particular, that there is only the case corresponding to $C=0$
when all the levels have finite length, namely, for $m>k+2$ the coefficients $\tilde{c}_{2k-1,m}=0$, that is,
the non-vanishing coefficients correspond to $m=-2\lfloor k/2\rfloor,\ldots, k+2$, so that the length
of the level of order $k$ (the number of non-vanishing coefficients) is $2\lfloor k/2\rfloor+k+3$.

The fact that for $C=0$, all levels have finite length does not require as elaborate
a conjecture as \ref{app:con:tilde-c-2k-1m}. Since we know that the coefficients
$\tilde{c}_{2k-1,m}$ are defined via the Laurent-series expansion of $\tilde{A}_k(x)$ about $x=0$
(cf. equation~\eqref{app:eq:tildeAk-series}) and $\tilde{A}_k(x)$ satisfies the linear
ODE~\eqref{app:eq:recurrence-tildeA-k}, \eqref{app:eq:tilde-f-k} with singular points only at
$x=0$ and $x=1/C$, the rational solution $\tilde{A}_k(x)$ has poles at these, and only these, points;
in fact, the orders of the poles at $x=0$ and $x=1/C$ are $-2\lfloor k/2\rfloor$ and $k+2$, respectively,
but we will not use these facts explicitly.
The coefficients $\tilde{c}_{2k-1,m}$, for positive $m$, are defined via the Taylor-series expansions of
$\kappa_l/(Cx-1)^l$ which constitute the partial-fraction decomposition of $\tilde{A}_k(x)$.
It is clear that the numerator of the function $\tilde{A}_k(x)$ is a polynomial in the variables
$x$, $a$, and $C$, and its denominator is just the product of $x^{r_1}$ and $(Cx-1)^{r_2}$, where $r_1$ and
$r_2$ are some positive integers; however, the decomposition of $\tilde{A}_k(x)$ into partial fractions may lead
to the appearance of a non-trivial denominator of the residues $\kappa_l$: in our examples, we see that this
denominator actually appears, and equals $C^{k+2}$.
We expand our partial fractions into Taylor series centred at $x=0$ and take the coefficient of the term $x^m$;
then, this coefficient will be a linear combination of the residues multiplied
by $C^m$. Therefore, whichever power of $C$ appeared in the denominator of the residues will be cancelled for large
enough $m$, because the power of $C$ in the denominator is less than or equal to $k+2$, while $m\to+\infty$.
Now, we set $C=0$, and conclude that, for fixed $k>0$, all the coefficients $\tilde{c}_{2k-1,m}$ vanish for large
enough $m$. What happens for small values of $m$? At the end of
Subsection~\ref{app:subsec:coefficients-expansionLog2}, we proved that the coefficients $\tilde{c}_{2k-1,m}$
are polynomials of $a$ and $C$, which, in turn, implies that the coefficients $\tilde{c}_{2k-1,m}$ are well defined
for all values of $m$, and that a possible negative power of $C$ which appeared in the construction should be
cancelled by a proper positive power of $C$ that is ``reserved" for this purpose in the linear combination of
the residues.
The formula~\eqref{app:eq:tilde-c-2k-1-m-formula} sheds light as to how this occurs.

The recurrence relation for the determination of the generating functions $\tilde{A}_{k}(x)$, $k\in\mathbb{N}$,
reads:
\begin{equation}\label{app:eq:recurrence-tildeA-k}
\left(\left(\tilde{D}_x+2k\right)^2+8\tilde{A}_0(x)\right)\tilde{A}_{k0}=
\tilde{f}_k(\tilde{A}_0(x);\tilde{A}_{10},\ldots,\tilde{A}_{(k-1)0}),
\qquad
\tilde{A}_{p0}=\frac{\tilde{A}_p(x)}{\tilde{A}_0(x)},
\quad
p=1,\ldots,k,
\end{equation}
where
\begin{multline}\label{app:eq:tilde-f-k}
\tilde{f}_k(\tilde{A}_0(x);\tilde{A}_{10},\ldots,\tilde{A}_{(k-1)0})=
\left(\tilde{D}_x+2k\right)^2\left(\sum_{l=2}^k\frac{(-1)^l}{l}
\sum_{\substack{p_1+\ldots+p_l=k\\p_i\geqslant1,i=1,\ldots,l}}
\tilde{A}_{p_10}\cdot\ldots\cdot\tilde{A}_{p_l0}\right)\\
{}+\frac{2ab}{\tilde{A}_0(x)}\sum_{l=1}^{k-1}(-1)^l\sum_{\substack{p_1+\ldots+p_l=k-1\\p_i\geqslant1,i=1,\ldots,l}}
\tilde{A}_{p_10}\cdot\ldots\cdot\tilde{A}_{p_l0}\\
{}+\frac{b^2}{(\tilde{A}_0(x))^2}\sum_{l=1}^{k-2}(-1)^l(l+1)
\sum_{\substack{p_1+\ldots+p_l=k-2\\p_i\geqslant1,i=1,\ldots,l}}\tilde{A}_{p_10}\cdot\ldots\cdot\tilde{A}_{p_l0}.
\end{multline}

\vspace*{0.35cm}
\vspace*{0.35cm}
\noindent
\textbf{\Large Acknowledgements}
\vspace*{0.25cm}

\noindent
A.~V. is grateful to the St.~Petersburg Department of the Steklov Mathematical Institute
for hospitality during the summer of 2023, when this work began.
%\clearpage

%%%%%%%%%%%%%%%%%%%%%%%%%%%%%%%%%%%%%%%%%%%%%%%%%%%%%%%%%%%%%%%%%%%%%%%%%%%%%%%%%%%%%%%%%%%%%%%%%%%%%%%%%%%
%%%%%%%%%%%%%%%%%%%%%%%%%%%%%%%%%%%%%%%%%%%%%%%%%%%%%%%%%%%%%%%%%%%%%%%%%%%%%%%%%%%%%%%%%%%%%%%%%%%%%%%%%%%
\end{document}